\documentclass[11pt]{amsart}
\pdfoutput=1
\usepackage{amssymb}
\usepackage{amsmath}
\usepackage{graphicx,wrapfig}
\usepackage{amsthm}
\usepackage{latexsym}
\usepackage{subfigure}
\usepackage{rotating}
\usepackage{hyperref}
\usepackage[top=3cm, bottom=3cm, left=2.75cm, right=2.75cm]{geometry}
\usepackage{amsaddr}

\newtheorem{theorem}{Theorem}
\newtheorem{lemma}{Lemma}
\newtheorem{remark}{Remark}

\newcommand{\aql}{a^{\rm char}}
\newcommand{\lql}{l^{\rm char}}

\newcommand{\nperp}{n^{\perp}}

\newcommand{\GammaWall}{\Gamma_{\rm wall}}
\newcommand{\GammaIn}{\Gamma_{\rm in}}
\newcommand{\GammaOut}{\Gamma_{\rm out}}
\newcommand{\GammaSym}{\Gamma_{\rm sym}}
\newcommand{\GammaCharacteristic}{\Gamma_{\rm char}}

\newcommand{\vD}{v^{\rm D}}
\newcommand{\pD}{p^{\rm D}}
\newcommand{\uD}{u^{\rm D}}
\newcommand{\lambdap}{\lambda_{\rm p}}
\newcommand{\lambdam}{\lambda_{\rm m}}
\newcommand{\lambdapm}{\lambda_{\rm p/m}}

\newcommand{\rightpm}{r_{\rm p/m}}

\renewcommand{\div}{\operatorname{div}}

\newcommand{\norm}[1]{\|#1\|}

\newcommand{\eps}{\varepsilon}

\newcommand{\R}{\mathbb R}
\newcommand{\Pe}{\mathrm {Pe}}

\newcommand{\Cells}{\mathcal K}

\newcommand{\intOmega}{\int_{\Omega}}
\newcommand{\aStokes}{a^{\rm St}}
\newcommand{\lStokes}{l^{\rm St}}
\newcommand{\aVisc}{a^{\rm visc}}
\newcommand{\lVisc}{l^{\rm visc}}
\newcommand{\aEuler}{a^{\rm Eu}}
\newcommand{\astab}{a_{\rm stab}}
\newcommand{\lstab}{l_{\rm stab}}
\newcommand{\lEuler}{l^{\rm Eu}}

\newcommand{\aNavierStokes}{a^{\rm NaSt}}
\newcommand{\aNavierStokesSUPG}{a^{\rm NaSt}_{\rm stab}}
\newcommand{\lNavierStokesSUPG}{l^{\rm NaSt}_{\rm stab}}
\newcommand{\lNavierStokes}{l^{\rm NaSt}}
\newcommand{\ThetaM}[1]{#1^{-_{\Theta}}}
\newcommand{\ThetaP}[1]{#1^{+_{\Theta}}}
\title{Finite element formulation of general boundary conditions for incompressible flows}
\author{Roland Becker, Daniela Capatina, Robert Luce, and David Trujillo} 
\address{Equipe Concha, University of Pau, IPRA-LMA, avenue de l'universit{\a'e} BP 1155, 64013 Pau, France}
\begin{document}
\maketitle





%
%
\begin{abstract}
We study the finite element formulation of general boundary conditions for incompressible flow problems.
Distinguishing between the contributions from the inviscid and viscid parts of the equations, we use Nitsche's method to develop a discrete weighted weak formulation valid for all values of the viscosity parameter, including the limit case of the Euler equations. In order to control the discrete kinetic energy, additional consistent terms are introduced.
We treat the limit case as a (degenerate) system of hyperbolic equations, using a balanced spectral decomposition of the flux Jacobian matrix, in analogy with compressible flows. Then, following the theory of Friedrich's systems, the natural characteristic boundary condition is generalized to the considered physical boundary conditions.
Several numerical experiments, including standard benchmarks for viscous flows as well as inviscid flows are presented. 
\end{abstract}
%


\section{Introduction}\label{sec:introduction}
%
The subject of this article is the finite element formulation of general boundary conditions for incompressible flow problems in a bounded domain $\Omega\subset\mathbb R^d$ ($d=2,\,3$). The velocity field $v$ and pressure $p$ are governed by the Navier-Stokes equations
\begin{equation}\label{eq:NavierStokesIncompressible}
\rho(\frac{\partial v}{\partial t} + v \cdot \nabla v) +\nabla p - \mu\Delta v= f,\quad \div v = 0,
\end{equation}
together with initial condition  $v(0) = v_0$  and constants $\rho>0$ and $\mu\ge0$. For $\mu=0$ we have the Euler equations.

We consider five types of boundary conditions for (\ref{eq:NavierStokesIncompressible}): wall, inflow, outflow, symmetry and characteristic conditions, see Table~\ref{tab:BC}. Depending on whether the flow is inviscid or not, the boundary conditions change in nature, e.g., no-penetration versus no-slip in the case of a rigid wall. Correspondingly, we subdivide the boundary into $\partial\Omega=\GammaWall\cup\GammaIn\cup\GammaOut\cup\GammaSym\cup\GammaCharacteristic$ with $\GammaSym$ a hyperplane.
In what follows, $u=(v,p)$ and $B$ is a symmetric matrix related to the negative part of the Jacobian, see below.
\begin{table}[htdp]
\begin{center}
\begin{tabular}{c|c|c}
& $\mu=0$ & $\mu>0$\\\hline
$\GammaWall$ & $v\cdot n = 0$ & $v=0$,\\
$\GammaIn$ & $v = \vD$ & $v=\vD$,\\
$\GammaOut$ & $p = \pD$ & $\displaystyle{\mu\frac{\partial v}{\partial n}}- p n = -\pD n$,\\
$\GammaSym$ & $v\cdot n = 0$ & $v\cdot n = 0,\; \mu \displaystyle{\frac{\partial v}{\partial n} }\times n= 0$,\\
$\GammaCharacteristic$ & $B(u-\uD) = 0$ & $\displaystyle{(\mu\frac{\partial v}{\partial n},0)^T}-B(u-\uD)=0$.\\
\end{tabular}
\end{center}
\caption{Considered boundary conditions}
\label{tab:BC}
\end{table}
In contrast to the first four boundary conditions, the physical meaning of the characteristic boundary condition is less obvious, since it corresponds to an a priori unknown weighting of the different variables, depending on the definition of $B$. It is however the most natural one for a first-order system in the sense of Friedrich, see for example \cite{ErnGuermond06a, Jensen05}. Note that the outflow boundary condition is often used in order to limit the computational domain by introduction of an artificial boundary $\GammaOut$.

Our approach for developing a discrete weak formulation is outlined as follows. 
We distinguish between the contributions from the inviscid (Euler) and viscid (Stokes) parts of the equations and use Nitsche's method \cite{Nitsche71}, which has originally been developed for the Poisson problem; it has been extended to the Navier-Stokes equations, see for instance \cite{Becker02,BurmanFernandezHansbo06,BazilevsHughes07}. In the last cited paper the potential of the method to produce a physically meaningful weighting between diffusive and convective terms has been clearly demonstrated by comparison with the strong implementation of boundary conditions. This idea, which is particularly interesting for high P\a'eclet numbers, has then been extended in \cite{BazilevsMichlerCalo07} to turbulent flows by incorporating a wall law into the weak formulation.

In this paper, we use Nitsche's method to define a weighted weak formulation valid for all values of the viscosity parameter, including the limit case of the Euler equations. Our goal being the control of the discrete kinetic energy, additional consistent terms are further introduced in the discrete formulation. In order to limit the presentation, we focus here on continuous finite element spaces. Furthermore, in this paper we only discuss space discretization.

The analogous treatment for the convection-diffusion equation has been successfully applied in the literature, leading to robustness with respect to the diffusion parameter, see for example \cite{BeckerHansbo00}. In contrast to the case of the Navier-Stokes equations, the singular limit (the linear transport equation) is theoretically well-understood. Additional difficulties which arise in the present situation are the variety of boundary conditions and the coupling between velocities and pressure. Moreover, the meaning of robustness is not well-understood, since the incompressible Euler equations are known to admit very complex solutions. Their mathematical theory is an active topic of research, for example the blow-up in three dimensions \cite{Hou09}, or the notion of weak solutions \cite{Lions96,De-LellisSzekelyhidi13,SzekelyhidiWiedemann12}. In contrast to the compressible Euler equations, we cannot use entropies as a roadmap for the development of numerical methods. We therefore use the kinetic energy as a guideline, making sure that the discrete equations do not generate unphysical growth in energy.

The summary of the article is as follows.  Section~\ref{sec:WeakEulerChar} is devoted to the inviscid equations with the characteristic boundary condition. We write the Euler equations as a degenerate first-order system and introduce a balanced spectral decomposition of the flux Jacobian in order to define the boundary matrix $B$ in Table~\ref{tab:BC}. The term 'balanced' refers to the fact that the resulting boundary condition has the same dimensioning as the equations (\ref{eq:NavierStokesIncompressible}) in the interior of the domain.
 
Then in Section~\ref{sec:WeakEulerGen} we generalize this boundary condition to the other physical conditions of Table~\ref{tab:BC}, by letting the data of the characteristic condition depend on the unknowns. For the wall condition, such a technique is often employed in compressible flows, using reflection at a solid wall. However, it turns out that additional terms should be introduced in order to control the kinetic energy. These terms are consistent, except for the outflow condition in case of re-entrant flows, where we add an integral which corresponds to a modification of the outflow condition. 
Modifications aimed to increase stability in this case have previously been proposed \cite{BruneauFabrie96,BoyerFabrie07,BazilevsMichlerCalo10a,BraackMucha13} from a different point of view. 

In Section~\ref{subsec:WeakNavierStokes} we add the viscous terms to recover the Navier-Stokes equations.  We first introduce the discrete weak formulation for the Stokes equations, based on a generalization of Nitsche's method. Then we present the weak formulation for the Navier-Stokes equations and we briefly discuss the choice of stabilization terms in light of the balanced scaling of the absolute value of the Jacobian.
Further, for comparison with the proposed method, we present an alternative finite element discretization of the Navier-Stokes equations based on strong enforcement of the normal velocity in the discrete space. 
 
Finally, Section~\ref{sec:NumExp} presents various numerical experiments involving standard test cases. We use the backward facing step problem and the flow around a cylinder to investigate the behavior of the outflow boundary condition. The first example also illustrates the necessity to control the kinetic energy.
Then the Kovasznay flow is used to investigate robustness with respect to the viscosity parameter. As examples for inviscid flows, we consider the standing vortex problem, the rotational flow given by Fraenkel, and the impact of a jet.
Comparisons with the alternative discretization are also carried out for the Kovasznay flow and the jet impact problem. 
Although the presented computations were all based on equal-order $Q^1$ finite elements with SUPG stabilization, our theoretical results carry over to other continuous discrete finite element spaces. 
%
\subsection{Notation}\label{subsec:}
%
Let us first introduce some notation. The outward unit normal to $\partial\Omega$ is denoted by $n$. In 2D, we also use the notation $n^{\perp}=(-n_2,n_1)^T$. We will frequently write $v_n$ for $v\cdot n$ and $v_{n}^{\perp}=v-v_n n$. 
We denote for $x\in\R$ the positive part by $x^{+}:=\max\{x,0\}$ and the negative part by $x^{-}:=x-x^{+}$. Notice that 
$(-x)^{+}=-x^{-}$ and $(-x)^{-}=-x^{+}$.
Similarly, for a symmetric matrix $A$, we define $A^- :=  R\Lambda^- R^{T}$ and $|A|:=R|\Lambda| R^{T}$ if $A = R\Lambda R^{T}$ with $\Lambda$ diagonal and $R$ orthogonal. 

We will frequently use the symbols $\lesssim$ (and $\simeq$) in order to indicate that a quantity is bounded above (and below) up to a positive constant independent of the parameters of interest, such as physical parameters and discretization parameters. 

Throughout, we let $\mathbb{V}_h$ and $\mathbb W_h$ denote finite element spaces of continuous functions constructed on simplicial, quadrilateral, or hexahedral meshes of maximal cell width $d_h$. We denote by $d_K$ the diameter of the cell $K$. 
We denote by $u_h(t)=(v_h(t),p_h(t))\in \mathbb{V}_h\times \mathbb W_h$ the space-discrete solution and by $\psi_h=(\phi_h,\chi_h)\in \mathbb{V}_h\times \mathbb W_h$ a couple of test functions.

\section{Euler equations with characteristic boundary condition}\label{sec:WeakEulerChar}

We define the matrices
\begin{equation*}
A^i(u)=
\begin{bmatrix}
\rho v_{i}I & J_i\\
 J_i^T & 0 \\ 
\end{bmatrix},\quad A_n (u) = \sum_{i=1}^d n_i A^i(u)=\begin{bmatrix}
\rho v_n I & n\\
n^T & 0
\end{bmatrix},\quad
M = 
\begin{bmatrix}
\rho I& 0\\
 0& 0  \\ 
\end{bmatrix},
\end{equation*}
where $I$ is the $d\times d$ identity  whereas $J_i$ is the $d\times 1$ matrix of elements equal to $\delta_{ij}$ ($1\leq i,\, j\leq d$).

Then the inviscid part of (\ref{eq:NavierStokesIncompressible}) can be written as a first-oder system in quasi-linear form as
\begin{equation}\label{eq:QL}
M \frac{\partial u}{\partial t} +\sum_{i=1}^d A^i(u) \frac{\partial u}{\partial x_i} =F,\quad F=\begin{bmatrix} f\\ 0\end{bmatrix}.
\end{equation}

The natural boundary condition associated to (\ref{eq:QL}) is the characteristic boundary condition, which we write here as $B (u-\uD) = 0$ on $\partial\Omega=\GammaCharacteristic$ with $B$ to be defined.
 
 In the following, we first address the question of non-dimensionalization of the characteristic boundary condition and consequently, the choice of $B$. The simplest choice $B=A_n^-$ on $\GammaCharacteristic$ does not yield the desired property, as discussed in the next subsection. Then we give a weak formulation which satisfies an energy estimate.
%
\subsection{Balancing of the boundary condition}\label{subsec:BalancedBC}
%
Let us introduce 
\begin{equation}\label{eq:}
\Theta=
\begin{bmatrix}
I & 0 \\ 0 & \theta
\end{bmatrix}, \quad \theta >0
\end{equation}
and define:
\begin{equation*}
|A|_{\Theta} := \Theta^{-1}|\Theta A \Theta|\Theta^{-1}, \quad \ThetaM{A}:=\Theta^{-1}(\Theta A \Theta)^{-}\Theta^{-1}.
\end{equation*}
It is important to note that
\begin{equation}\label{eq:A_theta}
\frac12|A|_{\Theta}+\ThetaM{A}=\frac12 A.
\end{equation}
An appropriate absolute value function for Jacobian matrices has been used in the context of compressible flows in \cite{Barth99} in order to obtain a proper scaling of eigenvectors.

Here, the idea is to impose the characteristic boundary condition associated to (\ref{eq:QL}):
\begin{equation}\label{eq:CharBdryCond}
\ThetaM{A_n(u)}(u-\uD) = 0 \quad \text{on } \partial\Omega=\GammaCharacteristic,
\end{equation}
and to choose $\theta$ in such a way that this boundary condition scales as the Euler equations.

For this purpose, let us begin by stating the spectral decomposition of the matrix $\Theta A_n(u) \Theta$; the proofs of the two following lemmas are given in the Appendix.
\begin{lemma}\label{lemma:EulerEigen}
The symmetric matrix 
\begin{equation*}
\Theta A_n(u) \Theta= \begin{bmatrix}\rho v_n I & \theta n\\ \theta n^T & 0 \end{bmatrix}
\end{equation*}
 has real eigenvalues and a basis of orthonormal eigenvectors.  For $d=2$, the eigenvalues are given by
\begin{equation}\label{eq:Eigenvalues}
\lambda = \rho v_n,\quad \lambdapm :=  \frac{\rho v_{n}}{2} \pm \frac{\sqrt{4\theta^2+\rho ^2v_{n}^2}}{2},
\end{equation}
and the corresponding right eigenvectors by
\begin{equation*}
r=
\begin{bmatrix}
\nperp\\
0
\end{bmatrix},\quad
\rightpm =
\frac{1}{\sqrt{\theta^2+\lambdapm^2}}
\begin{bmatrix}
 \lambdapm n\\
\theta
\end{bmatrix}.
\end{equation*}
For $d=3$, $\lambdapm$ and $ \rightpm $ are the same and $\lambda=\rho v_n$ is a double eigenvalue, of corresponding eigenvectors $\begin{bmatrix} t_1\\0\end{bmatrix}$  and $\begin{bmatrix} t_2\\0\end{bmatrix}$ with $\lbrace t_1,\,t_2,\, n\rbrace $ an orthonormal basis of $\R^3$. 
\end{lemma}

\begin{lemma}\label{lemma:2}
Let $\psi=(\phi,\chi)$ and $\psi'=(\phi',\chi')$. Then
\begin{equation}\label{eq:Norm}
|A_n(u)|_{\Theta}\psi\cdot \psi'= \rho |v_n| \phi_{n}^{\perp}\cdot \phi_{n}^{'\perp} + \frac{\chi\chi'+2\theta^2\phi_n\phi'_n + (\chi+\rho v_n \phi_n)(\chi'+\rho v_n \phi'_n)}{\sqrt{4\theta^2+\rho^2 v_n^2}}
\end{equation}
and
\begin{equation}\label{eq:ThetaM_A_psi}
\begin{split}
\ThetaM{A_n(u)}\psi\cdot \psi' &=
 \rho v_n^- \phi_{n}^{\perp}\cdot \phi_{n}^{'\perp} - \frac{1}{\sqrt{4\theta^2+\rho^2 v_n^2}}(\chi+\lambdam \phi_n)(\chi'+\lambdam \phi'_n).
\end{split}
\end{equation}
One also has that 
\begin{equation*}
|A_n(u)|_{\Theta}\psi\cdot \psi\simeq \rho |v_n| \left(\phi_{n}^{\perp}\right)^2 + (\rho |v_n| + \theta)\phi_n^2 + \frac{\chi^2}{\rho |v_n| + \theta}.
\end{equation*}
\end{lemma}

Now, let us discuss the scaling of the Euler equations. Let $q=p/\rho$, $\varphi=f/\rho $  and
\begin{equation*}
\tilde x = s_x x,\quad\tilde t = s_t t,\quad\tilde v = s_v v,\quad \tilde q = s_q q,\quad \tilde \varphi = s_f \varphi.
\end{equation*}
Then, with $\mathcal{L}$ and $\mathcal{\tilde L}$ the differential operators in $\Omega$ and $\tilde\Omega=s_x\Omega$, we wish to have that
\begin{equation}\label{eq:Equival}
\mathcal{L}(v,q) = (\varphi, 0) \quad\Longleftrightarrow\quad \mathcal{\tilde L}(\tilde v,\tilde q) =(\tilde \varphi, 0) .
\end{equation}
Since the first component of $\mathcal{\tilde L}(\tilde v,\tilde q)$ is 
\begin{equation*}
\frac{\partial \tilde v}{\partial \tilde t} + \tilde v\cdot\tilde\nabla \tilde v + \tilde\nabla \tilde q = 
\frac{s_v^2}{s_x} \left( \frac{s_x}{s_ts_v}\frac{\partial v}{\partial t} + v\cdot\nabla v + \frac{s_q}{s_v^2}\nabla q  \right)
\end{equation*}
it follows that
\begin{equation*}
s_v = \frac{s_x}{s_t},\quad s_q = s_v^2,\quad s_f  = \frac{s_v^2}{s_x}.
\end{equation*}

Then, denoting now by $\mathcal{B}$ and $\mathcal{\tilde B}$ the differential operators associated to the boundary conditions on $\partial\Omega$ and $\partial\tilde\Omega$, we would like to choose $\theta$ such that the same property (\ref{eq:Equival}) holds true for  $\mathcal{B}$ and $\mathcal{\tilde B}$.
From the expression of the matrix $\ThetaM{A_n(u)}$ in Lemma~\ref{lemma:2}, it follows that
the characteristic boundary condition (\ref{eq:CharBdryCond}) translates into
\begin{equation*}
\lambdam(v-\vD)_n + (p- \pD)=0,\quad \rho v_n^-(v- \vD)_{n}^{\perp}=0.
\end{equation*}
Therefore, the scaling is balanced if $\tilde{\lambda}_{\rm m} = \lambdam s_v$, where 
$\tilde{\lambda}_{\rm m}=\left(\rho \tilde v_{n}-\sqrt{4\tilde{\theta}^2+\rho ^2\tilde {v}_{n}^2}\right)/2$ according to (\ref{eq:Eigenvalues}). 
This yields $\tilde\theta = s_v \theta$. Thus, $\theta$ is proportional to $ \rho|v_n|$ and to $ \sqrt{\rho |p|}$  at the continuous level. 

Now, taking into account time and space discretization with parameters $d_t$ and $d_K$, a natural extension is to require similar scaling for discrete solutions corresponding to $d_t$, $d_K$ and  $\tilde d_t = s_t d_t$, $\tilde d_K=s_x d_K$ respectively. An additional possibility  thus appears: $\theta$ proportional to $\rho d_K/d_t $.  In order to cover all choices, we take $\theta$ as a homogeneous function $\Psi$ of degree one: 
\begin{equation*}
\theta = \Psi \left (\rho|v_{h,n}|, \sqrt{\rho |p_h|},\frac{\rho d_K}{d_t}\right).
\end{equation*}

Although $\theta$ is defined locally on each cell $K$, we do not use a subscript for readability. Note that $\theta$ is not necessarily constant on $K$.

The obvious choice $\theta=1$ (such that $\Theta$ is the identity matrix) corresponds to the standard characteristic condition with $B=A_n^-$ but it does not yield a correct scaling. 

In what follows, we shall mostly use a general $\theta$. However, we will sometimes discuss the choice
\begin{equation}\label{eq:ChoiceTheta}
\theta^2= (\rho v_{h,n})^2 +  c_{\rm{dt}}^2\left(\frac{\rho d_K}{d_t}\right)^2.
\end{equation}

\begin{remark}\label{rmk:3}
For $\theta$ given in (\ref{eq:ChoiceTheta}), we obtain
\begin{equation}\label{eq:coerc}
|A_n(u)|_{\Theta}\psi\cdot \psi \simeq \rho |v_n| \phi^2 + \theta\phi_n^2 + \frac{\chi^2}{\theta}.
\end{equation}
We underline that the constants involved in the equivalence (\ref{eq:coerc}) are independent of all physical and numerical parameters, especially of 
$d_K$ and $d_t$ and are therefore valid for arbitrary CFL numbers and on locally refined meshes.
\end{remark}
%
%
\subsection{Weak formulation}\label{subsec:}
%

We now propose a weak space-discrete formulation for (\ref{eq:QL}) and (\ref{eq:CharBdryCond}). For the moment, we neglect the interior stabilization which will be discussed in Section \ref{subsec:WeakNavierStokes}. For this purpose, we introduce
\begin{equation}\label{eq:ql}
\begin{split}
\aql(u_h)(\psi_h) :=& \frac12 \sum_{i=1}^d\intOmega\left( A^i(u_h) \frac{\partial u_h}{\partial{x_i}}\cdot \psi_h - u_h\cdot A^i (u_h)\frac{\partial\psi_h}{\partial{x_i}} \right)
+ \frac12\int_{\partial\Omega}  |A_n(u_h)|_{\Theta} u_h \cdot\psi_h\\
\lql(u_h)(\psi_h) :=&\int_{\Omega} f\cdot\phi_h  -\int_{\partial\Omega}  \ThetaM{A_n (u_h)}\uD\cdot\psi_h\\
\end{split}
\end{equation}
and we consider the space-discrete problem: $u_h(t)\in \mathbb{V}_h\times\mathbb W_h $,
\begin{equation}\label{eq:QLDiscrete}
\rho\int_{\Omega}\frac{\partial v_h}{\partial t}\cdot\phi_h + \aql(u_h)(\psi_h) =\lql(u_h)(\psi_h)\quad\forall\psi_h=(\phi_h,\chi_h)\in \mathbb{V}_h\times\mathbb W_h.
\end{equation}

For a smooth solution $u=(v,p)$ to (\ref{eq:QL}) and (\ref{eq:CharBdryCond}), using integration by parts together with the property $\div v=0$, we have consistency:
%
\begin{equation}\label{eq:ConsistAql}
\rho\int_{\Omega}\frac{\partial v}{\partial t}\cdot\phi_h + \aql(u)(\psi_h) =\lql(u)(\psi_h)\quad\forall\psi_h\in \mathbb{V}_h\times\mathbb W_h.
\end{equation}

One immediately gets the energy balance of the formulation.
\begin{lemma}\label{lem:}(Energy estimate for $\aql$)
One has
\begin{equation*}
\aql(\psi_h)(\psi_h)= \frac12 \int_{\partial\Omega}|A_n(\psi_h)|_{\Theta} \psi_h \cdot \psi_h\geq 0, \quad \forall \psi_\in \mathbb{V}_h\times\mathbb W_h.
\end{equation*}
Clearly, any discrete solution $u_h$ to (\ref{eq:QLDiscrete}) satisfies
\begin{equation}\label{eq:QL_energy}
\frac{d}{d t}\int_{\Omega} \frac{\rho}{2} v_h^2+ \aql(u_h)(u_h)= 
\int_{\Omega} f\cdot v_h -\int_{\partial\Omega}\ThetaM{A_n(u_h)}\uD\cdot u_h.
\end{equation}
\end{lemma}
Since $|A_n(\psi_h)|_{\Theta}$ is positive symmetric, (\ref{eq:QL_energy}) shows energy dissipation through the boundary for $f=0$ and $\uD=0$.


Next we rewrite $\aql$ in order to get close to the standard mixed formulation, by removing the derivative on the pressure variables.
\begin{lemma}\label{lem:RefAql}(Reformulation of $\aql$)
Let $u_h=(v_h,p_h)$ and $\psi_h=(\phi_h,\chi_h)$. The form $\aql$ can be written as
\begin{equation*}
\begin{split}
\aql(u_h)(\psi_h)&= \intOmega \frac{\rho}{2}\bigg(\left(v_h\cdot \nabla v_h\right)\cdot\phi_h - v_h \cdot \left(v_h\cdot\nabla\phi_h\right)\bigg) + \intOmega \left( \chi_h \div v_h  - p_h\div\phi_h  \right)\\ 
&+\int_{\partial\Omega} \left(   \frac{\rho}{2} v_{h,n} v_h\cdot\phi_h +p_h \phi_{h,n} - \ThetaM{A_n(u_h)}u_h\cdot\psi_h \right)
\end{split}
\end{equation*}
\end{lemma}
\begin{proof}
We remark that
\begin{equation*}
\frac12 \int_{\Omega} \left( \nabla p_h\cdot\phi_h + \chi_h \div v_h  - p_h\div\phi_h- v_h\cdot\nabla\chi_h  \right) = 
 \int_{\Omega} \left( \chi_h \div v_h  - p_h\div\phi_h  \right) + \frac12\int_{\partial\Omega} \left( p_h  \phi_{h,n} - \chi_h v_{h,n}  \right). 
\end{equation*}
Thanks to (\ref{eq:A_theta}) and $A_n(u_h)u_h = ( \rho v_{h,n} v_h+p_h n , v_{h,n})$, it follows that
\begin{eqnarray*}
\int_{\partial\Omega}  \frac12 |A_n(u_h)|_{\Theta} u_h\cdot\psi_h 
&=&\int_{\partial\Omega}  \frac12 (\rho v_{h,n} v_h\cdot\phi_h +p_h \phi_{h,n} + \chi_h v_{h,n}) - \ThetaM{A_n(u_h)} u_h\cdot\psi_h
\end{eqnarray*}
and therefore
\begin{eqnarray*}
\int_{\partial\Omega} \frac12 |A_n(u_h)|_{\Theta} u_h\cdot\psi_h + \frac12\int_{\partial\Omega} ( p_h  \phi_{h,n} - \chi_h v_{h,n} )
=\int_{\partial\Omega}  \frac{\rho}{2} v_{h,n} v_h\cdot\phi_h +p_h  \phi_{h,n} - \ThetaM{A_n(u_h)} u_h\cdot\psi_h.
\end{eqnarray*}
\end{proof}

Thanks to the reformulation of $\aql$, we can write that 
\begin{equation}\label{eq:Reform_flux}
\begin{split}
\aql(u_h)(\psi_h) -\lql(u_h)(\psi_h)&=\intOmega \frac{\rho}{2}\bigg(\left(v_h\cdot \nabla v_h\right)\cdot\phi_h - v_h \cdot \left(v_h\cdot\nabla\phi_h\right)\bigg) \\
&+ \intOmega \left( \chi_h \div v_h  - p_h\div\phi_h  \right) -\int_{\Omega} f\cdot\phi_h + \int_{\partial\Omega}  \mathcal{F}( \uD, u_h,\psi_h)
\end{split}
\end{equation}
where the boundary contribution $ \mathcal{F}( \uD, u_h,\psi_h)$ is defined by:
\begin{equation}\label{ClassicalFluxReform}
\begin{split}
\mathcal{F}(\uD, u_h,\psi_h):&=\frac12 |A_n(u_h)|_{\Theta} u_h\cdot\psi_h+\ThetaM{A_n(u_h)}\uD\cdot \psi_h+\frac12 (p_h\phi_{h,n}-\chi_h v_{h,n})\\
&=\frac12 A_n(u_h) u_h\cdot\psi_h+\ThetaM{A_n(u_h)}(\uD-u_h)\cdot \psi_h+\frac12 (p_h\phi_{h,n}-\chi_h v_{h,n})\\
&=\frac{\rho}{2}v_{h,n} v_h\cdot\phi_h +p_h\phi_{h,n}+\ThetaM{A_n(u_h)}(\uD-u_h)\cdot \psi_h.
\end{split}
\end{equation}

%
%

\section{Euler equations with general boundary conditions}\label{sec:WeakEulerGen}
%

The relation (\ref{eq:Reform_flux}) is the basis for our definition of the weak form for the Euler equations endowed with the five types of boundary conditions described in Table~\ref{tab:BC}. For this purpose, we write the other boundary conditions in characteristic form, by replacing $\uD$ (given on $\GammaCharacteristic$) by some $\tilde u$, depending on the type of the boundary condition, on the available data as well as on the unknown itself.  Then we reformulate accordingly the boundary contribution $\mathcal{F}( \tilde u, u,\psi)$ in order to see which terms need to be added to the formulation, in order to obtain control over the kinetic energy. Next, in order to allow for re-entrant flows, the outflow condition is further modified. Finally, we propose a new formulation with boundary stabilization for the Euler equations. In order to focus on the treatment of boundary conditions, we do not discuss in this section the interior stabilization. We will come back to this topic in Subsection~\ref{subsec:Stab}.

%
\subsection{Reformulation of boundary conditions}\label{subsec:ReformEulerGen}
%
Let us note that all boundary conditions on $\partial\Omega$ can be written under the same form as on $\GammaCharacteristic$, that is 
\begin{equation}\label{eq:CLequivalent}
\ThetaM{A_n(u)}(\tilde u-u)=0
\end{equation}
with $\tilde u$ specific to each type of boundary conditions. On $\GammaCharacteristic$, we have $\tilde u=\uD$. 

In what follows, we choose  $\tilde u$ according to the available data on the remaining boundaries. 
%
\subsubsection{Wall and symmetry boundary condition}\label{subsubsec:}
%
We use the reflection of $u$, i.e.
\begin{equation*}
\tilde u  = u - (2v_n n, 0)^T.
\end{equation*}
Thanks to relation (\ref{eq:ThetaM_A_psi}) from Lemma \ref{lemma:2}, we have that
\begin{equation}\label{CLWall}
\ThetaM{A_n(u)} (\tilde u-u)\cdot \psi = -\frac{2\lambdam^2}{\theta^2+\lambdam^2}  v_n (\chi + \lambdam \phi_n).
\end{equation}

Hence, the wall condition $v_n=0$ is  equivalent to (\ref{eq:CLequivalent}).
%
\subsubsection{Inflow boundary condition}\label{subsubsec:}
%
Since the velocity is known, we choose $\tilde u = (v^D,p)^T$ such that $\tilde u - u = (v^D-v,0)^T$. Then
\begin{eqnarray*}
\ThetaM{A_n(u)}(\tilde u-u)\cdot \psi &=& \rho v_n^{-}(v^D-v)_{n}^{\perp}\cdot\phi_{n}^{\perp} +\frac{\lambdam^2}{\theta^2+\lambdam^2} (v^D-v)_n (\chi + \lambdam \phi_n)\\
&=&-\rho v_n^{-}(v-v^D)\cdot\phi  +\alpha (v-v^D)_n \phi_n-  \beta(v-v^D)_n \chi 
\end{eqnarray*}
where we have put for abbreviation
\begin{eqnarray}
\alpha(v) :&=& \rho v_n^{-}-\frac{\lambdam^3}{\theta^2+\lambdam^2}  = \frac{(\rho |v_n| - \sqrt{4\theta^2+\rho^2 v_n^2})^2}{4\sqrt{4\theta^2+\rho^2 v_n^2}}>0,\label{eq:alpha}\\
\beta (v) :&=& \frac{\lambdam^2}{\theta^2+\lambdam^2} = -\frac{\lambdam}{\sqrt{4\theta^2+\rho^2 v_n^2}}>0.\label{eq:beta}
\end{eqnarray}

The inflow condition $v=v^D $ implies  (\ref{eq:CLequivalent}); the equivalence holds if $v_n^{-}\neq 0$ on $\GammaIn$.

%
\subsubsection{Outflow boundary condition}\label{subsubsec:}
%
Since the pressure is known, we now choose $\tilde u = (v,p^D)^T$, such that $\tilde u - u = (0,p^D-p)^T$. Then
\begin{eqnarray*}
\ThetaM{A_n(u)}(\tilde u-u)\cdot \psi &=& \frac{\lambdam}{\theta^2+\lambdam^2} (p^D-p) (\chi + \lambdam \phi_n)\\
&=& \frac{1}{\sqrt{4\theta^2+\rho^2 v_n^2}} (p-p^D)\chi - \beta  (p-p^D)\phi_n
\end{eqnarray*}
with $\beta$ introduced in (\ref{eq:beta}). Again, the outflow condition $p=p^D$ is equivalent to  (\ref{eq:CLequivalent}).

%
\subsection{Reformulation of boundary contributions}\label{subsec:}
%
We now compute the term $ \mathcal F(\tilde u,u,\psi)$ introduced in (\ref{ClassicalFluxReform}) on each boundary.This allows us to see what stabilization terms are needed in order to obtain positivity of the form when $\psi=u$. Without loss of generality, we take here $u^D=0$. In what follows, $\alpha$ and $\beta$ are those introduced in (\ref{eq:alpha}) and (\ref{eq:beta}) respectively.
%
%
\subsubsection{Characteristic boundary condition}\label{subsubsec:}
%
%
%
Since $\tilde u=u^D=0$, obviously 
\begin{equation*}
\mathcal{F}(\tilde u, u,u)=\frac12 |A_n(u)|_{\Theta} u\cdot u
\end{equation*}
so $\mathcal{F}(\tilde u, u,u)$ is non-negative and no additional term is needed.
%
\subsubsection{Wall and symmetry boundary conditions}\label{subsubsec:}
%
%
It is useful to note first that
\begin{equation*}
\frac{\rho}{2} |v_n|+\alpha= \frac{2 \theta^2+ \rho^2 v_n^2 }{2 \sqrt{4\theta^2+\rho^2 v_n^2}},
\end{equation*}
which in view of Lemma \ref{lemma:2} yields that 
\begin{equation*}
\frac12|A_n(u)|_{\Theta} (v,0)\cdot (\phi,0)=\frac{\rho}{2} |v_n| v\cdot \phi+\alpha v_n\phi_n.
\end{equation*}
It follows that 
\begin{eqnarray}\label{CalculFlux}
\begin{split}
\frac{\rho}{2}v_n v\cdot\phi+p\phi_n &=\frac12|A_n(u)|_{\Theta} (v,0)\cdot(\phi,0)\\
&+(p\phi_n-\chi v_n)+\rho v_n^- v_n^{\perp}\cdot\phi_n^{\perp}+(\rho v_n^--\alpha)v_n\phi_n+\chi v_n.\\
\end{split}
\end{eqnarray}

So finally, we get thanks to (\ref{CLWall}) and to the relation $\rho v_n^--\alpha =\frac{\lambdam^3}{\theta^2+\lambdam^2}$ that
\begin{equation*}
\mathcal{F}(\tilde u, u,u)=\frac12 |A_n(u)|_{\Theta}(v,0)\cdot (v,0) -\frac{\lambdam^3}{\theta^2+\lambdam^2} v_n^2+\rho v_n^- \left(v_n^{\perp}\right)^2 +\left (1-\frac{2\lambdam^2}{\theta^2+\lambdam^2}\right) p v_n .
\end{equation*}

The terms $\rho v_n^- \left(v_n^{\perp}\right)^2$ and $\left(1-\frac{2\lambdam^2}{\theta^2+\lambdam^2}\right ) p v_n$ are not necessarily positive. Since they are both consistent with the boundary condition $v_n=0$, in order to control them we shall subtract the terms $\rho v_n^- v_n^{\perp}\cdot\phi_n^{\perp}$ and $\left(1-\frac{2\lambdam^2}{\theta^2+\lambdam^2}\right) \chi v_n$ from the weak formulation. The same approach is used for the other boundary conditions. 

%
\subsubsection{Inflow boundary condition}\label{subsubsec:}
%
%
We obtain thanks to (\ref{CalculFlux}) that
\begin{equation*}
\mathcal{F}(\tilde u, u,\psi)=\frac12 |A_n(u)|_{\Theta} (v,0)\cdot (\phi,0) +(p\phi_n-\chi v_n)+(1-\beta) \chi v_n, 
\end{equation*}
so $\mathcal{F}(\tilde u, u,u)=\frac12 |A_n(u)|_{\Theta} (v,0)\cdot (v,0) +(1-\beta) p v_n $. The term $(1-\beta) \chi v_n$ being of indefinite sign and consistent, it will be subtracted from the formulation.
%
\subsubsection{Outflow boundary condition}\label{subsubsec:}
%
%
We now get in view of the relation (\ref{eq:Norm}) that
\begin{equation*}
\mathcal{F}(\tilde u, u,\psi)=\frac12 |A_n(u)|_{\Theta} (0,p)\cdot (0,\chi)+\frac{\rho}{2} |v_n|v\cdot \phi +(1-\beta) p\phi_n +\rho v_n^-v\cdot \phi
\end{equation*}
which gives 
\begin{equation*}
\mathcal{F}(\tilde u, u,u)=\frac12 |A_n(u)|_{\Theta} (0,p)\cdot (0,p)+\frac{\rho}{2} |v_n|v^2+(1-\beta) p v_n +\rho v_n^-v^2.
\end{equation*}

The terms of indefinite sign are $(1-\beta) p \phi_n $ and $\rho v_n^-v\cdot \phi$.  Note that the latter is consistent only under the additional hypothesis $v_n^-=0$. 

\subsection{'Energy' boundary condition on the outflow}\label{subsec:ModifOutflow}
%
In order to avoid the hypothesis $v_n^-=0$ on $\GammaOut$, needed for  consistency, and to allow thus to treat re-entrant flows, we modify the boundary condition as follows:

\begin{equation}\label{eq:ModifOutflow}
\delta\rho v_n^- v+pn=p^Dn
\end{equation}
where $\delta$ is a numerical parameter to be determined later.
If $v_n^-=0$ we retrieve the initial outflow condition.  

We now choose $\tilde u = ((1-\displaystyle{\frac{\delta \rho v_n^-}{\lambda_m}})v_nn,p^D)^T$, such that 
\begin{equation*}
\tilde u - u = (-v_n^{\perp}-\frac{\delta \rho v_n^-}{\lambda_m}v_nn,p^D-p)^T.
\end{equation*}
Then thanks to Lemma \ref{lemma:2} we have
\begin{equation*}
\ThetaM{A_n(u)}(\tilde u-u)\cdot \psi =-\rho v_n^-v_{\nperp}\cdot\phi_{\nperp} -\frac{\lambdam}{\theta^2+\lambdam^2} (\delta\rho v_n^- v_n+p-p^D) (\lambdam \phi_n+\chi)
\end{equation*}
and the new outflow condition (\ref{eq:ModifOutflow}) is equivalent to  (\ref{eq:CLequivalent}) for any $\delta$.

We assume $p^D=0$ for the moment and write the boundary contribution with the help of (\ref{ClassicalFluxReform}) as
\begin{eqnarray*}
\mathcal{F}(\tilde u, u,\psi)&=&\frac{\rho}{2} |v_n|v_{n}^{\perp}\cdot \phi_{n}^{\perp} +\frac{1}{\sqrt{4\theta^2+\rho^2 v_n^2}}(\delta\rho v_n^-v_n+p)\chi + \frac{\rho}{2} v_n v_n \phi_n + p\phi_n\\
& & -\beta (\delta\rho v_n^-v_n+p)\phi_n, 
\end{eqnarray*}
where again $\beta$ is defined by (\ref{eq:beta}).
Putting
\begin{equation*}
Q(v, p,\delta)= \rho\left(\frac{1}{2} v_n^+ + (\frac{1}{2}-\delta)v_n^-\right) v_n ^2+\frac{1}{\theta}(\delta\rho v_n^- v_n+p)p
\end{equation*}
and using  
\begin{equation*}
\frac{1}{2} |v_n|=\frac{1}{2}v_n-v_n^- ,\quad v_n=v_n^++v_n^- ,
\end{equation*} 
yields
\begin{equation*}
\mathcal{F}(\tilde u, u,u)=\frac{\rho}{2} |v_n| \left (v_{n}^{\perp}\right)^2+Q(v, p,\delta) +(\delta\rho v_n^-v_n+p) 
\left( 
(1-\beta) v_n
+(4\theta^2+\rho^2 v_n^2)^{-1/2}p-\theta^{-1}p
\right)
.
\end{equation*}

The last term of $\mathcal{F}(\tilde u, u,u)$ is of indefinite sign but is consistent with the modified boundary condition, so it can be controlled as previously, by subtracting it from the weak formulation. Next, we choose $\delta$ in order to control the remaining terms. 

\begin{lemma}
Let $\delta=1$ and $\theta\geq \rho|v_n$. Then 
\begin{equation*}
Q(v, p,1)\simeq  \frac{p^2}{ \theta} + \rho  |v_n| v_n^2.
\end{equation*}
\end{lemma}
\begin{proof}
Using   $v_n^+-v_n^-=|v_n|$, we have 
\begin{equation*}
Q(v, p,1)= \frac{2p^2+2p\rho v_n^- v_n+\theta\rho |v_n| v_n ^2}{2 \theta}.
\end{equation*}
Thanks to Young's inequality and to $(v_n^-)^2=- v_n^-|v_n|$, we can write for any $\varepsilon>0$ that
\begin{align*}
2p^2+2p\rho v_n^- v_n+\theta\rho |v_n| v_n^2\geq   (2-\frac{1}{\varepsilon})p^2+\rho |v_n|(\theta+\varepsilon \rho v_n^-) v_n^2
\ge (2-\frac{1}{\varepsilon})p^2+(1-\varepsilon)\theta\rho |v_n| v_n^2,
\end{align*}
using $\theta\geq \rho|v_n|$ as well as $|v_n|+v_n^-\geq 0$. By choosing $\frac{1}{2}<\varepsilon<1$ we get the lower bound. The other inequality is obvious, so the announced equivalence holds.
\end{proof}
%
%
\begin{remark}\label{rmk:EnergyEuler}
For $\theta$ given in (\ref{eq:ChoiceTheta}) it immediately follows that the outflow terms are controlled by $ p^2/\theta+\rho  |v_n| v^2$, which is not the complete energy norm $\frac12 |A_n(u)|_{\Theta} u\cdot u$, since the term $\theta v_n^2$ is missing, see Remark~\ref{rmk:3}. However, our approach allows to control the pressure and normal velocities, which is not the case in \cite{BruneauFabrie96,BraackMucha13,BazilevsMichlerCalo10a}.
\end{remark}
From now on, we take $\delta=1$, for which we get (for an arbitrary $p^D$):
\begin{equation*}
\mathcal{F}(\tilde u, u,\psi)-(\rho v_n^-v_n+p-p^D)
\left( 
(1-\beta) \phi_n
+(4\theta^2+\rho^2 v_n^2)^{-\frac12}\chi-\theta^{-1}\chi
\right)
=\frac{\rho}{2} |v_n| v\cdot \phi +\frac{1}{\theta}(\rho v_n^-v_n+p-p^D)\chi +p^D\phi_n.
\end{equation*}
%
%

\subsection{Weak formulation}\label{subsec:StabForm}

Taking into account the previously developed stabilization terms as well as the choice of $\theta$ (\ref{eq:ChoiceTheta}) we now define
\begin{equation*}
\begin{split}
\aEuler(u_h)(\psi_h):=& \intOmega \frac{\rho}{2}\bigg(\left(v_h\cdot \nabla v_h\right)\cdot\phi_h - v_h \cdot \left(v_h\cdot\nabla\phi_h\right)\bigg) + \intOmega \left( \chi_h \div v_h  - p_h\div\phi_h  \right)\\ 
&+\int_{\GammaCharacteristic} \left(   \frac{\rho}{2} v_{h,n} v_h\cdot\phi_h +p _h \phi_{h,n} - \ThetaM{A_n(u_h)}u_h\cdot\psi_h \right)\\
&+\int_{\GammaWall\cup\GammaSym\cup\GammaIn} \left( \frac{\rho}{2}|v_{h,n}|v_h\cdot\phi_h +\alpha v_{h,n}\phi_{h,n} -\chi_h v_{h,n}  +p_h  \phi_{h,n} \right)\\
&+\int_{\GammaOut}\bigg( \frac{\rho}{2} |v_{h,n}| v_h\cdot\phi_h   +  \frac{1}{\theta} (\rho v_{h,n}^-v_{h,n}+p_h)\chi_h \bigg),\\
\end{split}
\end{equation*}
\begin{equation*}
\begin{split}
\lEuler(u_h)(\psi_h):=&\int_{\Omega} f\cdot \phi_h -\int_{\GammaCharacteristic}  \ThetaM{A_n(u_h)}\uD\cdot\psi_h +\int_{\GammaIn}\left( -\rho  v_{h,n}^{-}\vD\cdot\phi_h +\alpha \vD_n\phi_{h,n} -  \chi_h  \vD_n \right)\\
&+\int_{\GammaOut}\bigg( \frac{1}{\theta} \pD\chi_h -  \pD\phi_{h,n}\bigg).
\end{split}
\end{equation*}

\begin{remark}
 We have subtracted the positive (and consistent) term $-\frac{\lambdam^3}{\theta^2+\lambdam^2} v_{h,n}\phi_{h,n} $ on $\GammaWall\cup\GammaSym$, in order to get the same boundary term as on $\GammaIn$ (provided that $v^D=0$). 
\end{remark}

We consider the space-discrete problem: $u_h(t)\in \mathbb{V}_h\times\mathbb W_h $,
\begin{equation}\label{eq:EulerDiscrete}
\rho\int_{\Omega}\frac{\partial v_h}{\partial t}\cdot\phi_h + \aEuler(u_h)(\psi_h) =\lEuler(u_h)(\psi_h)\quad\forall\psi_h=(\phi_h,\chi_h)\in \mathbb{V}_h\times\mathbb W_h.
\end{equation}

\begin{lemma}\label{lem:}(Consistency of $\aEuler$)
Let $u=(v,p)$ be a smooth solution to (\ref{eq:NavierStokesIncompressible}) with $\mu=0$ and with the boundary conditions given in Table~\ref{tab:BC}, except that the outflow condition is replaced by (\ref{eq:ModifOutflow}). Then $u$  satisfies
\begin{equation}\label{eq:EulerWeak}
\rho \int_{\Omega}\frac{\partial v}{\partial t}\cdot\phi_h + \aEuler(u)(\psi_h) =\lEuler(u)(\psi_h)\quad\forall\psi_h=(\phi_h,\chi_h)\in \mathbb{V}_h\times\mathbb W_h.
\end{equation}
\end{lemma}
\begin{proof}
Since $u$ satisfies the characteristic boundary condition (\ref{eq:CLequivalent}) on $\partial \Omega$ with  $\tilde u$ defined in Subsection~\ref{subsec:ReformEulerGen}, it follows from (\ref{eq:ConsistAql}) that 
\begin{equation*}
\rho \int_{\Omega}\frac{\partial v}{\partial t}\cdot\phi_h + \aql(u)(\psi_h) =\lql(u)(\psi_h)\quad\forall\psi_h\in \mathbb{V}_h\times\mathbb W_h.
\end{equation*}
So we only have to check that
\begin{equation}\label{eq:EulerWeak2}
(\aEuler- \aql)(u)(\psi_h) -(\lEuler-\lql)(u)(\psi_h)=0\quad\forall\psi_h\in \mathbb{V}_h\times\mathbb W_h.
\end{equation}

Next, we recall from Lemma \ref{lem:RefAql} that
\begin{eqnarray*}
\begin{split}
(\aql -\lql)(u)(\psi_h) &=\intOmega \frac{\rho}{2}\bigg((v\cdot \nabla v)\cdot\phi_h - v \cdot (v\cdot\nabla\phi_h)\bigg) +\intOmega(  \chi_h \div v  - p\div\phi_h )\\
& +\int_{\partial \Omega} \mathcal{F}(\tilde u, u,\psi_h)-\int_{\Omega} f\cdot \phi_h
\end{split}
\end{eqnarray*}
so using the different expressions of $ \mathcal{F}(\tilde u, u,\psi_h)$ and the definitions of $\aEuler$ and $\lEuler$, (\ref{eq:EulerWeak2}) is equivalent to
\begin{eqnarray}\label{eq:StabTerm}
\begin{split}
&-\int_{\GammaWall\cup\GammaSym} \bigg( \rho v_n^- v_{n}^{\perp}\cdot\left(\phi_{h}\right)_{n}^{\perp}+(1-\frac{2\lambdam^2}{\theta^2+\lambdam^2}) \chi_h v_n-\frac{\lambdam^3}{\theta^2+\lambdam^2} v_n\phi_{h,n}\bigg)\\
&-\int_{\GammaIn}  (1-\beta) \chi_h (v-v^D)_n -\int_{\GammaOut}(\rho v_n^-v_n+p-p^D) \left( (1-\beta) \phi_{h,n}
+\left(\frac{1}{\sqrt{4\theta^2+\rho^2 v_{n}^2}}-\frac{1}{\theta}\right)\chi_h\right) =0.
\end{split}
\end{eqnarray}
This equality holds true due to the considered boundary conditions. It translates the consistency of the stabilization terms.
\end{proof}

\begin{lemma}\label{lem:}(Energy estimate for $\aEuler$)
For all $\psi_h=(\phi_h,\chi_h)\in \mathbb{V}_h\times\mathbb W_h$, one has
\begin{equation}\label{eq:Euler_positivity}
\begin{split}
\aEuler(\psi_h)(\psi_h)&= \frac12 \int_{\GammaCharacteristic}|A_n(\psi_h)|_{\Theta} \psi_h\cdot \psi_h  +\frac12 \int_{\GammaWall\cup\GammaSym\cup\GammaIn}   |A_n(\psi_h)|_{\Theta} (\phi_h,0)\cdot (\phi_h,0)  \\ 
&+ \int_{\GammaOut}   \left(\frac{\rho}{2}|\phi_{h,n}|\left(\phi_h\right)_{n}^{\perp}\cdot \left(\phi_h\right)_{n}^{\perp}+Q(\phi_h,\chi_h,1)\right).
\end{split}
\end{equation}
For vanishing data $f$, $\vD$, $\pD$, any discrete solution $u_h$ to (\ref{eq:EulerDiscrete}) clearly satisfies
\begin{equation*}\label{eq:Euler_energy}
\frac{d}{d t}\int_{\Omega} \frac{\rho}{2} v_h^2=- \aEuler(u_h)(u_h)\leq 0.
\end{equation*}
\end{lemma}
\begin{proof}
We use the corresponding expressions of $\mathcal{F}(\tilde {\psi}_h, \psi_h, \psi_h)$ on each boundary.
\end{proof}

\section{Navier-Stokes equations with general boundary conditions}\label{subsec:WeakNavierStokes}


\subsection{Weak formulation of the Stokes equations}\label{subsec:WeakStokes}
%
Here, we consider the Stokes equations, endowed with wall, inflow, outflow and symmetry boundary conditions. We first define the bilinear and linear forms corresponding to the viscous term with weak boundary conditions in the sense of Nitsche:
\begin{equation}\label{eq:NitscheVisc}
\begin{split}
\aVisc(u_h,\psi_h) :=& \intOmega\nabla v_h:\nabla\phi_h - \int_{\GammaWall\cup\GammaIn} \left(  \frac{\partial v_h}{\partial n}\cdot\phi_h + v_h\cdot\left (\frac{\partial \phi_h}{\partial n} - \frac{\gamma}{d_K}\phi_h\right )  \right)\\
&-\int_{\GammaSym} \left(  \frac{\partial v_h}{\partial n}\cdot n\phi_{h,n} + v_{h,n} \left (\frac{\partial \phi_h}{\partial n}\cdot n - \frac{\gamma}{d_K}\phi_{h,n}\right )  \right)\\
&-\int_{\GammaOut}\frac{1}{\theta}\left( \frac{\partial v_h}{\partial n}\cdot n \,\chi_h + \frac{\partial \phi_h}{\partial n}\cdot n \,p_h - \mu\frac{\partial v_h}{\partial n}\cdot n\,\frac{\partial \phi_h}{\partial n}\cdot n\right),
\\
\lVisc (\phi_h)&:=- \int_{\GammaIn} \vD\cdot\left (\frac{\partial \phi_h}{\partial n} - \frac{\gamma}{d_K}\phi_h\right ) 
-\int_{\GammaOut}\frac{1}{\theta}\frac{\partial \phi_h}{\partial n}\cdot n \,\pD,
\end{split}
\end{equation}
where $\gamma$ is a stabilization parameter; the terms on wall, inflow, and symmetry boundaries are standard, see for instance \cite{Becker02,BurmanFernandezHansbo06,BazilevsHughes07}, whereas the outflow is treated in order to fit with the previous formulation for the inviscid case. 

%
%

Then the space-discrete Stokes problem reads: $u_h(t)\in \mathbb{V}_h\times\mathbb W_h$,
\begin{equation}\label{eq:StokesDiscrete}
\rho\int_{\Omega}\frac{\partial v_h}{\partial t}\cdot\phi_h + \aStokes(u_h,\psi_h) = \lStokes (\psi_h) \quad\forall\psi_h\in \mathbb{V}_h\times\mathbb W_h
\end{equation}
where
\begin{equation*}
\begin{split}
\aStokes(u_h,\psi_h) &:= \mu \aVisc(u_h,\psi_h) + \intOmega \left( \chi_h \div v_h- p_h \div\phi_h \right)+ \int_{\GammaWall\cup\GammaIn\cup\GammaSym} \left( p_h \phi_{h,n} - \chi_h v_{h,n} \right) + \int_{\GammaOut}\frac{1}{\theta} p_h\chi_h\\
\lStokes (\psi_h)&:=\mu \lVisc(\phi_h) +\intOmega f\cdot\phi_h  - \int_{\GammaIn} \chi_h\vD_n -\int_{\GammaOut} \pD \phi_{h,n}+ \int_{\GammaOut}\frac{1}{\theta} \pD\chi_h.\\
\end{split}
\end{equation*}
The well-posedness of this discrete Stokes problem for $\gamma$ sufficiently large follows from standard arguments.

\begin{remark}
Here we have not considered the characteristic condition $(\mu \displaystyle{\frac{\partial v}{\partial n} },0)^T-\ThetaM{A_n(u)}(u-\uD)=0$  since it is not natural for the Stokes equations. For the generalization to the Navier-Stokes formulation, this condition leads to the following additional  terms to $\aStokes$ and $\lStokes$ respectively:
\begin{equation*}
\int_{\GammaCharacteristic} \left(  p _h \phi_{h,n} - \ThetaM{A_n(u_h)}u_h\cdot\psi_h \right),\quad -\int_{\GammaCharacteristic}  \ThetaM{A_n(u_h)}\uD\cdot\psi_h. 
\end{equation*}
\end{remark}

It is obvious that the formulation (\ref{eq:StokesDiscrete}) is consistent in the sense that a sufficiently smooth solution $u=(v,p)$ to the Stokes equations satisfies the discrete equations
\begin{equation*}
\rho\int_{\Omega}\frac{\partial v}{\partial t}\cdot\phi_h + \aStokes(u,\psi_h) = \lStokes (\psi_h) \quad\forall\psi_h=(\phi_h,\chi_h)\in \mathbb{V}_h\times\mathbb W_h.
\end{equation*}
Depending on the employed discrete spaces, one may need to introduce additional stabilization terms; we discuss it in subsection~\ref{subsec:Stab}.

%

\subsection{Weak formulation of the Navier-Stokes equations}
The scaling of the Euler equations has been discussed in Subsection \ref{subsec:BalancedBC}. We are now dealing with the Navier-Stokes equations, so we also have to consider the viscous term.  Since
\begin{equation*}
 \tilde{ \Delta}\tilde{ v}=\frac{s_v}{s_x^2}\Delta v,
 \end{equation*}
it follows according to (\ref{eq:Equival}) that $s_v s_x=1$. Together with the previous relation $s_v = s_x/s_t $, this leads to the well-known unique scaling of the Navier-Stokes equations which satisfies (\ref{eq:Equival}), 
\begin{equation*}
s_t=s_x^2,\quad s_v = \frac{1}{s_x},\quad s_q =\frac{1}{s_x^2},\quad s_f  = \frac{1}{s_x^3}.
\end{equation*}
Recalling that $\tilde{\theta}=s_v\theta$, we deduce that $\theta$ may now also depend on $\mu/d_K $.  

We now take into account both the viscid and the inviscid parts and define:
\begin{equation*}
\begin{split}
\aNavierStokes(u_h)(\psi_h):=&\aEuler(u_h)(\psi_h)+\mu \aVisc(u_h,\psi_h),\\
\lNavierStokes(u_h)(\psi_h):=&\lEuler(u_h)(\psi_h)+\mu \lVisc (\phi_h).
\end{split}
\end{equation*}



In order to take into account stabilization, we define
\begin{equation*}
\aNavierStokesSUPG(u_h)(\psi_h) := \aNavierStokes(u_h)(\psi_h) + \astab(u_h)(\psi_h),\quad
\lNavierStokesSUPG(u_h)(\psi_h) := \lNavierStokes(u_h)(\psi_h) + \lstab(u_h)(\psi_h)
\end{equation*}
with the forms $\astab$ and $\lstab $  to be defined in Subsection~\ref{subsec:Stab}.

The space-discrete problem then reads: $u_h(t)\in \mathbb{V}_h\times\mathbb W_h $,
\begin{equation}\label{eq:NavierStokesDiscrete}
\rho\int_{\Omega}\frac{\partial v_h}{\partial t}\cdot\phi_h +\aNavierStokesSUPG(u_h)(\psi_h) =\lNavierStokesSUPG(u_h)(\psi_h)\quad\forall\psi_h=(\phi_h,\chi_h)\in \mathbb{V}_h\times\mathbb W_h.
\end{equation}

\subsection{Balanced SUPG stabilization}\label{subsec:Stab}
%
Following the idea of SUPG \cite{BrooksHughes82, HughesFrancaMallet86}, we consider the following stabilization for the Navier-Stokes equations:
\begin{eqnarray*}
\begin{split}
\astab(u_h)(\psi_h) &:= \sum_{K\in\Cells_h}  \left( \int_K \gamma_{K,1} E^{u_h}(u_h)\cdot E^{u_h}(\psi_h) +  \int_K \gamma_{K,2} \div v_h \div\phi_h\right),\\
\lstab(u_h)(\psi_h) &:= \sum_{K\in\Cells_h}  \int_K \gamma_{K,1} f\cdot E^{u_h}(\psi_h)
\end{split}
\end{eqnarray*}
with the parameters $\gamma_{K,i}$ to be specified and with
\begin{equation*}
E^{u_h}(\psi_h):=\rho\frac{\partial \phi_h}{\partial t}+\rho v_h\cdot\nabla \phi_h-\mu \Delta \phi_h+\nabla \chi_h.
\end{equation*}

Clearly, $\astab(u_h)(u_h)\geq 0 $ and $\astab(u)(\psi_h)=\lstab(u)(\psi_h)$ for $u$ sufficiently smooth solution of the Navier-Stokes system.

In \cite{BeckerCapatinaLuceTrujillo14d}, we have discussed how to tune the stabilisation parameters in order to get robustness with respect to the (local) P\'eclet number
$\Pe=\rho |v_h |d_K/\mu$. For large P\'eclet numbers, the flow is governed by the Euler part of the equations, whereas for small P\'eclet numbers, the Stokes part becomes dominant. 
The parameters $\gamma_{K,i} $  depend on the (local) parameter $\theta$ such that the stabilisation has the same scaling as the equations. Following \cite{BeckerCapatinaLuceTrujillo14d}, we take next on each cell $K$
\begin{equation}\label{eq:ChoiceThetaFinal}
\theta^2:=(\rho v_h)^2 +c_{\rm{dt}}^2(\frac{\rho d_K}{d_t})^2+c_{\rm{St}}^2(\frac{\mu}{d_K})^2
\end{equation}
with the constants $c_{\rm{dt}}$, $c_{\rm{St}}$  satisfying $c_{\rm{dt}}^2+c_{\rm{St}}^2>0$. The stabilization parameters are taken as follows:
\begin{equation*}
\gamma_{K,1}=\gamma_1 \frac{d_K}{\theta},\quad \gamma_{K,2}=\gamma_2\, d_K \theta,\quad \gamma_1,\,\gamma_2>0.
\end{equation*}
Note that $\theta$ is not necessarily constant on $K$. This cell-wise definition induced by the SUPG stabilisation is compatible with the one previously introduced in (\ref{eq:ChoiceTheta}) for the boundary conditions in the case $\mu=0$. 

Also note that we thus recover formulas similar to those proposed in the SUPG literature (see for instance \cite{FrancaFrey92}). One retrieves the stabilisation used in the inviscid case by simply taking $\mu=0$.

\begin{remark}
Recalling that $\min\lbrace a,b,c \rbrace$ can be approximated by $(1/a^2 +1/b^2+1/c^2)^{-1/2}$, we can  write that 
\begin{equation*}
\gamma_{K,1}\simeq \gamma_1 \min\lbrace \frac{d_K}{ \rho |v_h|}, \frac{d_t}{\rho}, \frac{d_K^2}{\mu}\rbrace,\quad \gamma_{K,2}\simeq \gamma_2 \max\lbrace \frac{ \rho |v_h|}{d_K}, \frac{\rho}{d_t}, \frac{\mu}{d_K^2} \rbrace.
\end{equation*}
\end{remark}

\subsection{Kinetic energy estimate}\label{subsec:NaStEB}

By putting together the previous results, we immediately get:

\begin{theorem}\label{thm:CoercNS}(Coercivity of  $\aNavierStokes$)
For all $\psi_h=(\phi_h,\chi_h)\in \mathbb{V}_h\times\mathbb W_h$, one has
\begin{equation*}
\begin{split}
\aNavierStokes(\psi_h)(\psi_h)=&
\mu\intOmega|\nabla \phi_h|^2 + \mu\int_{\GammaWall\cup\GammaIn}\left(\frac{\gamma}{d_K}\phi_h^2  - 2  \phi_h\cdot\frac{\partial \phi_h}{\partial n}\right)+ \mu\int_{\GammaSym}\left(\frac{\gamma}{d_K}\phi_{h,n}^2  - 2  \phi_{h,n}\frac{\partial \phi_h}{\partial n}\cdot n\right)\\
&+
\frac12 \int_{\GammaCharacteristic}|A_n(\psi_h)|_{\Theta} \psi_h\cdot \psi_h  +\frac12 \int_{\GammaWall\cup\GammaSym\cup\GammaIn}   |A_n(\psi_h)|_{\Theta} (\phi_h,0)\cdot (\phi_h,0)  \\ 
&+ \int_{\GammaOut}\left(\frac{\rho}{2}|\phi_{h,n}|\left(\phi_h\right)_{n}^{\perp}\cdot \left(\phi_h\right)_{n}^{\perp}+Q(\phi_h, \chi_h-\mu\frac{\partial \phi_h}{\partial n}\cdot n,1)\right).
\end{split}
\end{equation*}
\end{theorem}

\begin{remark}\label{rmk:}
For $\theta$ given in (\ref{eq:ChoiceThetaFinal}) and $\gamma$ sufficiently large we obtain, according to Remark~\ref{rmk:EnergyEuler},
\begin{equation*}
\begin{split}
\aNavierStokes(\psi_h)(\psi_h) \gtrsim &
\, \mu\intOmega|\nabla \phi_h|^2 + \int_{\GammaWall\cup\GammaIn}\frac{\mu}{d_K}\phi_h^2 + \int_{\GammaSym}\frac{\mu}{d_K}\phi_{h,n}^2 +\frac12\int_{\GammaCharacteristic}|A_n(\psi_h)|_{\Theta} \psi_h\cdot \psi_h  \\
+&\frac12 \int_{\GammaWall\cup\GammaSym\cup\GammaIn}   \left(\rho  |\phi_{h,n}| \phi_h^2 +\theta \phi_{h,n}^2\right) + \frac12 \int_{\GammaOut}    \left(\rho  |\phi_{h,n}| \phi_h^2 +\frac{1}{\theta}\left(\chi_h-\mu\frac{\partial \phi_h}{\partial n}\cdot n\right)^2\right).
\end{split}
\end{equation*}
\end{remark}

\begin{theorem}\label{thm:EnergyNS}(Energy estimate for Navier-Stokes)
Any discrete solution $u_h$ to (\ref{eq:NavierStokesDiscrete}) satisfies
\begin{equation*}
\begin{split}
\frac{d}{dt}\int_{\Omega}\frac{\rho}{2} v_h^2 =& -\mu\intOmega|\nabla v_h|^2 +\int_{\Omega} f\cdot v_h -\left(\astab(u_h)(u_h)-\lstab(u_h)(u_h)\right)\\
&- \int_{\GammaCharacteristic}\left(|A_n(u_h)|_{\Theta}u_h+\ThetaM{A_n(u_h)}u^D\right)\cdot u_h\\
&- \mu\int_{\GammaWall\cup\GammaIn}\left((v_h-\vD)\cdot\left(\frac{\gamma}{d_K}v_h-\frac{\partial v_h}{\partial n}\right) + v_h\cdot\frac{\partial v_h}{\partial n}\right) - \mu\int_{\GammaSym}v_{h,n}\left(\frac{\gamma}{d_K}v_{h,n}-2\frac{\partial v_h}{\partial n}\cdot n\right)\\
&-\frac12 \int_{\GammaWall\cup\GammaSym\cup\GammaIn}|A_n(u_h)|_{\Theta} (v_h,0)\cdot (v_h,0) 
- \int_{\GammaIn}  \rho v_{h,n}^{-}\vD\cdot v_h + \vD_n (p_h-\alpha v_{h,n})\\ 
&  -\int_{\GammaOut}\left(\frac{\rho}{2}|v_{h,n}|\left(v_{h}\right)_n^{\perp}\cdot \left(v_h\right)_n^{\perp}+Q(v_h,p_h-\mu\frac{\partial v_h}{\partial n}\cdot n, 1)\right) 
+ \int_{\GammaOut}   \left(  \frac{1}{\theta} \left( p_h -\mu\frac{\partial v_h}{\partial n}\cdot n \right) - v_{h,n} \right)\pD.
%
\end{split}
\end{equation*}
For vanishing data $f$, $v^D$, $p^D$ one gets
\begin{equation*}
\frac{d}{dt}\int_{\Omega}\frac{\rho}{2} v_h^2 = - \aNavierStokes(u_h)(u_h)-\astab(u_h)(u_h)
\end{equation*}
so $ \displaystyle{\frac{d}{dt}}\int_{\Omega}\displaystyle{\frac{\rho}{2}} v_h^2 \leq 0$ for $\theta$ given in (\ref{eq:ChoiceThetaFinal}) and for $\gamma$ sufficiently large.
\end{theorem}
%
\subsection{Conservation of momentum}\label{sec:Momentum}
%
In addition to kinetic energy, further physical quantities such as linear momentum, vorticity, and helicity are conserved by solutions to the Euler equations. A natural extension would then be to construct approximations with similar behavior. These questions have been recently discussed in \cite{EvansHughes13b,EvansHughes13c}, where for example momentum conservation in simple domains has been shown for isogeometric B-spline methods, with the key ingredient of discrete solenoidal functions. The purpose of this subsection is to give a discrete balance of linear momentum for our method based on classical $ Q^1\times Q^1$ finite elements.

The continuous equations have the global momentum balance:
\begin{equation}\label{eq:MomentumConservationContinuous}
\rho\frac{d}{dt}\int_{\Omega} v = \int_{\Omega} f +\int_{\partial\Omega}q(u)\quad\mbox{with}\quad  q(u) := (\mu \nabla v - p I + \rho v\otimes v)n.
\end{equation}
%

In order to obtain a discrete analogue of (\ref{eq:MomentumConservationContinuous}) we integrate by parts the term $ \rho v_h \cdot (v_h\cdot \nabla \phi_h) /2 $ coming from the anti-symmetrization in $\aEuler$. We then have 
\begin{equation}\label{eq:DefEuler}
\aEuler(u_h)(\psi_h)= a_0(u_h)(\psi_h)+ b(p_h,\phi_h) - b(\chi_h,v_h),
\end{equation}
where:
\begin{equation}\label{Def_a0primal}
\begin{split}
a_0(u_h)(\psi_h)=& \intOmega \rho \left( (v_h\cdot \nabla v_h)\cdot \phi_h + \frac{1}{2} \div v_h (v_h \cdot\phi_h)\right)-\int_{\partial \Omega}\rho v_{h,n}^- v_h\cdot\phi_h+\int_{\GammaOut} \frac{1}{\theta} \left(\rho v_{h,n}^- v_{h,n}+p_h\right) \chi_h\\ 
&+\int_{\GammaWall\cup\GammaSym\cup\GammaIn}  \alpha v_{h,n}\phi_{h,n}+\int_{\GammaCharacteristic} \left(\rho v_{h,n}^- v_h\cdot\phi_h + p_h\phi_{h,n}- \ThetaM{A_n(u_h)} u_h\cdot\psi_h\right),\\
b(p_h,\phi_h)=&-\int_{\Omega} p_h\div \phi_h+\int_{\GammaWall\cup\GammaIn\cup\GammaSym} p_h \phi_{h,n}.
\end{split} 
\end{equation}
Let $e^i$ denote the unit vector of the $i$-th coordinate. 
Taking the test function $\psi_h^i=(e^i,(\rho v_h\cdot e^i)/2)$, which we are allowed to do thanks to the weak formulation and to the equal-order approximation, we obtain from (\ref{Def_a0primal}) for the Euler part and from (\ref{eq:NitscheVisc}) for the viscous contribution:
\begin{align*}
\lNavierStokes(u_h)(\psi_h^i) - \aNavierStokes(u_h)(\psi_h^i) 
=\int_{\Omega} f^i  
+\int_{\partial\Omega} q^i(u_h)-\int_{\partial\Omega} \eps_h^i(u_h), \quad 1\leq i\leq d
\end{align*}
with
\begin{align*}
\eps_h(u_h) :=
\begin{cases}
\mu \frac{\partial v_h}{\partial n}- \Psi_h\ThetaM{A_n(u_h)}(u_h-\uD) & \mbox{on $\GammaCharacteristic$}\\
\mu \frac{\partial v_h}{\partial n}-\rho v_{h,n}^- v_h-(p_h-\pD) n +
\frac{\rho}{ 2\theta}  \left(\rho v_{h,n}^-v_{h,n}+p_h-\pD-\mu\frac{\partial v_h}{\partial n}\cdot n\right) v_h & \mbox{on $\GammaOut$}\\
v_{h,n} \left(\alpha n-\frac{\rho}{2} v_h\right)-\rho  v_{h,n}^{-}v_h + \frac{\gamma\mu}{d_K}v_h & \mbox{on $\GammaWall$}\\
v_{h,n} \left(\alpha n -\frac{\rho}{2} v_h\right)-\rho  v_{h,n}^{-}v_h +  \frac{\gamma \mu}{d_K} v_{h,n} n  -  \mu \left(\frac{\partial v_h}{\partial n}\right)_n^{\perp} & \mbox{on $\GammaSym$}\\
 ( v_{h,n} -   \vD_n ) \left(\alpha n- \frac{\rho}{2} v_h\right)-\rho  v_{h,n}^{-}(v_h -\vD)  +  \frac{\gamma\mu}{d_K}(v_h-\vD) & \mbox{on $\GammaIn$}\\
\end{cases}
\end{align*}
and with $\Psi_h $ the matrix of lines $\psi_h^i$.
Therefore with $q_h := q - \eps_h$ and taking into account SUPG stabilisation, we have the discrete momentum balance:
\begin{equation}\label{eq:MomentumConservationContinuousDiscrete}
\frac{d}{dt}\int_{\Omega} \rho v_h = \int_{\Omega} f -\sum_{K\in\Cells_h}  \int_K \frac{\gamma_{K,1} \rho}{2} \left(\rho\frac{\partial v_h}{\partial t}+\rho v_h\cdot\nabla v_h-\mu \Delta v_h+\nabla p_h\right) \cdot \nabla v_h + \int_{\partial\Omega}q_h(u_h).
\end{equation}
%

%
\subsection{Alternative discretization with strong enforcement of normal velocity}\label{sec:AlternativeDiscretization}
%
In order to ease the comparison with other formulations of the Navier-Stokes equations, we integrate by parts in $\aEuler$ the term $ \rho (v_h\cdot \nabla v_h) \cdot \phi_h /2 $ and use the expression of $ A_n(u_h)u_h \cdot\psi_h$  on $\GammaCharacteristic$ to obtain:
\begin{equation}\label{Def_a0}
\begin{split}
a_0(u_h)(\psi_h)=& -\intOmega \rho \left(v_h\cdot (v_h\cdot \nabla \phi_h) + \frac{1}{2} \div v_h (v_h \cdot\phi_h)\right)+\int_{\partial \Omega}\rho v_{h,n}^+ v_h\cdot\phi_h 
+\int_{\GammaOut} \frac{1}{\theta} \left(\rho v_{h,n}^- v_{h,n}+p_h\right) \chi_h\\
&+\int_{\GammaWall\cup\GammaSym\cup\GammaIn}  \alpha v_{h,n}\phi_{h,n}+\int_{\GammaCharacteristic} \left(-\rho v_{h,n}^+ v_h\cdot\phi_h-\chi_h v_{h,n}+ \ThetaP{A_n(u_h)} u_h\cdot\psi_h\right). 
\end{split} 
\end{equation}

Starting from (\ref{Def_a0}) we give an alternative formulation with strong enforcement of the normal velocity on the inflow and wall boundaries, following the ideas of \cite{BazilevsHughes07, BazilevsMichlerCalo10a}.
On the outflow, we impose the 'energy' boundary condition $\rho v_{n}^-v+pn=\pD n $. In addition, in order to highlight the influence of the scaling, we also consider a characteristic boundary with the standard condition $A_n^-(u)(u-\uD)=0$ with $\theta=1$ (that is, not correctly scaled).  

Let the space of the velocity test-functions
$\widetilde{\mathbb{V}}_h=\lbrace v_h \in \mathbb{V}_h; \, v_{h,n}=0 \text{ on }\partial\Omega\setminus (\GammaOut\cup\GammaCharacteristic)\rbrace$. The alternative discrete formulation is obtained by replacing $a_0 $ in (\ref{Def_a0}) by

\begin{equation*}
\tilde{a}_0(u_h)(\psi_h)= -\intOmega \rho \phi_h \cdot (v_h\cdot \nabla \phi_h) +\int_{\GammaOut}\rho v_{h,n}^+ v_h\cdot\phi_h + \int_{\GammaCharacteristic} \left(-\chi_h v_{h,n}+A_n^+(u_h)u_h\cdot\psi_h\right),
\end{equation*} 
Using (\ref{eq:DefEuler}) we obtain the corresponding expression of  $\tilde{a}^{\rm Eu}$ and consequently of $\tilde{a}^{\rm NaSt}$. The viscous term and the SUPG-stabilization are unchanged.

The alternative discrete problem which will be further used for comparison reads: $\tilde{u}_h(t)\in \widetilde{\mathbb{V}}_h\times\mathbb W_h $,
\begin{equation}\label{eq:NavierStokesDiscreteAlternative}
\rho\int_{\Omega}\frac{\partial \tilde{v}_h}{\partial t}\cdot\phi_h +(\tilde{a}^{\rm NaSt}+\astab)(\tilde{u}_h)(\psi_h) =(\tilde{l}^{\rm NaSt}+\lstab)(\tilde{u}_h)(\psi_h)\quad\forall\psi_h\in \widetilde{\mathbb{V}}_h\times\mathbb W_h.
\end{equation}

\section{Numerical experiments}\label{sec:NumExp}
%
First we use an analytical solution of the Navier-Stokes equations for all $\mu>0$ in order to validate our code. We next carry out some experiments concerning the outflow condition. In the case of re-entrant flow, comparisons with the 'do-nothing' condition are also made. Finally, we present three test cases for inviscid flows. 

We use quadrilateral meshes and $Q_1\times Q_1$ finite elements for the spatial discretization. For the time-dependent problems, we use the BDF2 scheme with small time-steps in order to make sure that the spatial errors dominate. The values that we use for the different stabilization constants are the following:
\begin{equation*}
c_{\rm dt}=0.1,\quad c_{\rm St}=4,\quad \gamma=100,\quad\gamma_1=0.25,\quad \gamma_2=0.1.
\end{equation*}
We have discussed these values in \cite{BeckerCapatinaLuceTrujillo14d}.
%

\subsection{Kovasznay's exact solution}\label{subsec:ExactSol}
%
\subsubsection{Convergence under mesh refinement}\label{subsubsec:ExactSol1}
%
\begin{figure}[htdp]
\begin{center}
\subfigure[$\mu=0.025$.]{\includegraphics[width=0.99\textwidth]{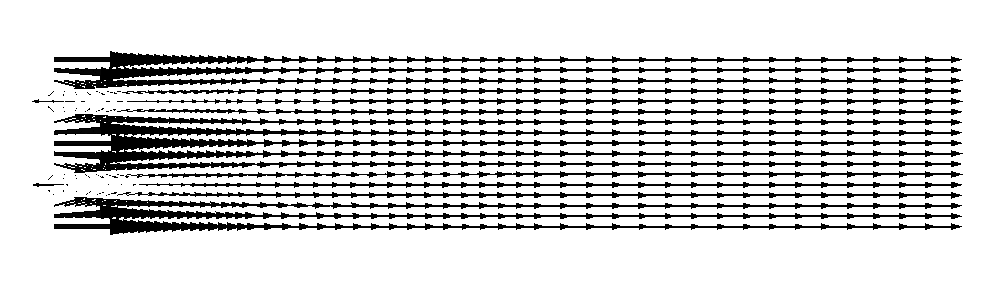}}
\subfigure[$\mu=0.0025$.]{\includegraphics[width=0.99\textwidth]{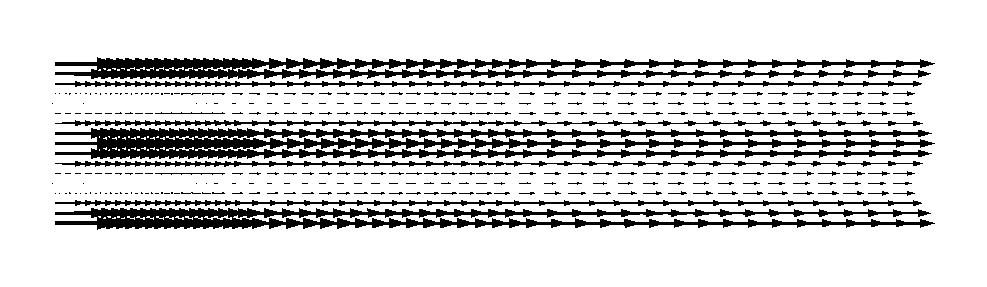}}
\subfigure[$\mu=0.00025$.]{\includegraphics[width=0.99\textwidth]{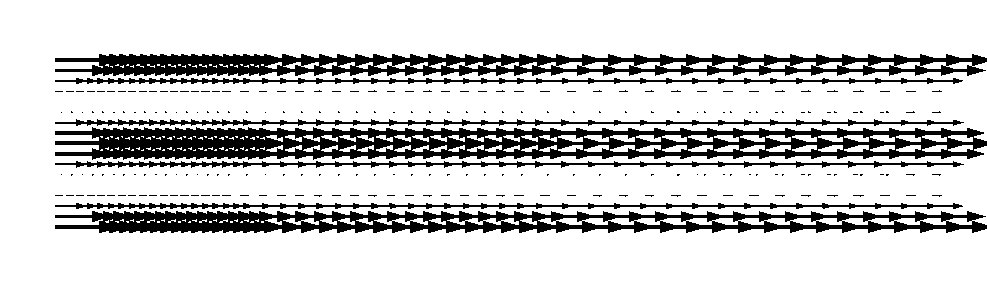}}
\subfigure[$\mu=0.0001$ at $t\approx30s$.]{\includegraphics[width=0.99\textwidth]{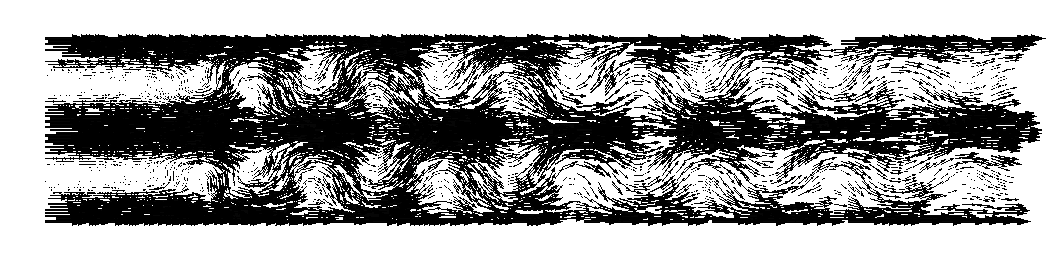}}
        \caption{Kovasznay flow with $N=3072$ nodes: velocity fields}
        \label{fig:Kovasznay1}
\end{center}
\end{figure}
In order to validate our code, we use the analytical solution given by Kovasznay for a solution of the Navier-Stokes equations, mimicking the two-dimensional flow behind a grid of circular cylinders \cite{Kovasznay48}. The solution is given by 
\begin{equation}\label{eq:kovasznay}
v = (1-e^{\lambda x_1}\cos(2\pi x_2),e^{\lambda x_1}\sin(2\pi x_2))^T,\; p= \frac12 -e^{2\lambda x_1},\;
\lambda = \frac{8\pi^2 \mu}{1+\sqrt{1+16\pi^2\mu^2}}
\end{equation}
We use the computational domain $\Omega=]-0.5, 10[\times]-0.5,1.5[$ and prescribe the analytical expression (\ref{eq:kovasznay}) on the left ($x_1=-0.5$), the upper ($x_2=1.5$) and lower ($x_2=-0.5$) boundaries, whereas the right ($x_1=10$) boundary is treated as an outflow.  Three test cases with $\mu=0.025,0.0025,0.00025$ corresponding to Reynolds numbers $40,400$, and $4000$ are considered. The vector fields is shown in Figure~\ref{fig:Kovasznay1}.

In Tables~\ref{tab:Kov1},\ref{tab:Kov2} and \ref{tab:Kov3} we show the pressure error in $L^2(\Omega)$-norm and the velocity errors in $H^1(\Omega)$- and $L^2(\Omega)$-norms; for each one, we also indicate the convergence order $O\left(h^a\right)$. The velocity errors in $H^1(\Omega)$ show the expected first-order convergence with slightly increasing values for higher Reynolds numbers. 
Second-order convergence of velocities in $L^2(\Omega)$ is observed. The velocity errors vary barely with respect to the Reynolds numbers. The convergence in pressure is better then first-order, as typically observed for equal-order stabilized methods. It shows a surprising decrease for $\mu=0.00025$, which is probably due to the special form of the pressure. 

Further decreasing the viscosity leads to bifurcation. A typical periodic solution for $\mu=0.0001$ is shown in Figure~\ref{fig:Kovasznay1}.

%
\begin{table}[htdp]
\begin{center}
\begin{tabular}{r|r|r|r|r|r|r|}
    $N$ &     $\norm{p-p_h}$ &    order $a$ &     $\norm{\nabla(v-v_h)}$ &    order $a$ &     $\norm{v-v_h}$ &    order $a$ \\\hline
   48&6.38e-01&--&5.89e+00&--&7.76e-01&--  \\\hline
  192&1.34e-01& 2.24&4.00e+00&1.31&2.58e-01&1.58     \\\hline
  768&2.83e-02&2.24&1.93e+00&1.04&5.43e-02&2.25     \\\hline
 3072&8.15e-03&1.79&9.61e-01&1.00&1.39e-02&1.96     \\\hline
12288&2.27e-03&1.84&4.80e-01&1.00&3.59e-03&1.94     \\\hline
49152&6.09e-04&1.89&2.40e-01&1.00&9.06e-04&1.98     \\
\end{tabular}\caption{Pressure and velocity errors, $\mu=0.025$.}\label{tab:Kov1}
\end{center}
\label{ref}
\end{table}
\begin{table}[htdp]
\begin{center}
	\begin{tabular}{r|r|r|r|r|r|r|}
	    $N$ &     $\norm{p-p_h}$ &    order $a$ &     $\norm{\nabla(v-v_h)}$ &    order $a$ &     $\norm{v-v_h}$ &    order $a$ \\\hline
             48&4.13e-01&--&7.22e+00&--&5.42e-01&--\\\hline
            192&1.64e-01&1.32&6.48e+00&0.15&4.01e-01&0.43   \\\hline
            768&2.03e-02&3.06&3.21e+00&1.01&8.86e-02&2.17   \\\hline
           3072&3.67e-03&2.46&1.60e+00&1.00&2.21e-02&2.00   \\\hline
          12288&1.12e-03&1.70&8.00e-01&1.00&5.69e-03&1.95   \\\hline
          49152&4.14e-04&1.43&4.00e-01&1.00&1.47e-03&1.95   \\
\end{tabular}\caption{Pressure and velocity errors, $\mu=0.0025$}\label{tab:Kov2}
\end{center}
\label{ref}
\end{table}
\begin{table}[htdp]
\begin{center}
	\begin{tabular}{r|r|r|r|r|r|r|}
	    $N$ &     $\norm{p-p_h}$ &    order a &     $\norm{\nabla(v-v_h)}$ &    order a &     $\norm{v-v_h}$ &    order a \\\hline
             48&1.76e-01&--&9.73e+00&--&4.76e-01&--  \\\hline
            192&6.40e-02&1.45&9.22e+00&0.08&4.44e-01&0.09     \\\hline
            768&1.27e-02&2.32&4.50e+00&1.03&8.95e-02&2.31     \\\hline
           3072&2.66e-03&2.26&2.22e+00&1.02&2.07e-02&2.11     \\\hline
          12288&4.10e-04&2.69&1.11e+00&1.00&5.20e-03&1.99     \\\hline
          49152&1.92e-05&4.43&5.52e-01&1.00&1.37e-03&1.92    \\
\end{tabular}\caption{Pressure and velocity errors, $\mu=0.00025$}\label{tab:Kov3}
\end{center}
\label{ref}
\end{table}
%
%
\subsubsection{Comparison with the alternative discretization}\label{subsubsec:ExactSol1}
%
In Figures~\ref{fig:CompareStrong1} and \ref{fig:CompareStrong2} we compare the errors with the alternative discretization with strongly imposed normal velocity given in (\ref{eq:NavierStokesDiscreteAlternative}) for $\mu=0.025$ and $\mu=0.00025$, respectively. The labels 'strong' and 'weak' in the figures correspond to the implementation type of the boundary condition.
We find that the velocity errors are globally very similar and that the pressure errors are slightly smaller for the discretization proposed in this article.
\begin{figure}[h!]
\begin{center}
\subfigure[Pressure error in $L^2$-norm]{\includegraphics[width=0.32\textwidth]{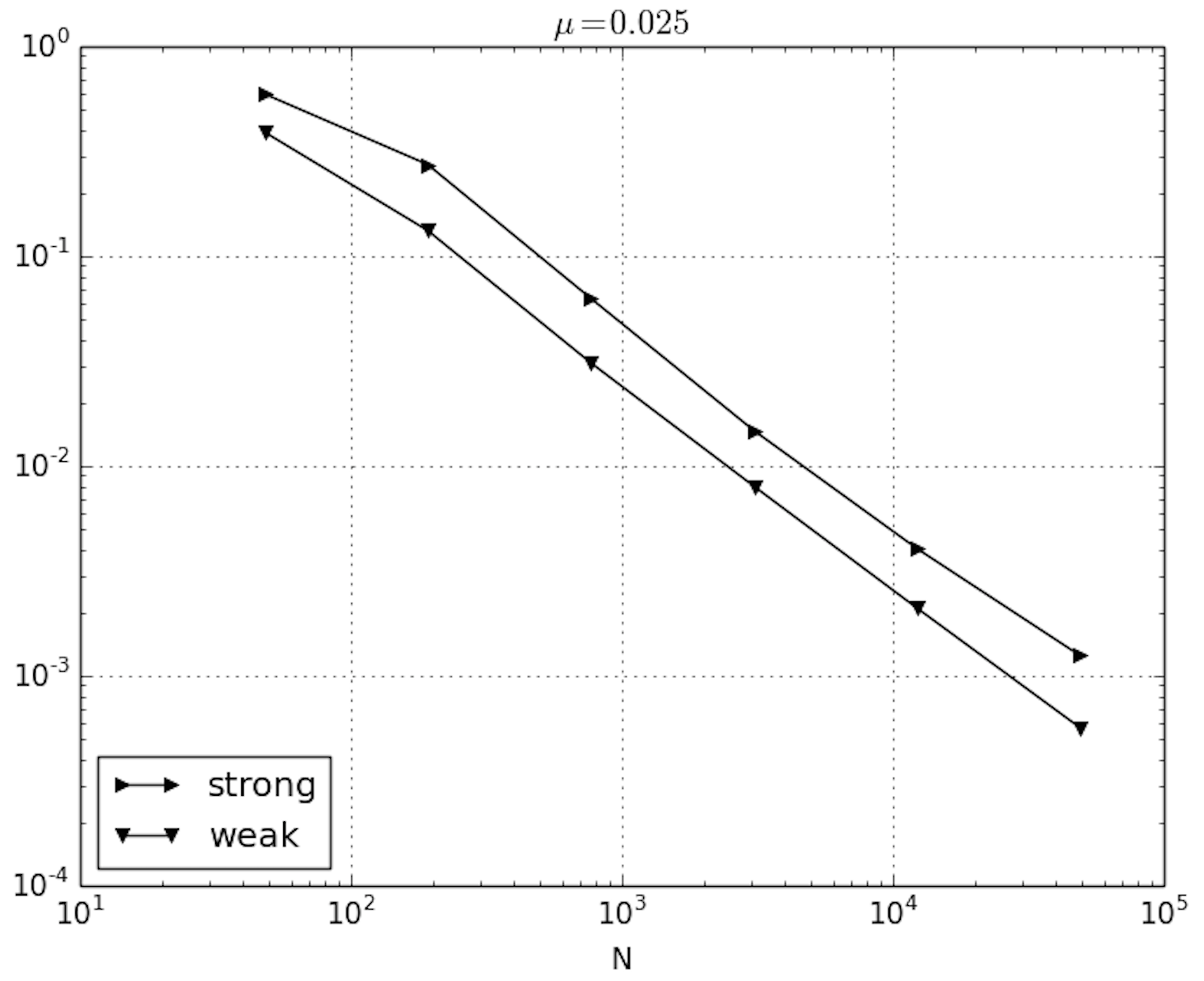}}
\subfigure[Velocity error in $H^1$-seminorm]{\includegraphics[width=0.32\textwidth]{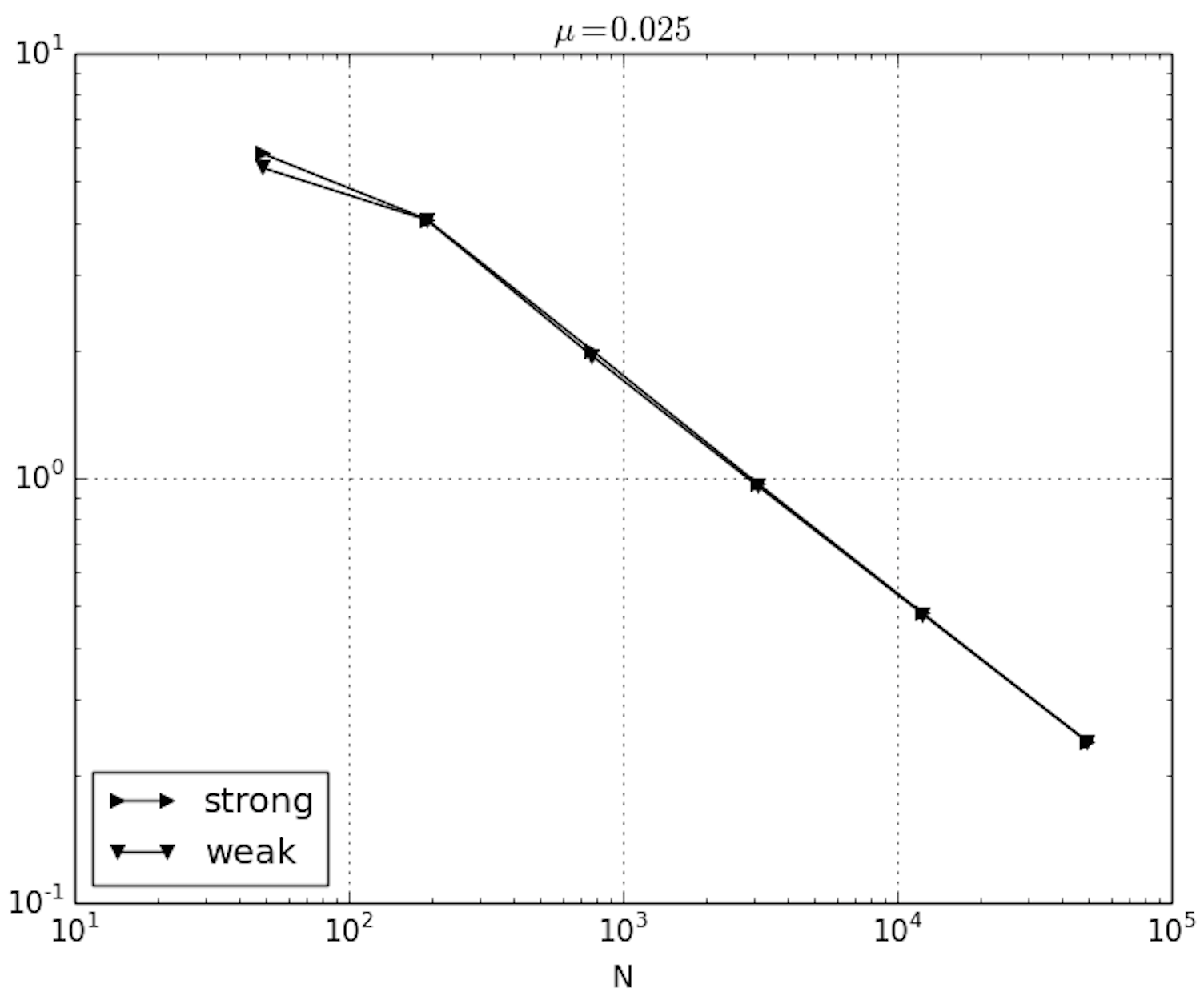}}
\subfigure[Velocity error in $L^2$-norm]{\includegraphics[width=0.32\textwidth]{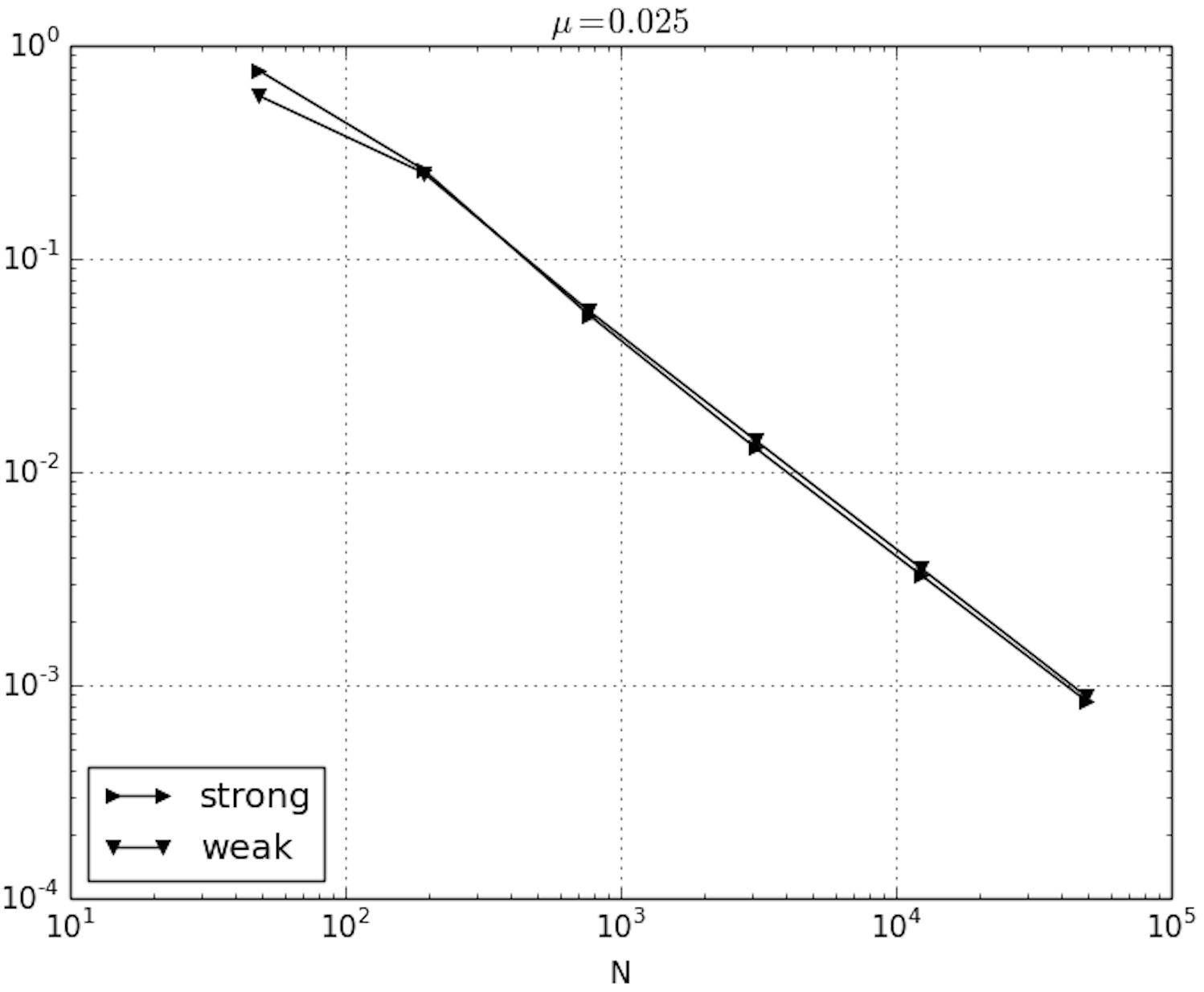}}
\caption{Comparison of discretization errors for $\mu=0.025$}\label{fig:CompareStrong1}
\end{center}
\end{figure}
\begin{figure}[h!]
\begin{center}
\subfigure[Pressure error in $L^2$-norm]{\includegraphics[width=0.32\textwidth]{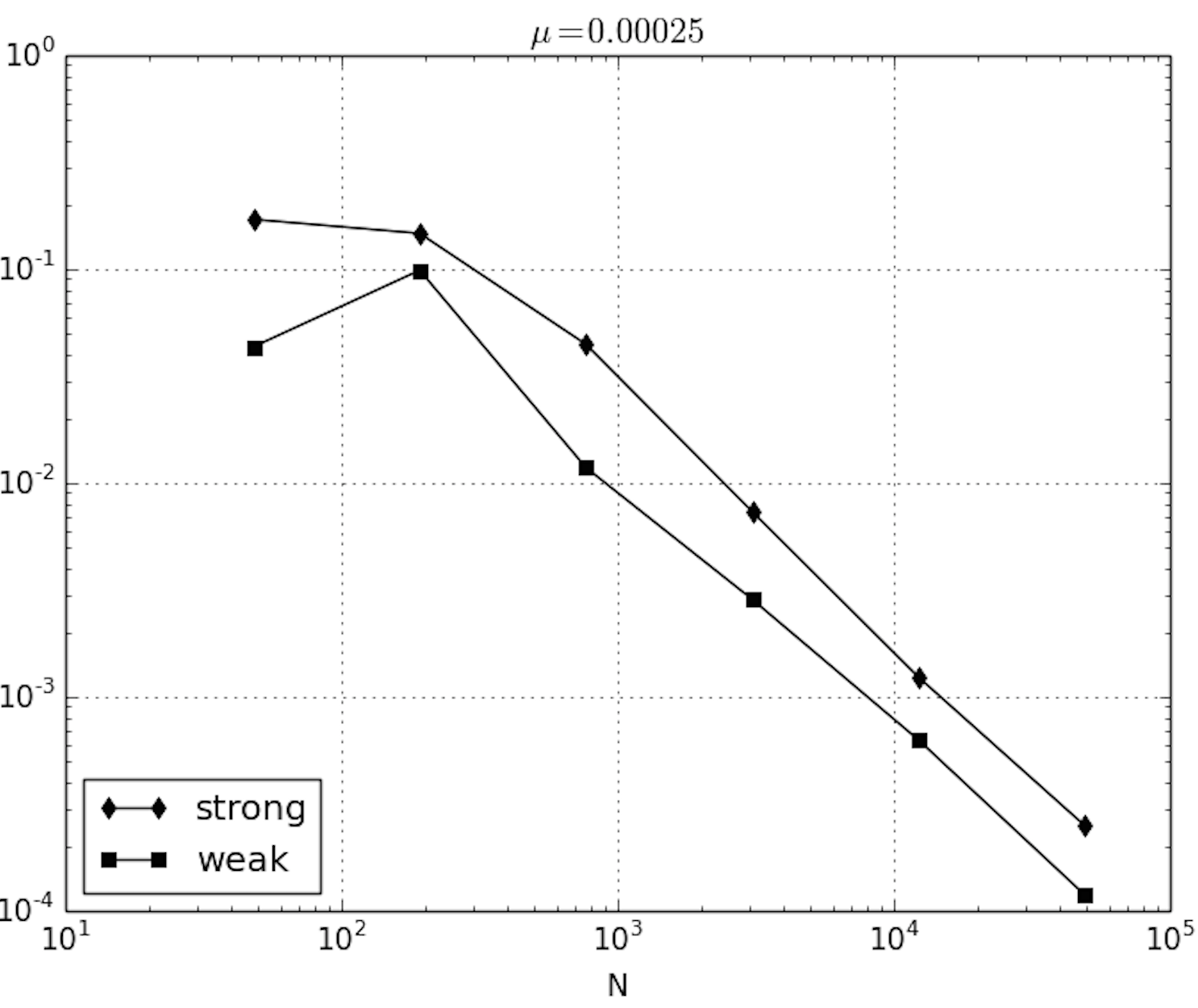}}
\subfigure[Velocity error in $H^1$-seminorm]{\includegraphics[width=0.32\textwidth]{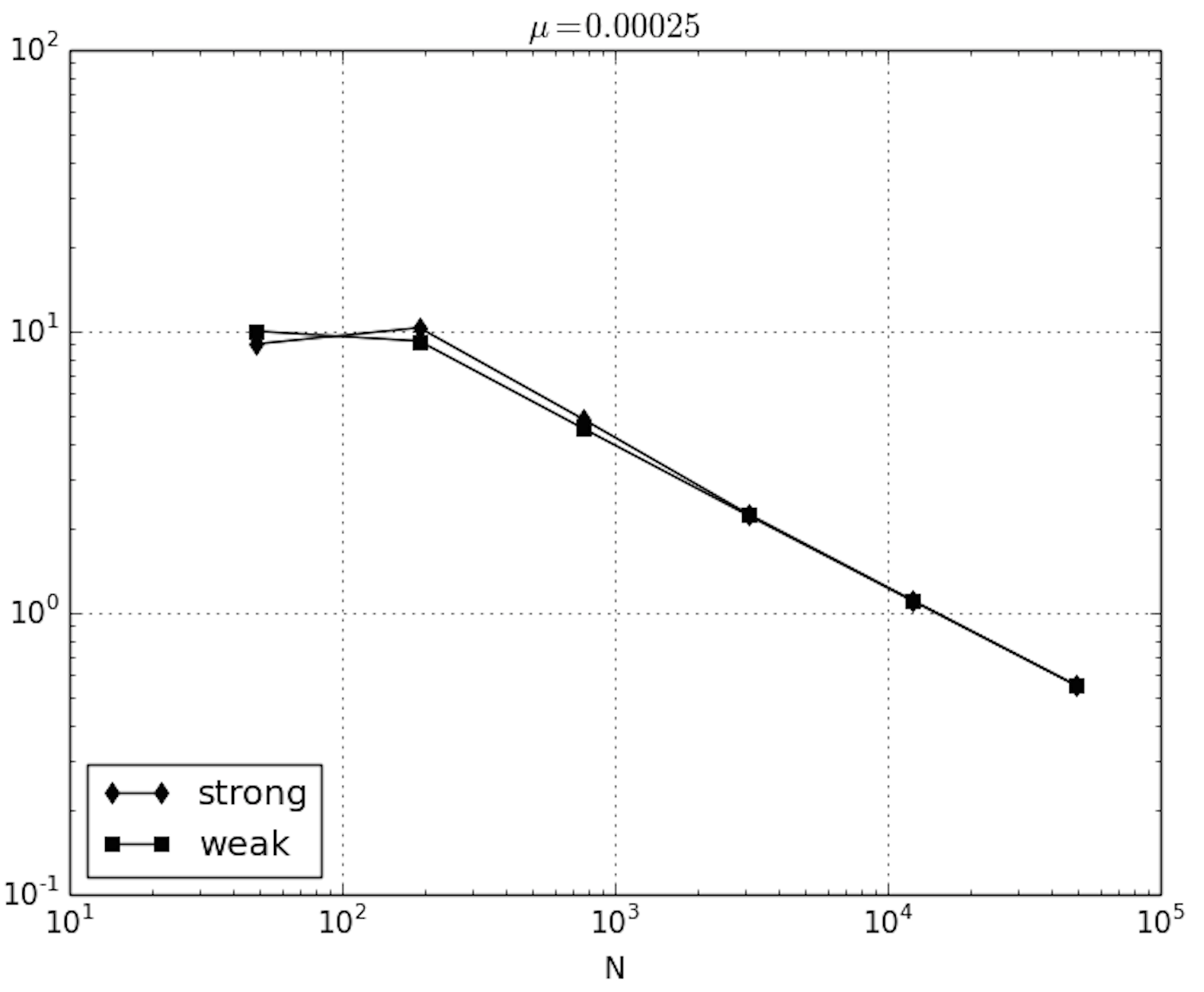}}
\subfigure[Velocity error in $L^2$-norm]{\includegraphics[width=0.32\textwidth]{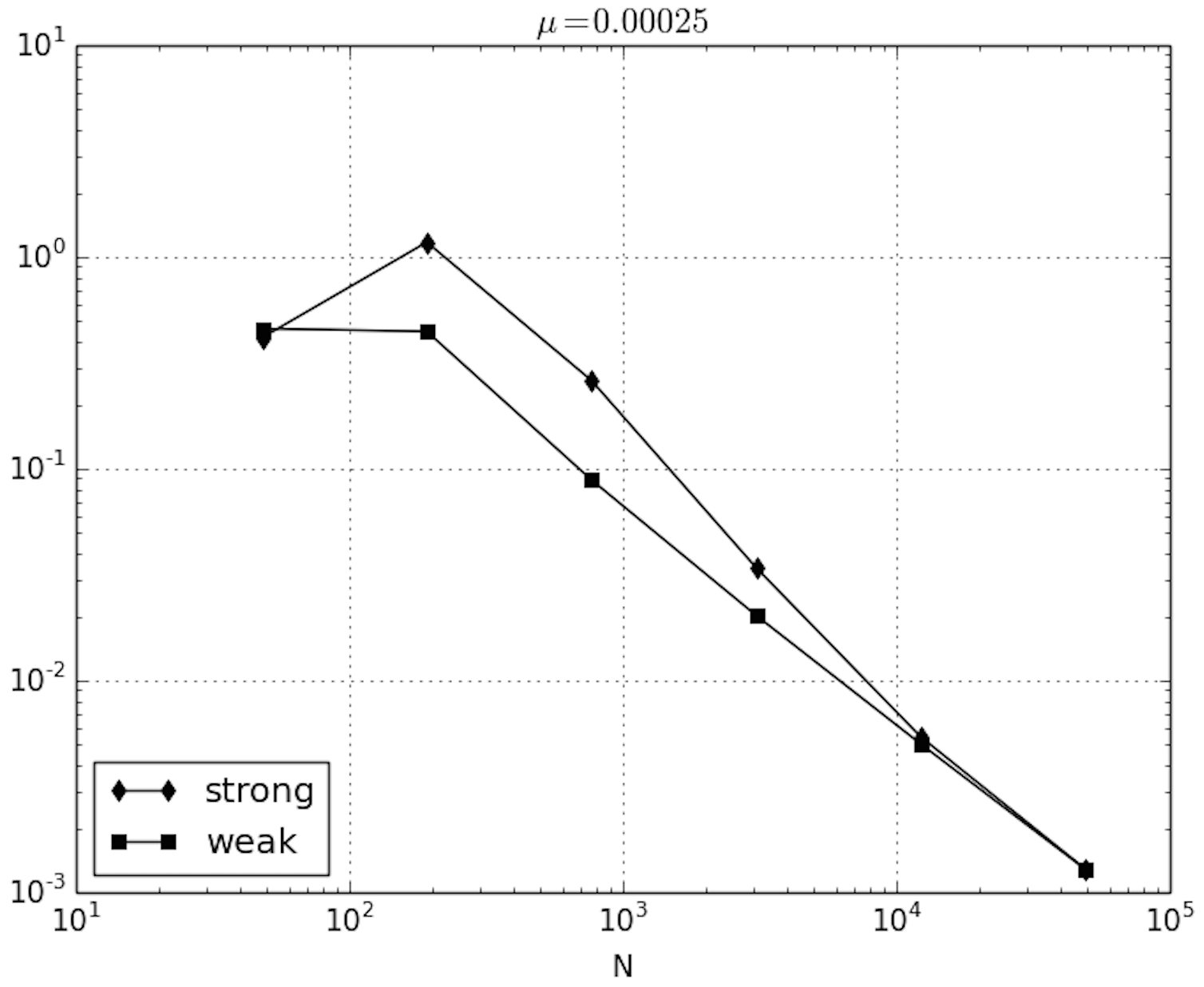}}
\caption{Comparison of discretization errors for $\mu=0.00025$}\label{fig:CompareStrong2}
\end{center}
\end{figure}
%
\subsection{Outflow boundary condition}\label{subsec:OutFlow}
%
In this subsection we investigate the behavior of the proposed outflow boundary condition. In practice, outflow boundary conditions are often used in order to limit the computational domain. The so-called 'do-nothing' condition $pn-\mu\partial_n v=0$, resulting from the  standard weak formulation of the Stokes equations with vector laplacian, is widely used in practice, since it is easy to implement and yields satisfactory results in many situations. 

In the case of a cylindrical domain cut by a plane section, it amounts to prescribing the mean pressure if no-slip boundary conditions are given on the walls of the cylinder \cite{HeywoodRannacherTurek92}. An explication for its success is the fact that the Poiseuille flow in a channel satisfies the boundary conditions, and the channel could therefore be cut at any position for this flow. Therefore, the general advise is to use a sufficiently long computational domain in order to cut the domain at a point where Poiseuille flow can be expected. However, in some applications it may be difficult to determine such a distance in advance, or even impossible due to changes in time or changes in parameters, as in optimization problems. Therefore, it is important to understand what happens if the domain is not cut at a stationary Poiseuille-type profile. 

One major difficulty related to the outflow condition is its stability. The 'do-nothing' condition has been reported in the literature to lead to instabilities, if the flow is re-entering the computational domain \cite{BraackMucha13}. As shown in Theorem~\ref{thm:EnergyNS}, we have control over the kinetic energy for our formulation.

Since our proposal leads to a modification of the boundary condition, the purpose of the following experiments is to investigate the behavior of the formulation, especially the changes induced by the additional terms.

\subsubsection{Backward facing step}\label{subsubsec:BFS}

In order to illustrate the possible instability of the 'do-nothing' boundary condition, we consider the following non stationary problem. The initial condition is defined by the stationary solution to the backward facing step problem at $Re=800$. The domain is chosen to cut the second recirculation zone, such that there is re-entrant flow. The non stationary computation is done with a ten-times smaller viscosity, such that $Re=8000$. We use the implicit Euler scheme with small time steps. It turns out that on coarse meshes, the 'do-nothing' boundary condition becomes unstable after a certain time.  As shown in Figure~\ref{fig:bfscompare}, the velocity field blows up at the re-entrant part of the outlet boundary. Consequently, for this condition the control over the kinetic energy is lost, see Figure~\ref{fig:bfscompareK}. On the contrary,  the 'energy' boundary condition allows to control the kinetic energy cf. Figure~\ref{fig:bfscompareK}, and this even for a longer-term simulation.

\begin{figure}[h!]
\begin{center}
\subfigure[$t=92s$]{\includegraphics[scale=0.2]{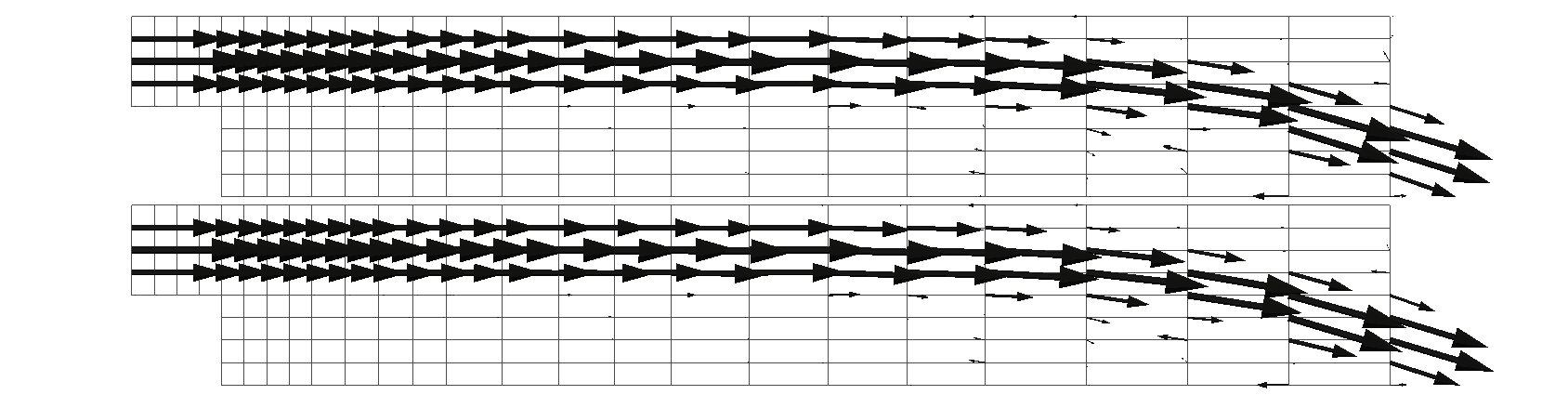}}
\subfigure[$t=115s$]{\includegraphics[scale=0.2]{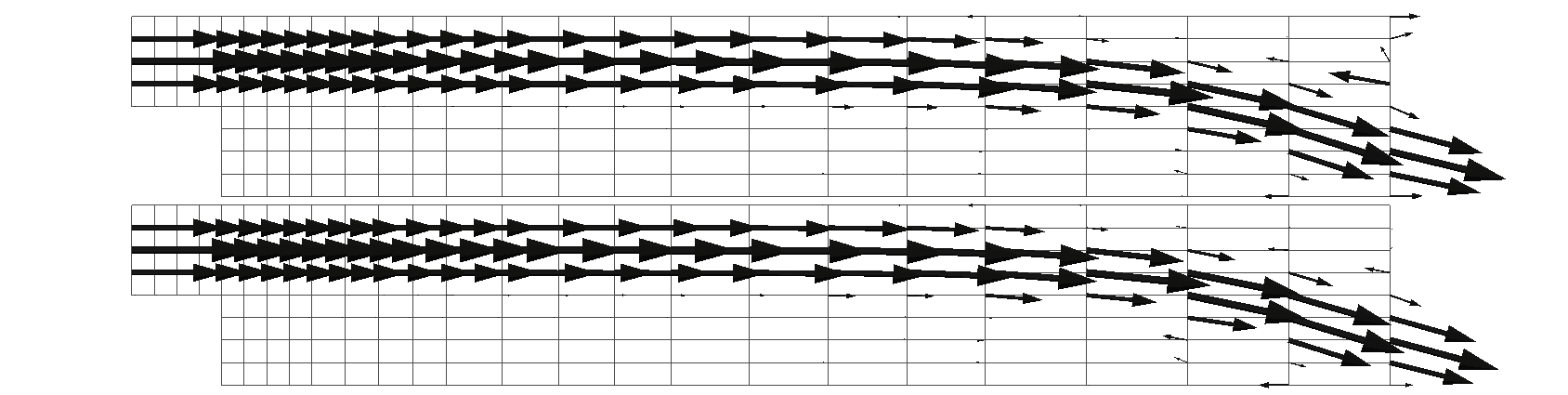}}
\subfigure[$t=125s$]{\includegraphics[scale=0.2]{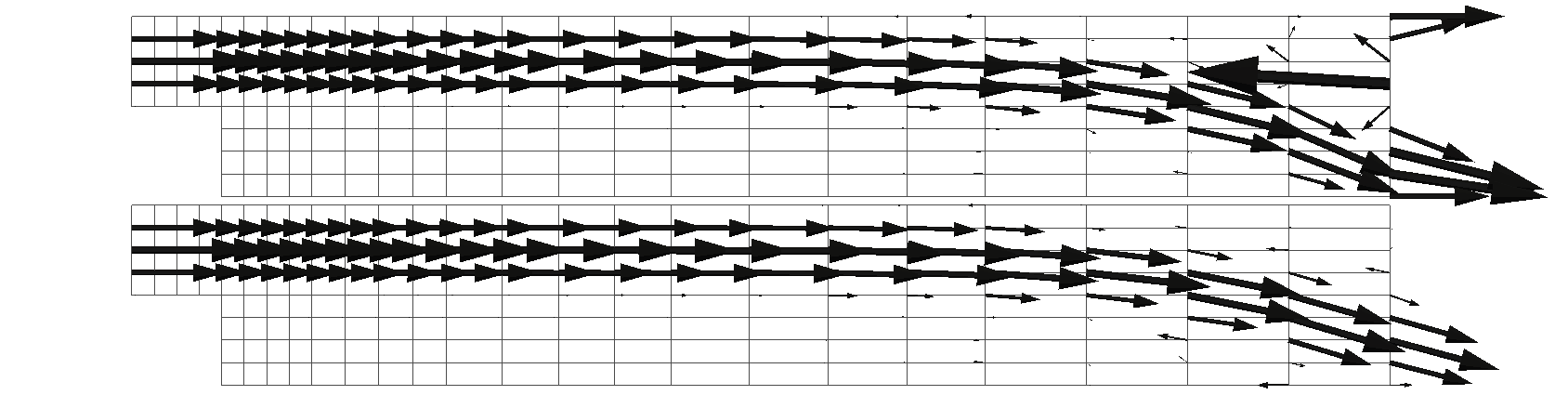}}
\caption{Comparison of velocity fields for the 'do-nothing' condition (upper image) and the 'energy' boundary condition (lower image) at three different times close to the blow-up in the first case}\label{fig:bfscompare}
\end{center}
\end{figure}

\begin{figure}[!ht]
\begin{center}
\includegraphics[scale=0.45]{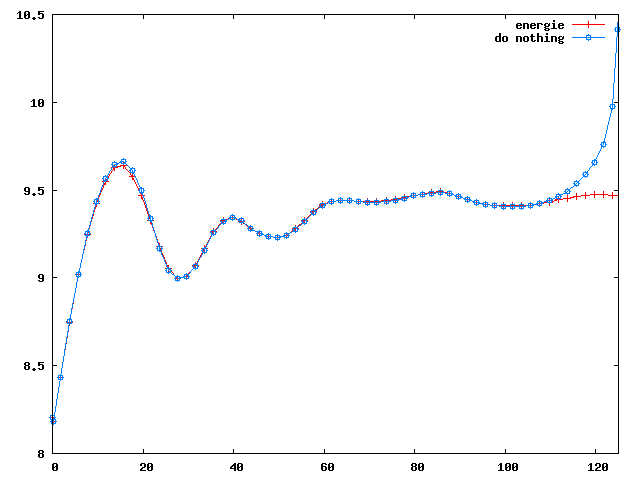}
\caption{Evolution of the kinetic energy for the 'do-nothing' and the 'energy' conditions.}\label{fig:bfscompareK}
\end{center}
\end{figure}

\subsubsection{Stationary cylinder benchmark}\label{subsubsec:Cyl}

Our next test is the computation of the drag coefficient in a stationary flow around a cylinder in a channel. The geometry, with a  slightly non-symmetric cylinder, is taken from \cite{SchaferTurek96}. However, we have chosen a two-times smaller viscosity (which yields a Reynolds number of $40$), in order to enlarge the recirculation zone behind the cylinder with respect to the original benchmark problem.

We consider four different configurations with different channel lengths, see Figure~\ref{fig:channels}. The figure also shows the streamlines of the flow, indicating the recirculation. Notice that the non-symmetry of the streamlines is due to the non-symmetry of the domain. The original length is chosen according to the 'rule of thumb'. The second computational domain is chosen in such a way that 
the recirculation zone is approximately cut in the middle. Clearly, there is an important part of the outflow boundary where the flow is entering the domain. The length of the two following domains is increased by 1 meter each. On the longer domain, we have $v_n^-=0$, such that our formulation coincides with the 'do-nothing' condition, but the velocity profile is still different from Poiseuille.

\begin{figure}[!ht]
\begin{center}
\subfigure[Original domain (22 m)]{\includegraphics[scale=0.35]{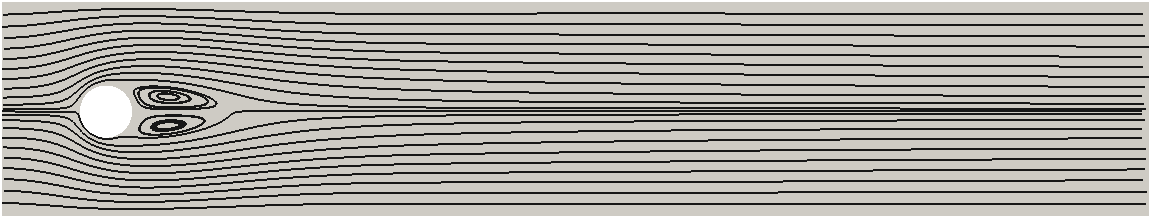}}\\
\subfigure[Length 4 m]{\includegraphics[scale=0.37]{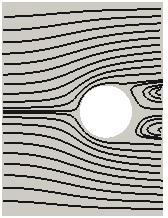}}\qquad
\subfigure[Length 5 m]{\includegraphics[scale=0.37]{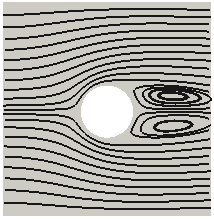}}\qquad
\subfigure[Length 6 m]{\includegraphics[scale=0.37]{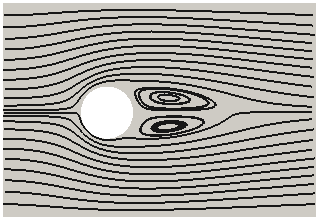}}
\caption{Tested geometries for the cylinder benchmark with computed streamlines}\label{fig:channels}
\end{center}
\end{figure}

In Figure~\ref{fig:streamlines} we compare the streamlines of the 'energy' boundary condition on the small domain with the solution on the whole domain. This zoom shows that, although the overall pictures look similar, the cut through the vertices clearly has an impact on the flow. 
Figure~\ref{fig:DoNothing}  shows a comparison between the 'energy' and the 'do nothing' conditions on the smallest domain. A closer look at the re-entrant part of the artificial boundary shows that with  the 'energy' condition, the streamlines (lines with circles) are perpendicular to the boundary as expected. This is not the case for the streamlines (lines with squares) resulting from the 'do nothing' condition.

\begin{figure}[!ht]
\begin{center}
\includegraphics[scale=0.38]{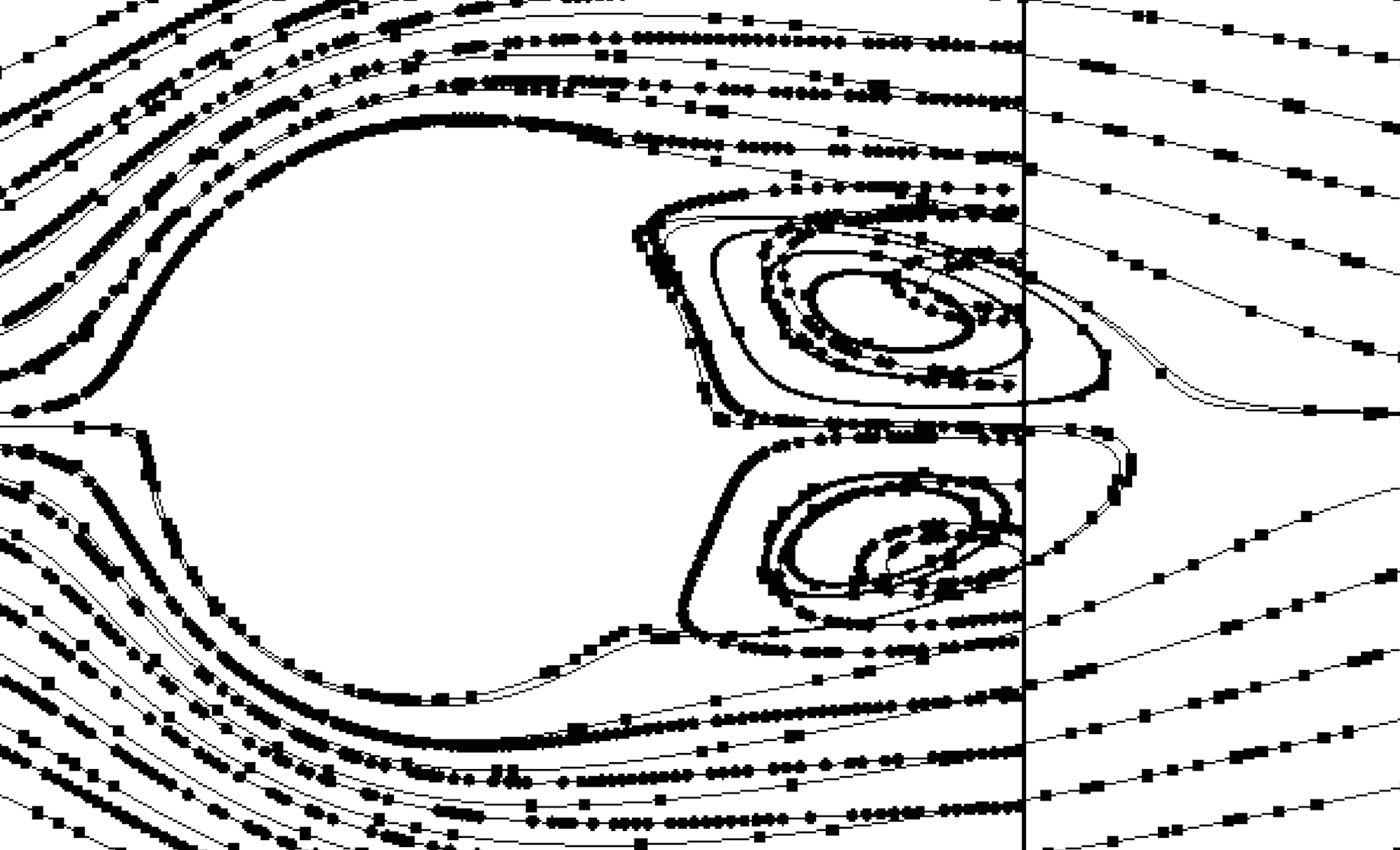}
\caption{Comparison of streamlines for 'energy' condition: smallest (\emph{circles}) and original domain (\emph{squares}).}\label{fig:streamlines}
\end{center}
\end{figure}
\begin{figure}[!ht]
\begin{center}
\subfigure[Streamlines]{\includegraphics[width=0.35\textwidth]{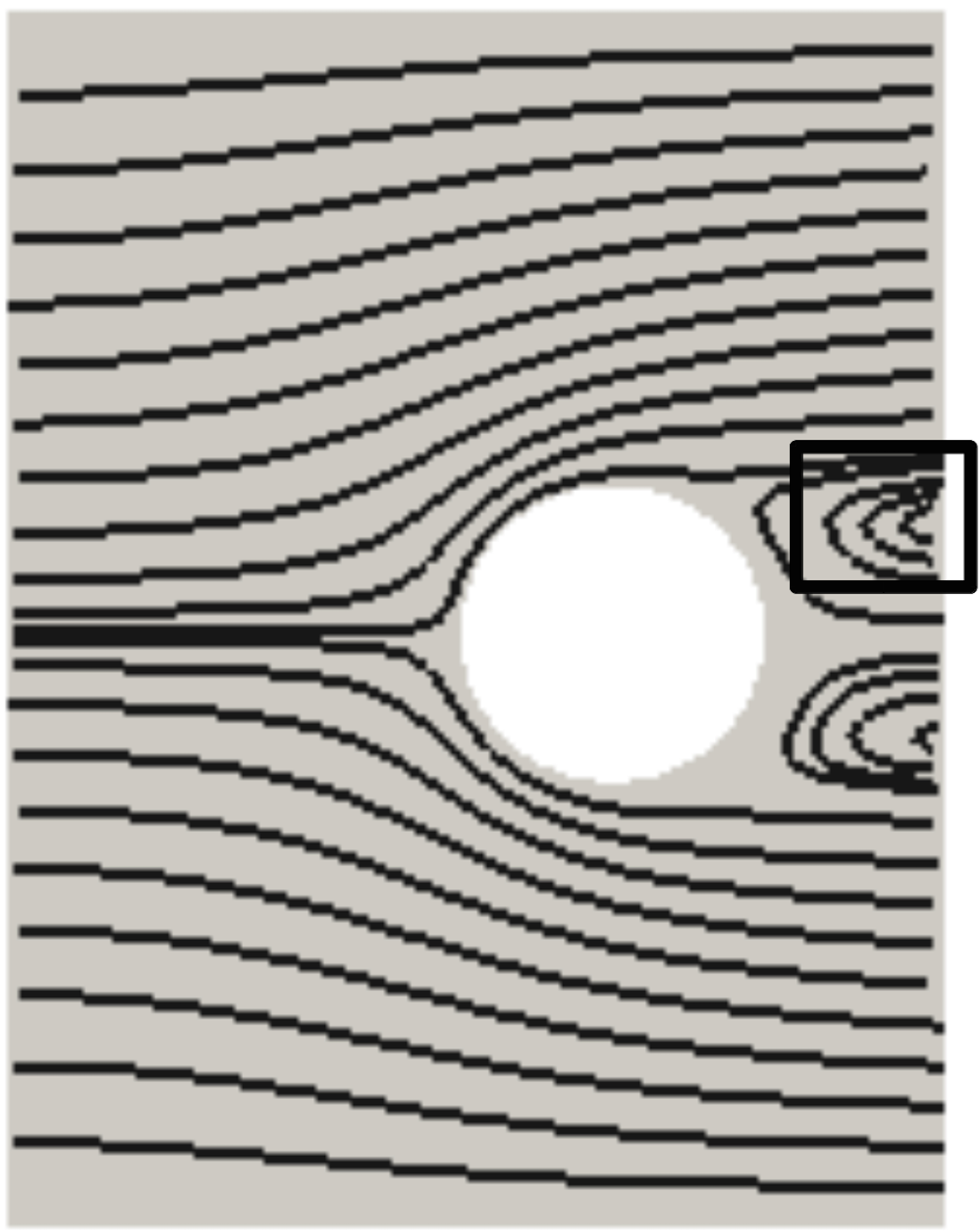}}\quad\qquad
\subfigure[Zoom]{\includegraphics[width=0.49\textwidth]{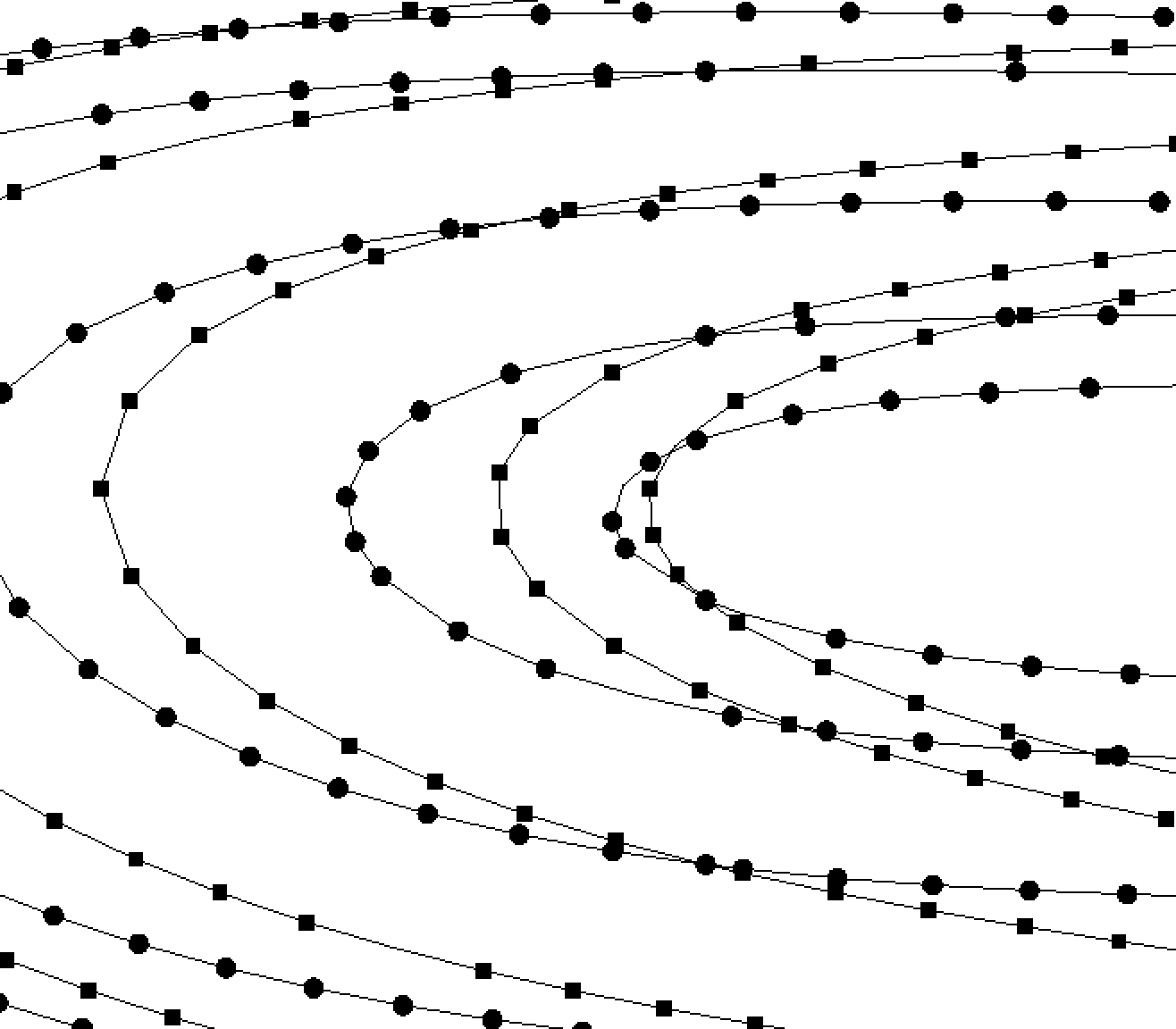}}
\caption{Comparison between 'energy' (\emph{circles}) and 'do-nothing' (\emph{squares}) conditions on the smallest domain}\label{fig:DoNothing}
\end{center}
\end{figure}
In order to evaluate the error induced by the boundary condition, we compute the drag coefficient of the cylinder.
The reference value, computed on a fine grid with the original geometry, is 4.0356. 
In Table~\ref{tab:drag}, we give the computed values on a sequence of globally refined meshes for each of the considered geometries. 
From the first column, we can see the expected second-order convergence starting from $h/4$. Already on the coarsest mesh with $160/96/80/64$ elements, the relative error is below $12 \%$ for all geometries.

Next we compare the errors due to the boundary conditions on the finest meshes. The relative error on the shortest domain is $5.8 \%$ and drops to 
$1.3 \%$ and $0.2 \%$ on the longer channels.

Finally, we compare the accuracy of our boundary condition with the 'do-nothing' condition. To this end, Table~\ref{tab:drag} contains the computed values of the drag coefficients for the latter one. It turns out that the relative error for the 'do-nothing' condition is slightly larger than ours.
\begin{table}[h!]
\begin{center}
\begin{tabular}{r|r|r|r|r||r}
	& $l=22$ & $l=6$ & $l=5$ & $l=4$ & $l=4$ 'do-nothing'\\\hline 
   $h$     &   4.492213  & 4.493186  & 4.476176 & 4.374147 & 4.370849\\
   $h/2$   &   3.996568  & 4.002665  & 3.949923 & 3.786517 & 3.777821\\
   $h/4$   &   4.021505  & 4.029468  & 3.968233 & 3.790646 & 3.778295\\
   $h/8$   &   4.032634  & 4.041017  & 3.978145 & 3.797552 & 3.784410\\
   $h/16$  &   4.034720  & 4.043192  & 3.979928 & 3.798635 & 3.785252\\
   $h/32$  &   4.035179  & 4.043660  & 3.980291 & 3.798886 & 3.785401\\\hline
 Error  & - &  \textbf{0.2 \%} &  \textbf{1.3} \% & \textbf{5.8} \%& \textbf{6.2} \%
\end{tabular}
\end{center}
\caption{Drag coefficients for different channel length $l$ and mesh resolution $h$ and relative error on finest mesh}
\label{tab:drag}
\end{table}

\subsection{Inviscid flow}\label{subsec:InviscidFlow}
%
In this subsection we consider two test problems for inviscid flows with analytical solutions. The first one investigates the numerical dissipation leading to a decrease in kinetic energy. The second one is a well-known analytical solution of the Euler equations with vortices. 
%
\subsubsection{Standing vortex}\label{subsubsec:}
%
The computational domain is $\Omega=]-1,+1[^2$ and the stationary solution is given in polar coordinates by $v_{r}=0$ and $v_{\theta}=0$ for $r>0.8$, $v_{\theta}=2-2.5r$ for $0.4\le r\le 0.8$ and $v_{\theta}=2.5r$ for $r<0.4$. Starting with the nodal interpolation of this velocity field as initial condition (see Figure~\ref{fig:StandingVortex1}(a)), we solve the discrete equations up to $t = 5s$ on a sequence of uniform meshes with number of nodes $N=256,1024,4096,16384$ (denoted by mesh1,..., mesh4 in the following). 
\begin{figure}[!ht]
\begin{center}
\subfigure[$t=0s$]{\includegraphics[scale=0.27]{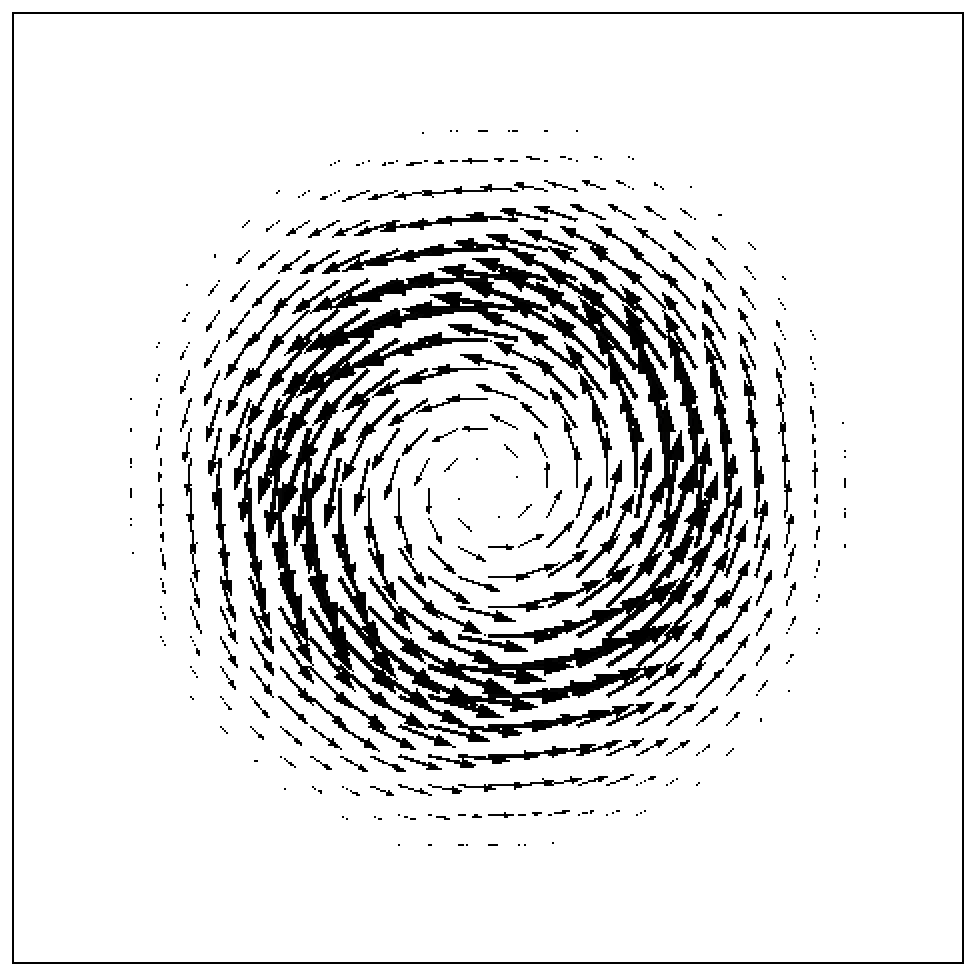}}\quad
\subfigure[$t=5s$]{\includegraphics[scale=0.27]{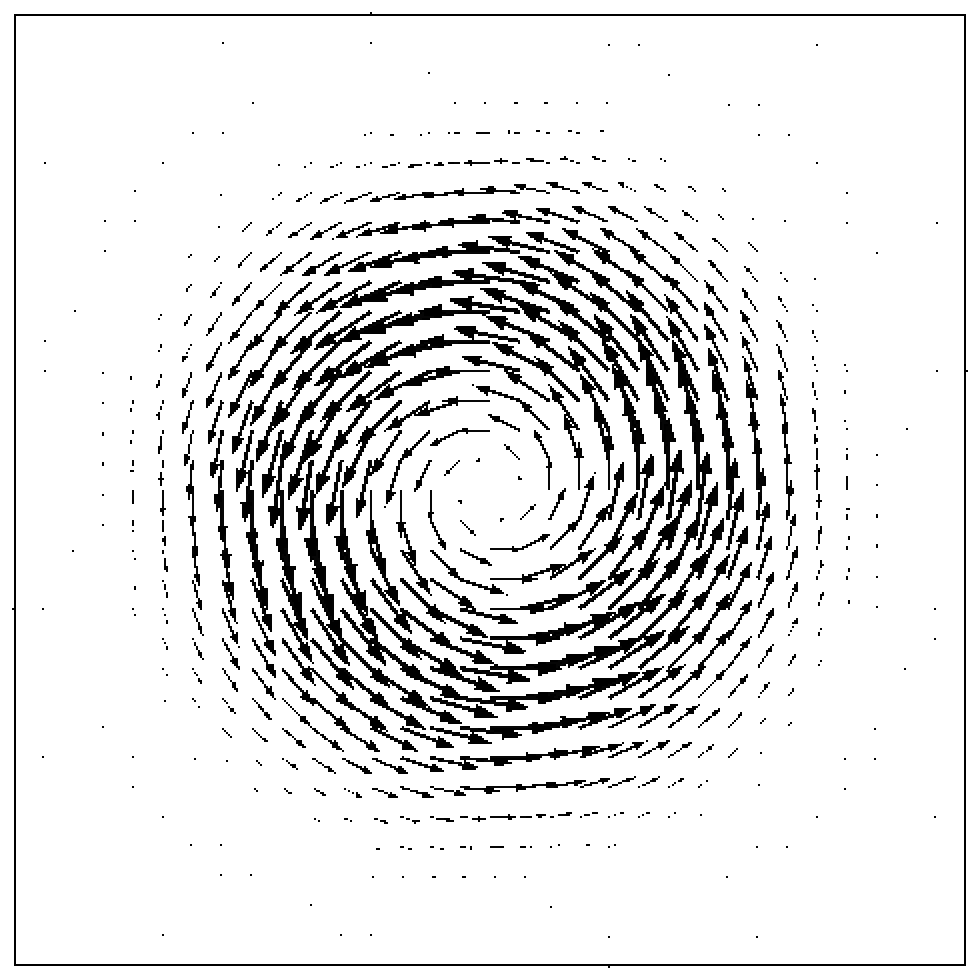}}
\caption{Standing vortex problem on a uniform mesh with $N=1024$ nodes: velocity}\label{fig:StandingVortex1}
\end{center}
\end{figure}
The numerical dissipation due to the boundary and interior stabilization terms leads to the creation of artificial vortices with very low energy. Since they cannot be detected by the vector fields, we use streamlines to visualize this effect in Figure~\ref{fig:StandingVortexSL}.
\begin{figure}[!ht]
\begin{center}
\subfigure[$t=0s$]{\includegraphics[scale=0.23]{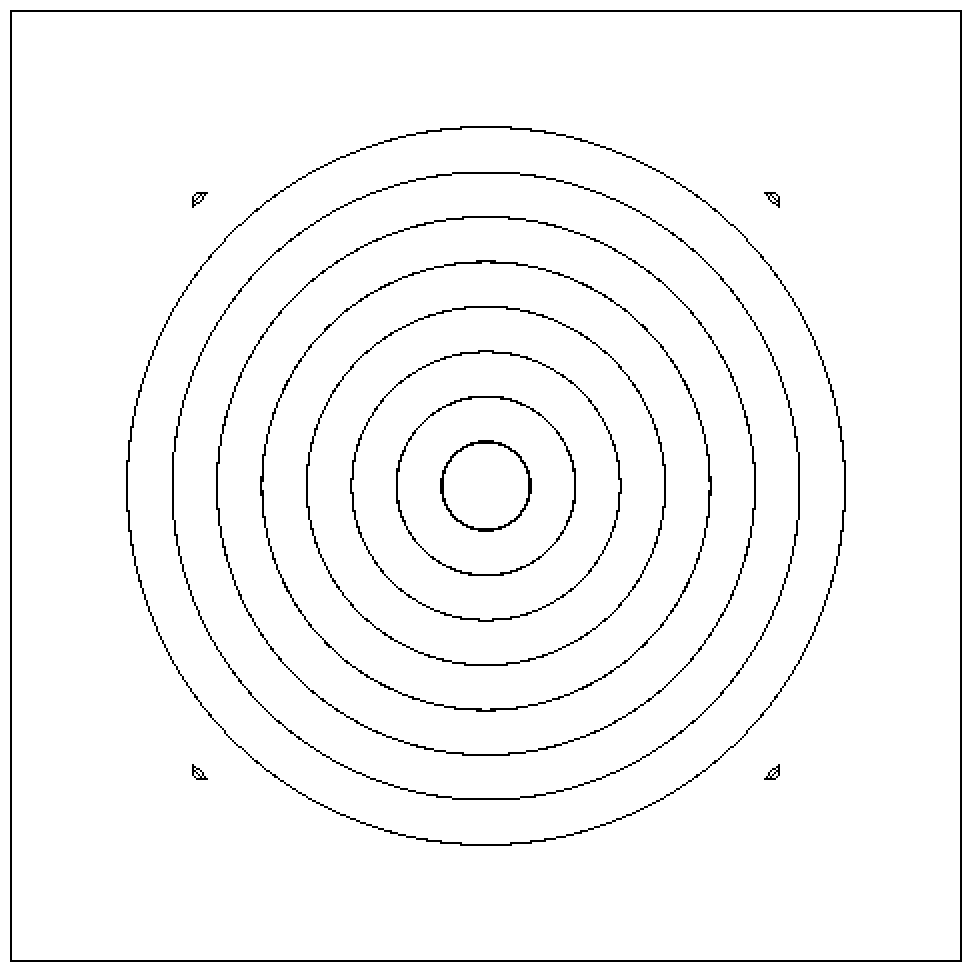}}\,
\subfigure[$t=2.5s$]{\includegraphics[scale=0.23]{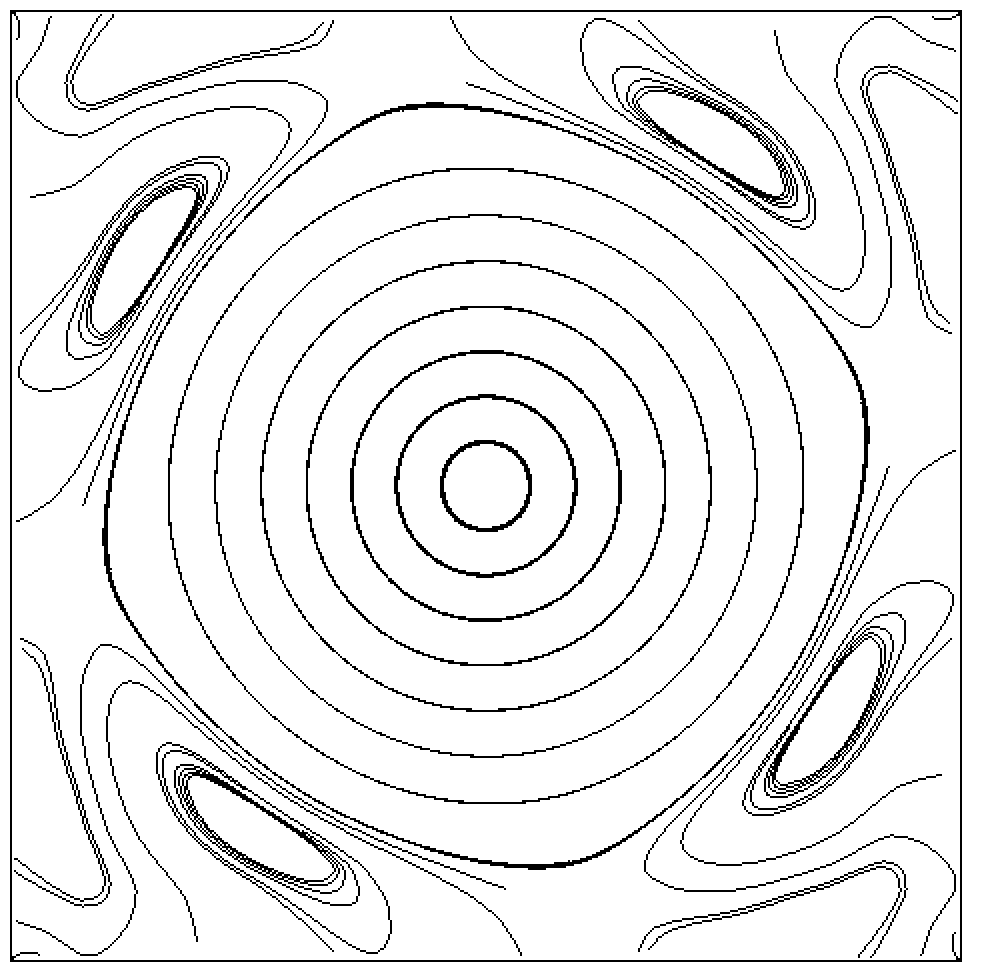}}\,
\subfigure[$t=5s$]{\includegraphics[scale=0.23]{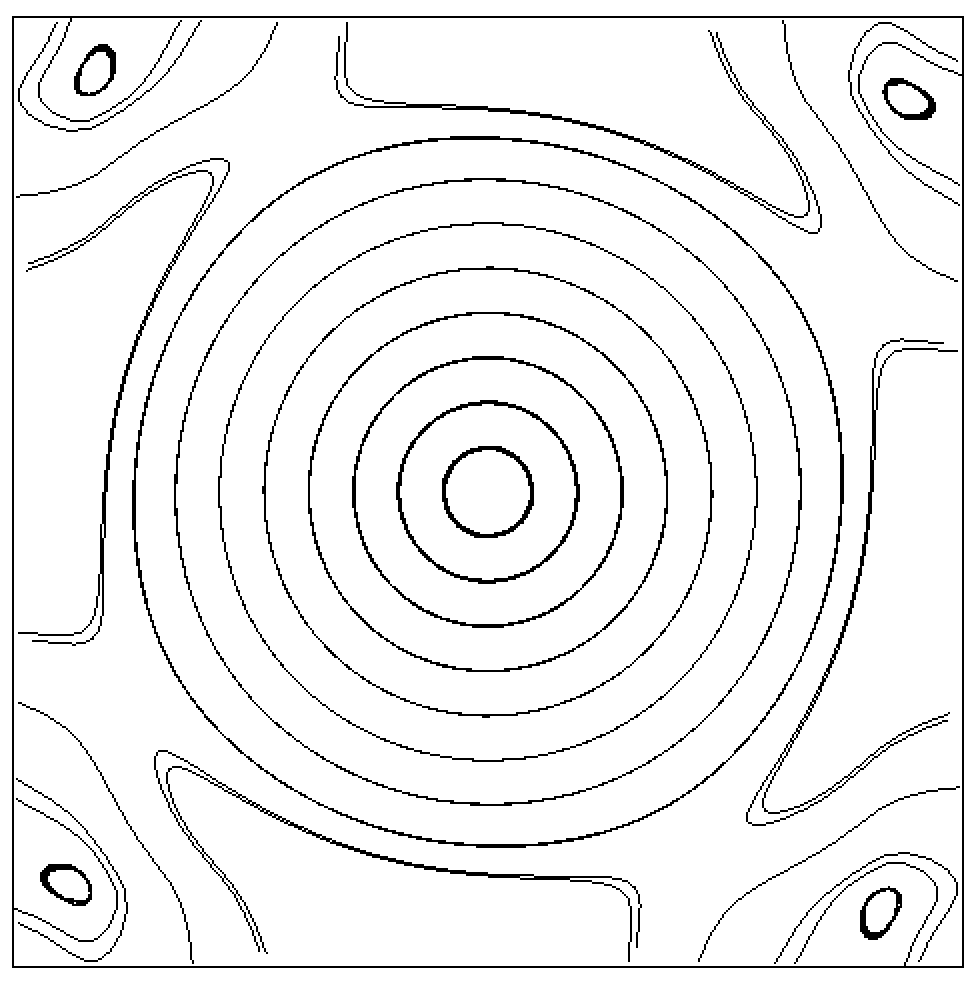}}
\caption{Standing vortex problem on a uniform mesh with $N=1024$ nodes: streamlines}\label{fig:StandingVortexSL}
\end{center}
\end{figure}
The temporal behavior of the kinetic energy, as well as the numerical dissipation due to the boundary stabilization and interior stabilization terms are shown in Figure~\ref{fig:StandingVortexTime}. 
The loss in kinetic energy on the four meshes is: $42\%$, $11\%$, $2\%$, $0.4\%$. This shows an empirical convergence order between two and three.

From Figure~\ref{fig:StandingVortexTime}(c) it becomes clear that the SUPG-stabilization is dominant. The ratio between boundary and interior stabilization at $t=5s$ for the sequence of meshes is: $25$, $221$, $876$, $880$. 
\begin{figure}[!ht]
\begin{center}
\subfigure[Kinetic energy.]{\includegraphics[scale=0.37]{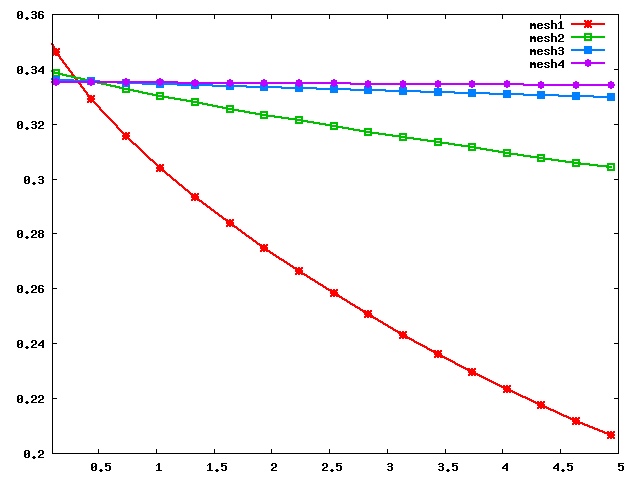}}
\subfigure[Numerical dissipation (boundary).]{\includegraphics[scale=0.37]{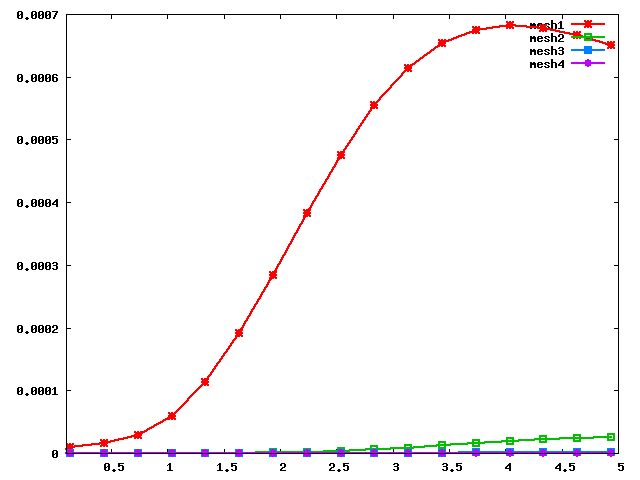}}
\subfigure[Numerical dissipation (stabilization).]{\includegraphics[scale=0.37]{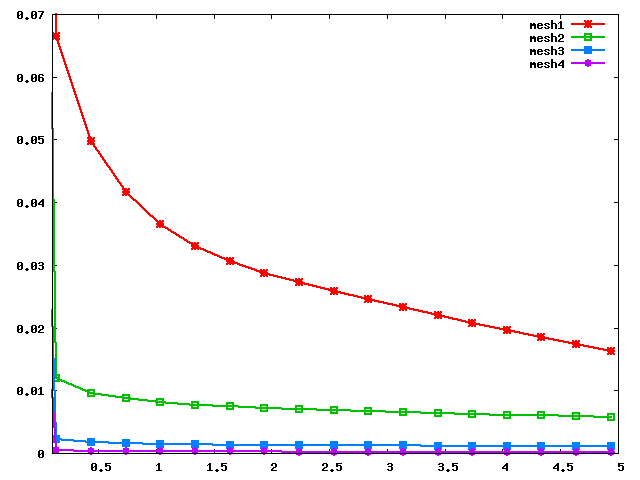}}
\caption{Standing vortex problem on a sequence of meshes.}\label{fig:StandingVortexTime}
\end{center}
\end{figure}

%
\subsubsection{Fraenkel flow}\label{subsubsec:}
%
%
\begin{figure}[!ht]
\begin{center}
\subfigure[Image from \cite{Fraenkel61}]{\includegraphics[scale=0.22,angle=-1]{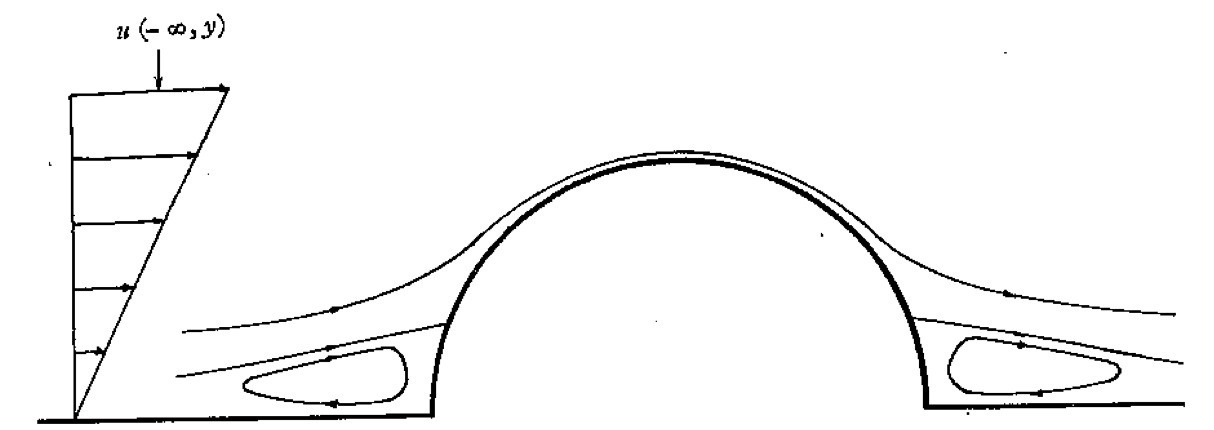}}
\subfigure[$t\approx0.01s$]{\includegraphics[scale=0.12,clip=true, trim=0 21mm 0 0]{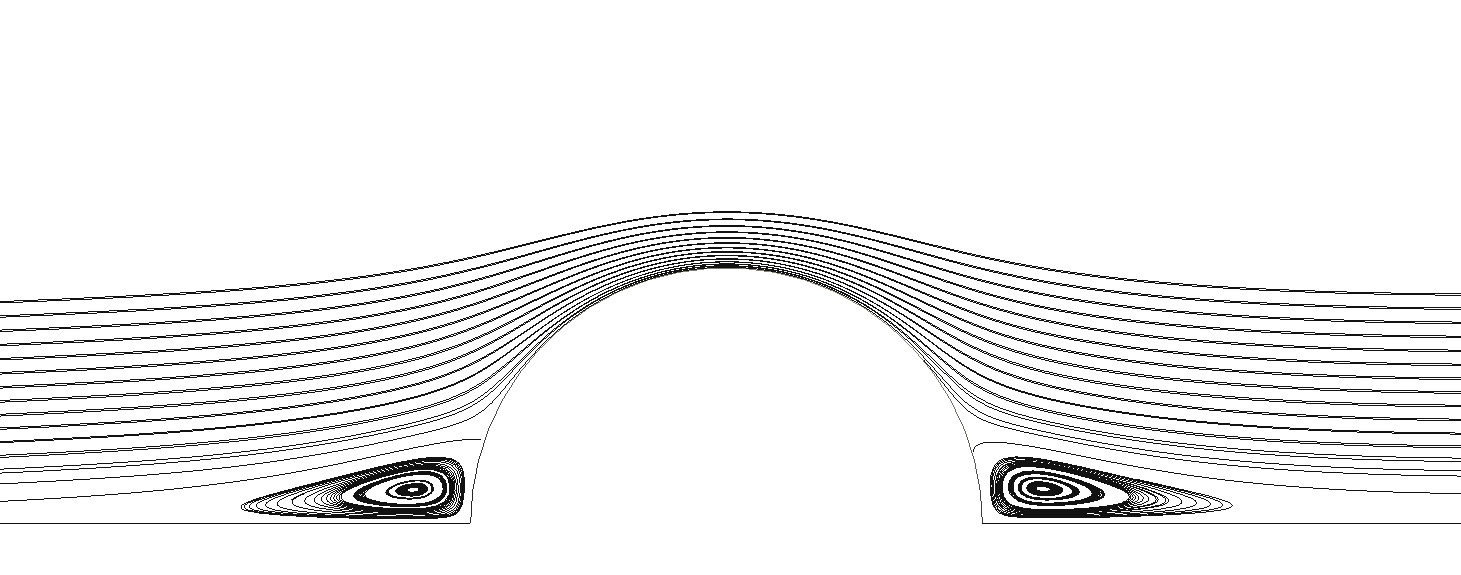}}
\caption{Comparison between Fraenkel's analytical solution and computation.}\label{fig:Fraenkel}
\end{center}
\end{figure}
As another example for inviscid flow computation, we consider the analytical stationary solution of a rotational flow given by Fraenkel in \cite{Fraenkel61}. This solution with a linear far field velocity profile satisfies the wall condition on the lower part of the boundary and presents two symmetric vortices before and behind the half-cylinder, see Figure~\ref{fig:Fraenkel}. It has for example been used as a test-case in \cite{FeistauerKuvcera07}.
\begin{figure}[!ht]
\begin{center}
\subfigure[$t\approx0.01s$]{\includegraphics[scale=0.12]{img/Fraenkel/fraenkel-001}}\,\,
\subfigure[$t\approx1s$]{\includegraphics[scale=0.12]{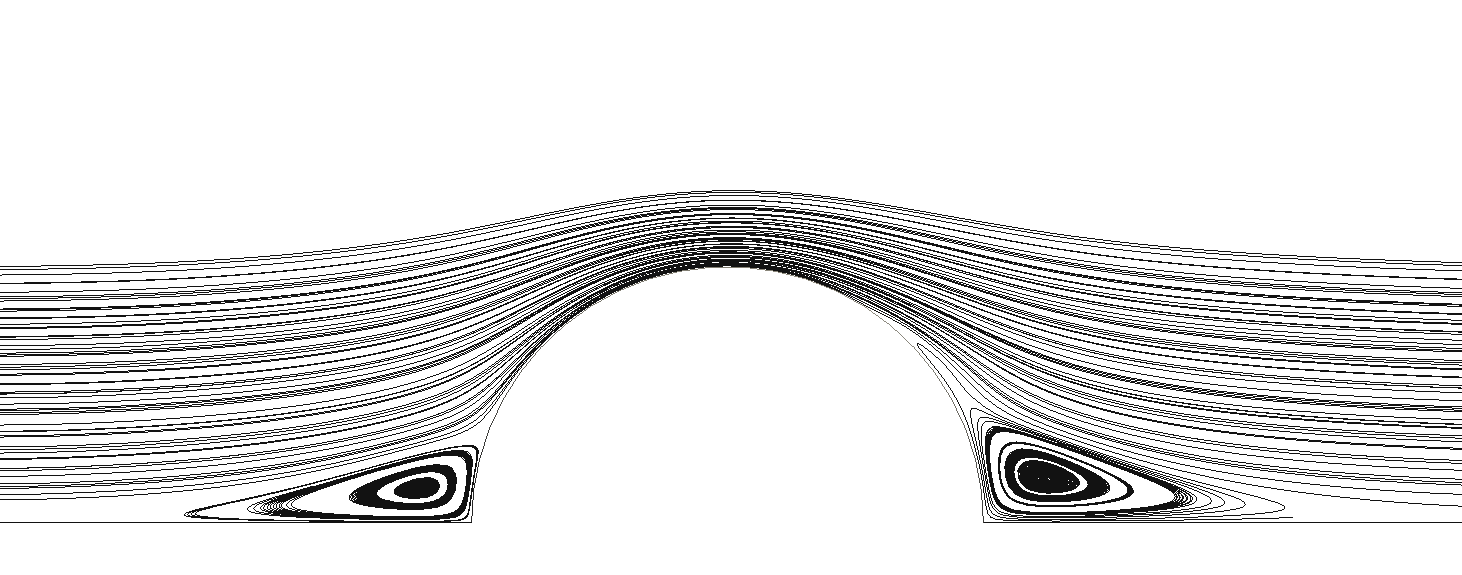}}
\subfigure[$t\approx2s$]{\includegraphics[scale=0.12]{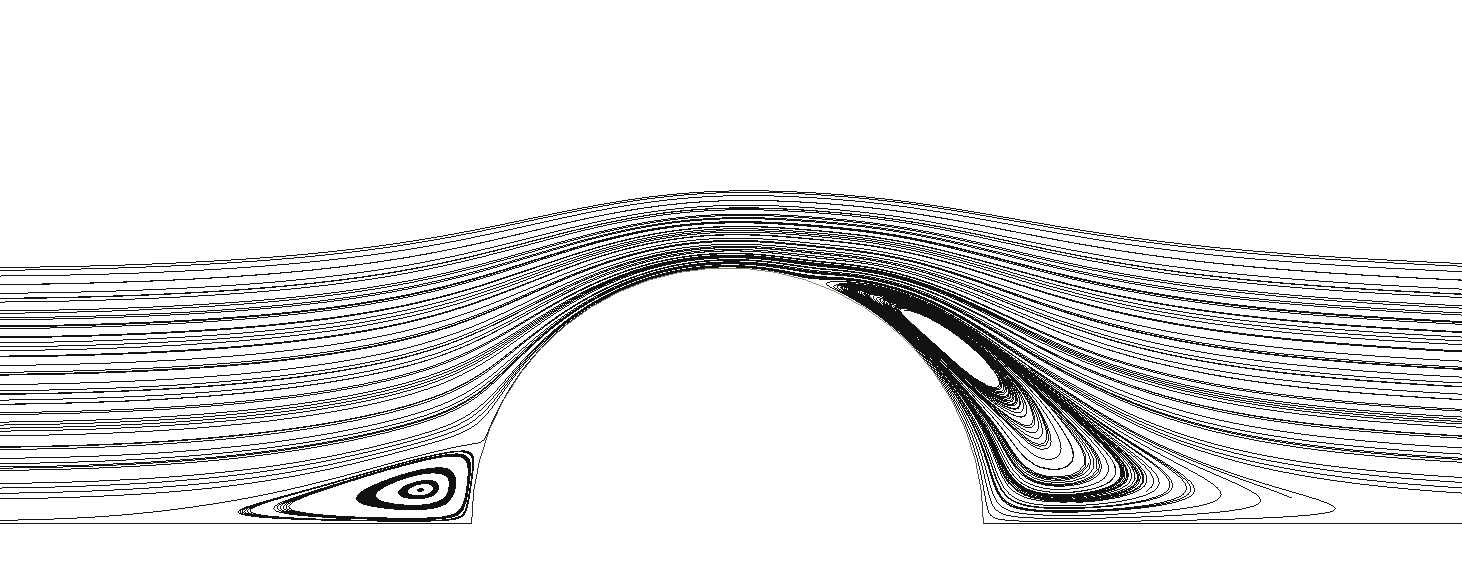}}\,\,
\subfigure[$t\approx4s$]{\includegraphics[scale=0.12]{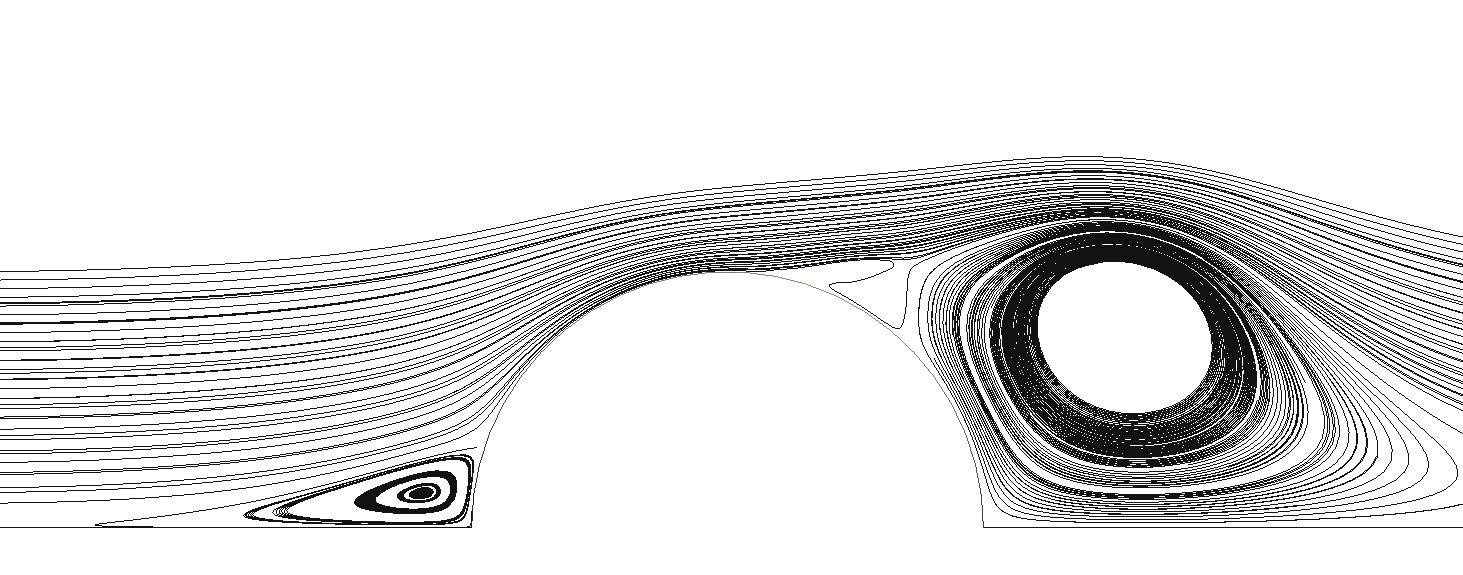}}
\caption{Instability of Fraenkel's analytical solution for inviscid flow: streamlines}\label{fig:FraenkelTime}
\end{center}
\end{figure}
We use the computational domain $\Omega=]-3,3[\times]0,3[ \setminus\{x_1^2+x_2^2\le 1\}$. The inflow condition with $\vD=(x_2,0)$ is given at $x_1=-3$, the outflow condition is used at $x_1=3$ and the remaining parts of the boundary are treated as solid walls. Figure~\ref{fig:FraenkelTime} shows some typical numerical results. 
Starting from rest, we retrieve the analytical solution after a short time $t\approx 0.01s$. 
However, starting at $t\approx1s$ a detachment appears behind the cylinder and the solution loses its symmetry and develops other vortices, as one can see in Figure \ref{fig:FraenkelTime} (c) and (d). We have conducted several numerical experiments using different mesh resolutions, larger computational domains, and different time schemes. The results were always quite similar to those reported here. We conclude that the analytical solution is not stable.
As a further illustration, we consider the computation of the flow around a cylinder by reflection of the domain $\Omega$. Noting that the analytical solution can be prolongated by reflection, we impose the inflow condition $\vD=(|x_2|,0)$. As shown in Figure~\ref{fig:FraenkelSymTime}, the flow loses the symmetry in $x_1$ at about $t\approx 1s$ as before. The symmetry in $x_2$ is broken at $t\approx100s$.
\begin{figure}[!ht]
\begin{center}
\subfigure[$t\approx0.01s$]{\includegraphics[scale=0.12]{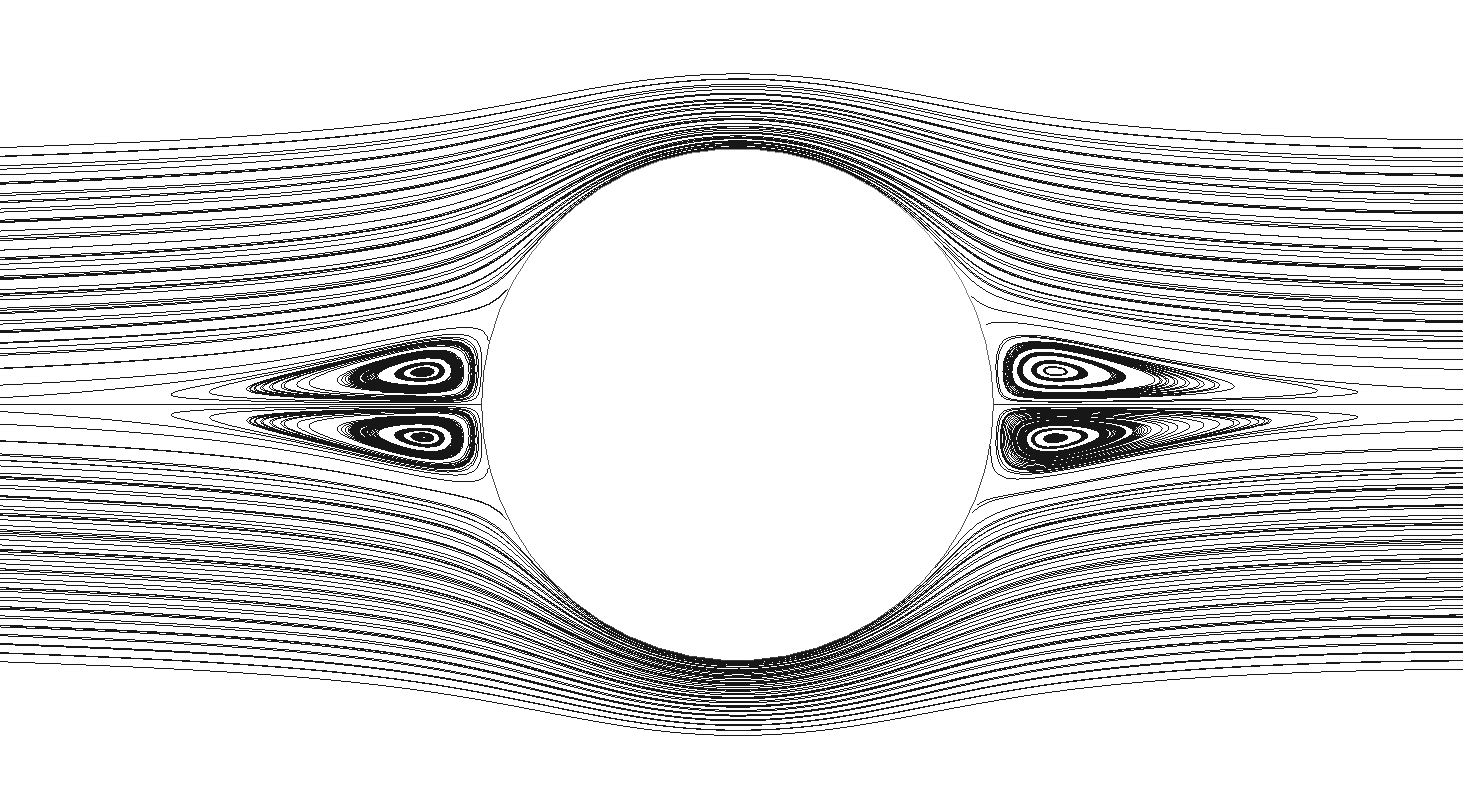}}\,\,
\subfigure[$t\approx1s$]{\includegraphics[scale=0.12]{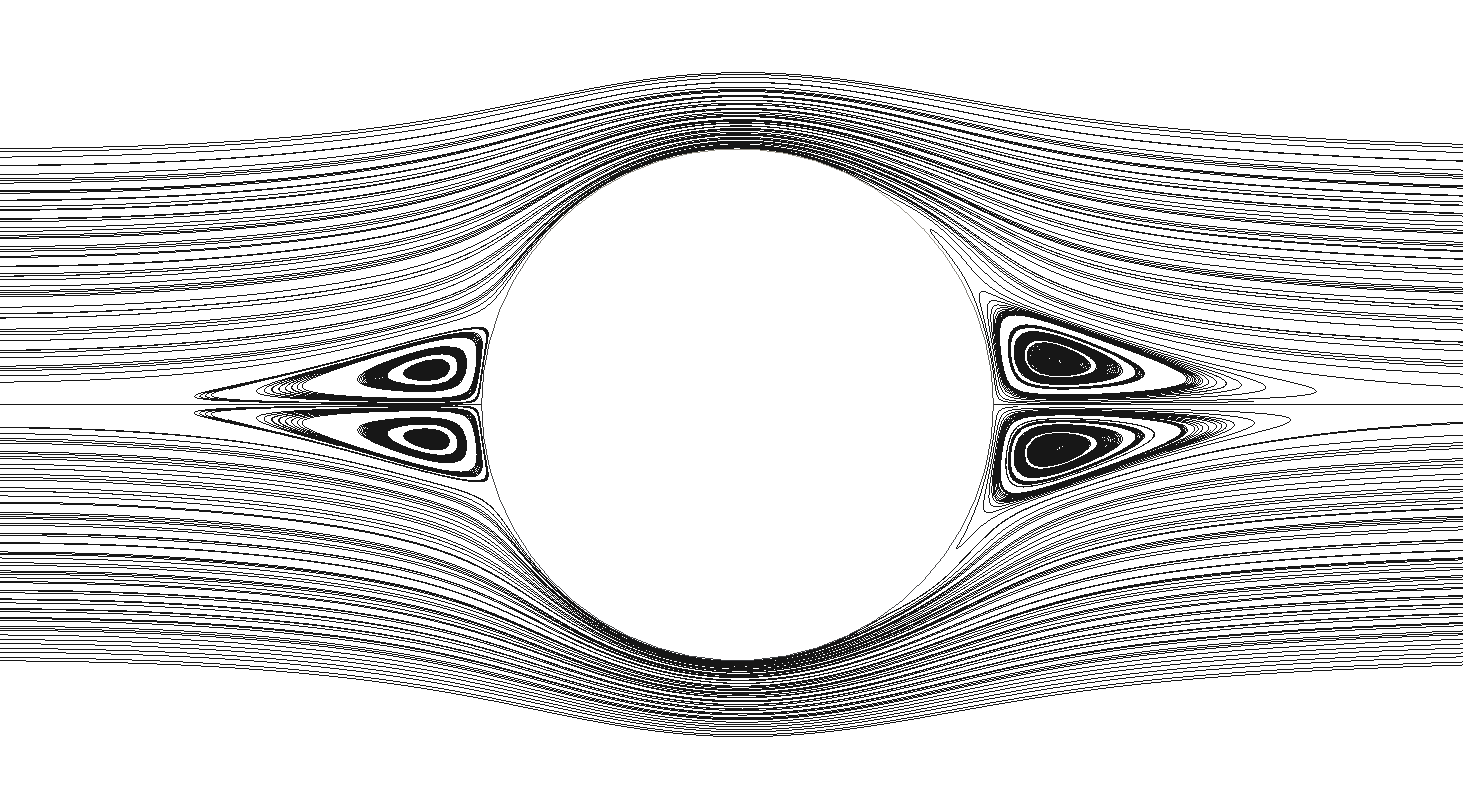}}
\subfigure[$t\approx2s$]{\includegraphics[scale=0.12]{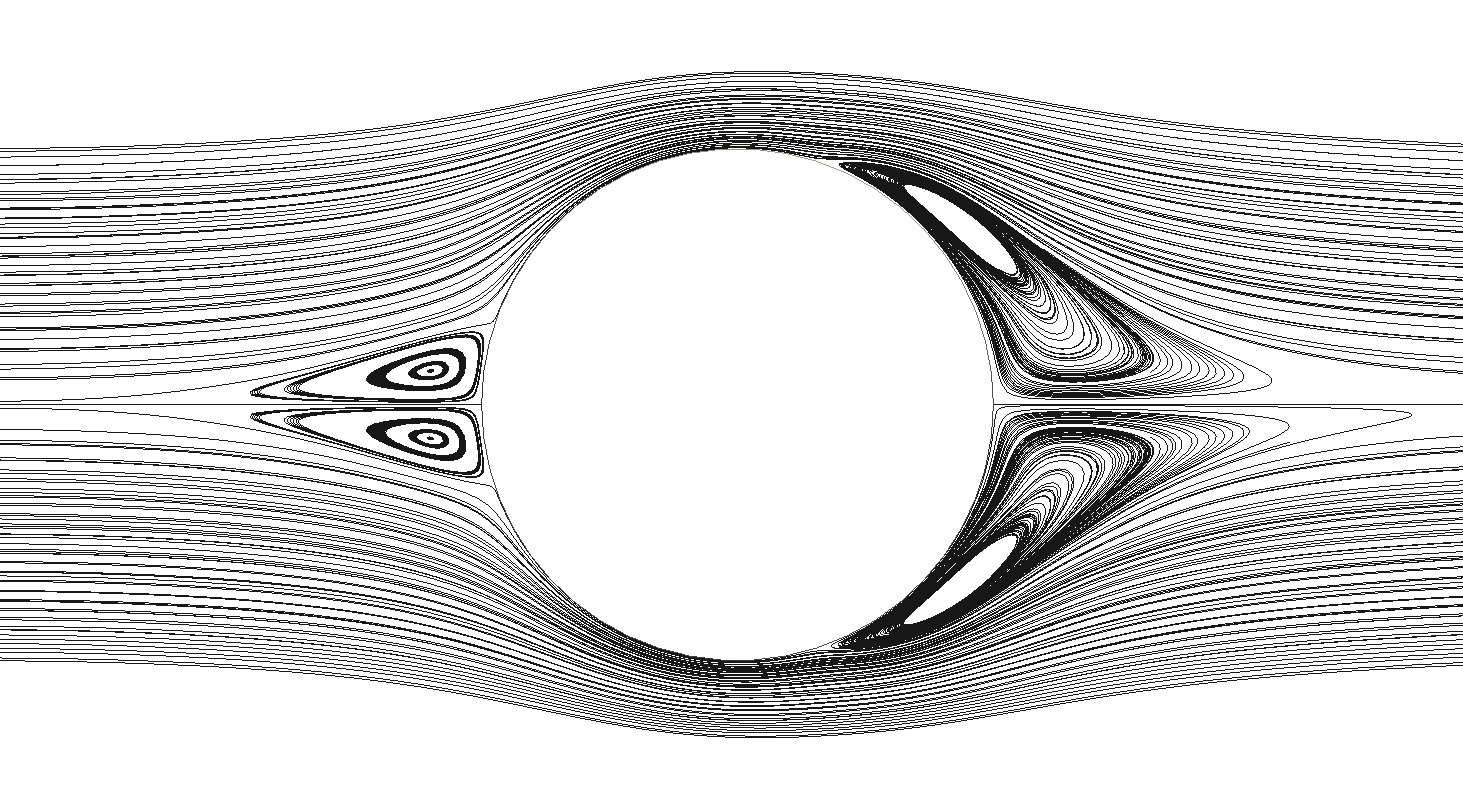}}\,\,
\subfigure[$t\approx4s$]{\includegraphics[scale=0.12]{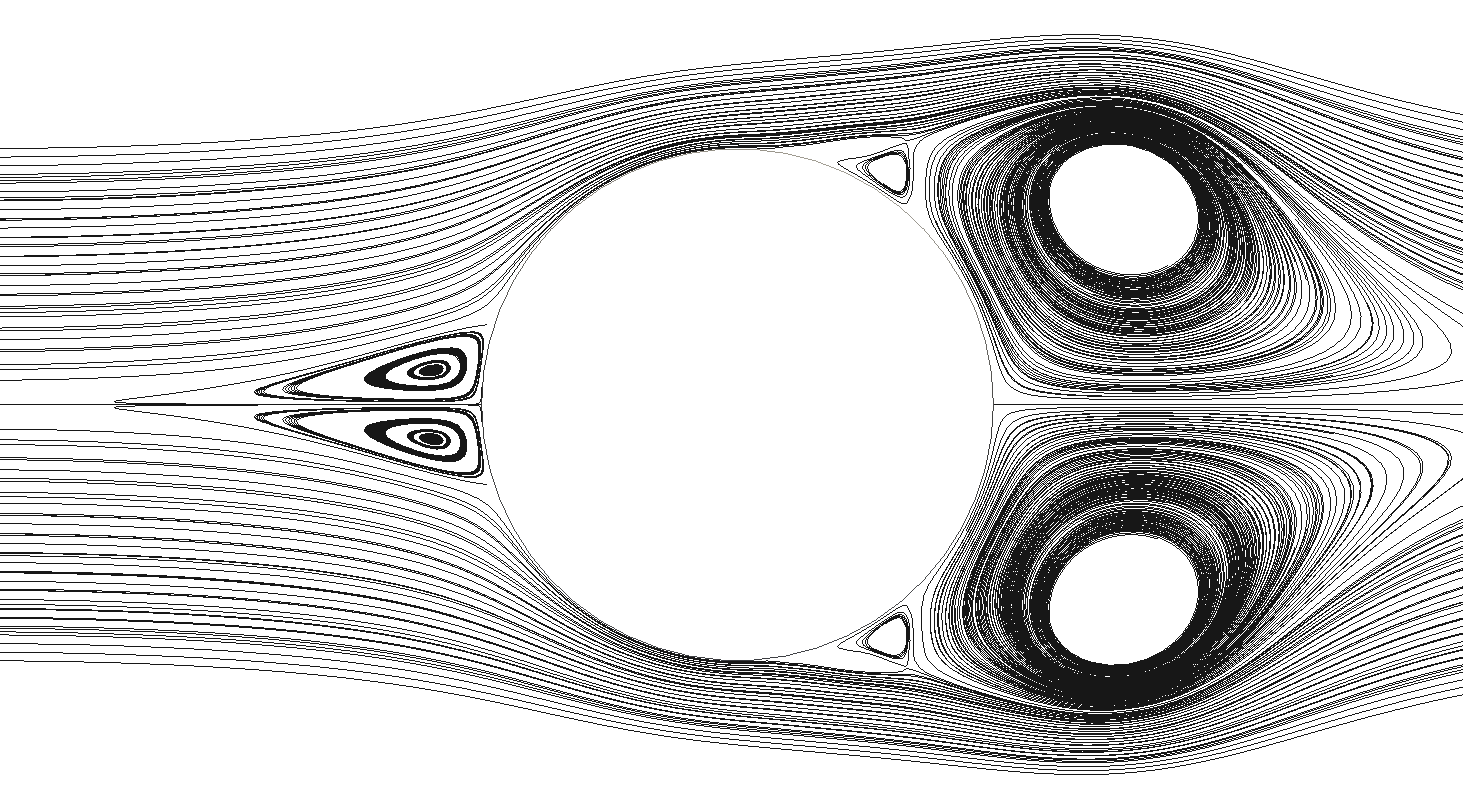}}
\subfigure[$t\approx110s$]{\includegraphics[scale=0.12]{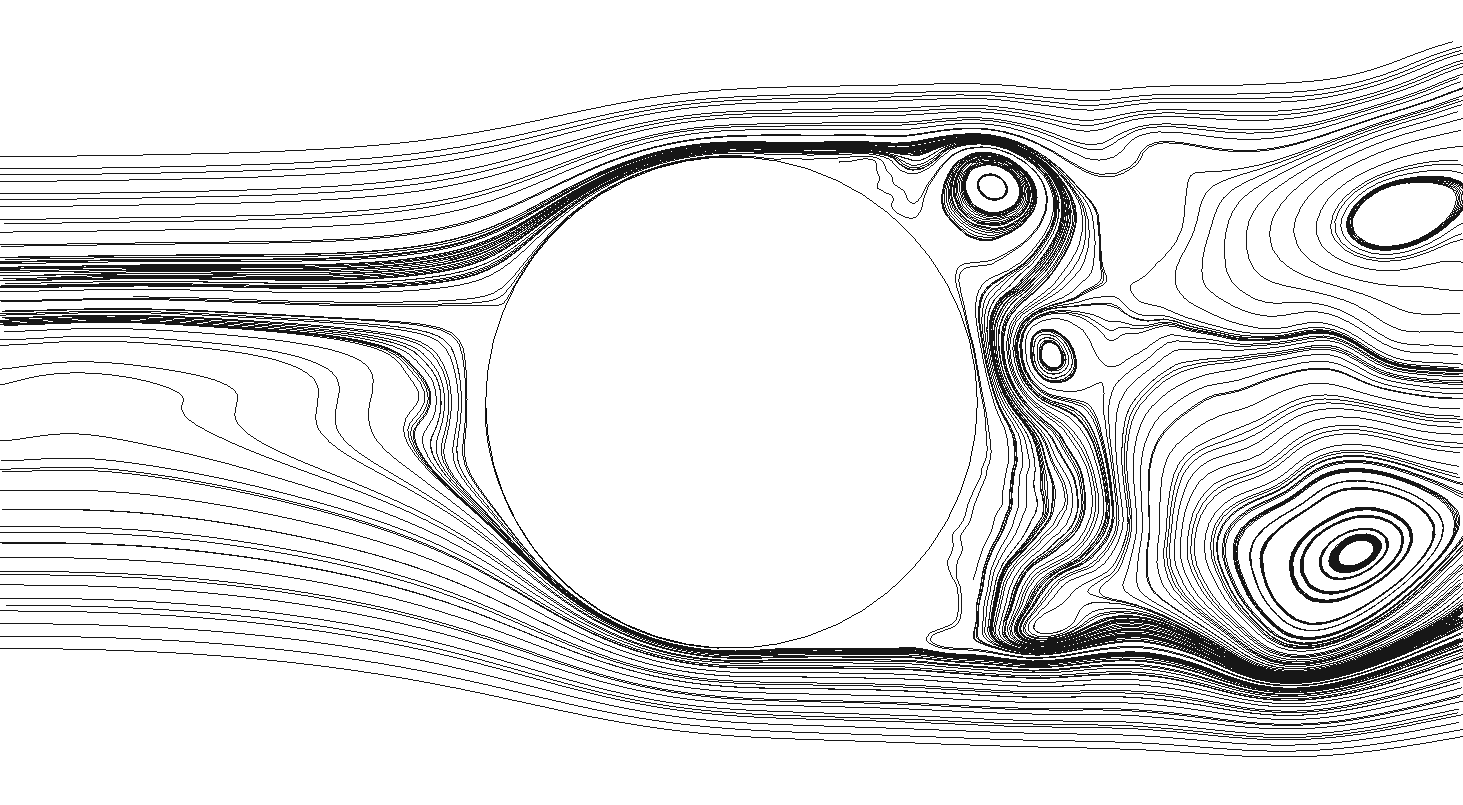}}
\subfigure[$t\approx120s$]{\includegraphics[scale=0.12]{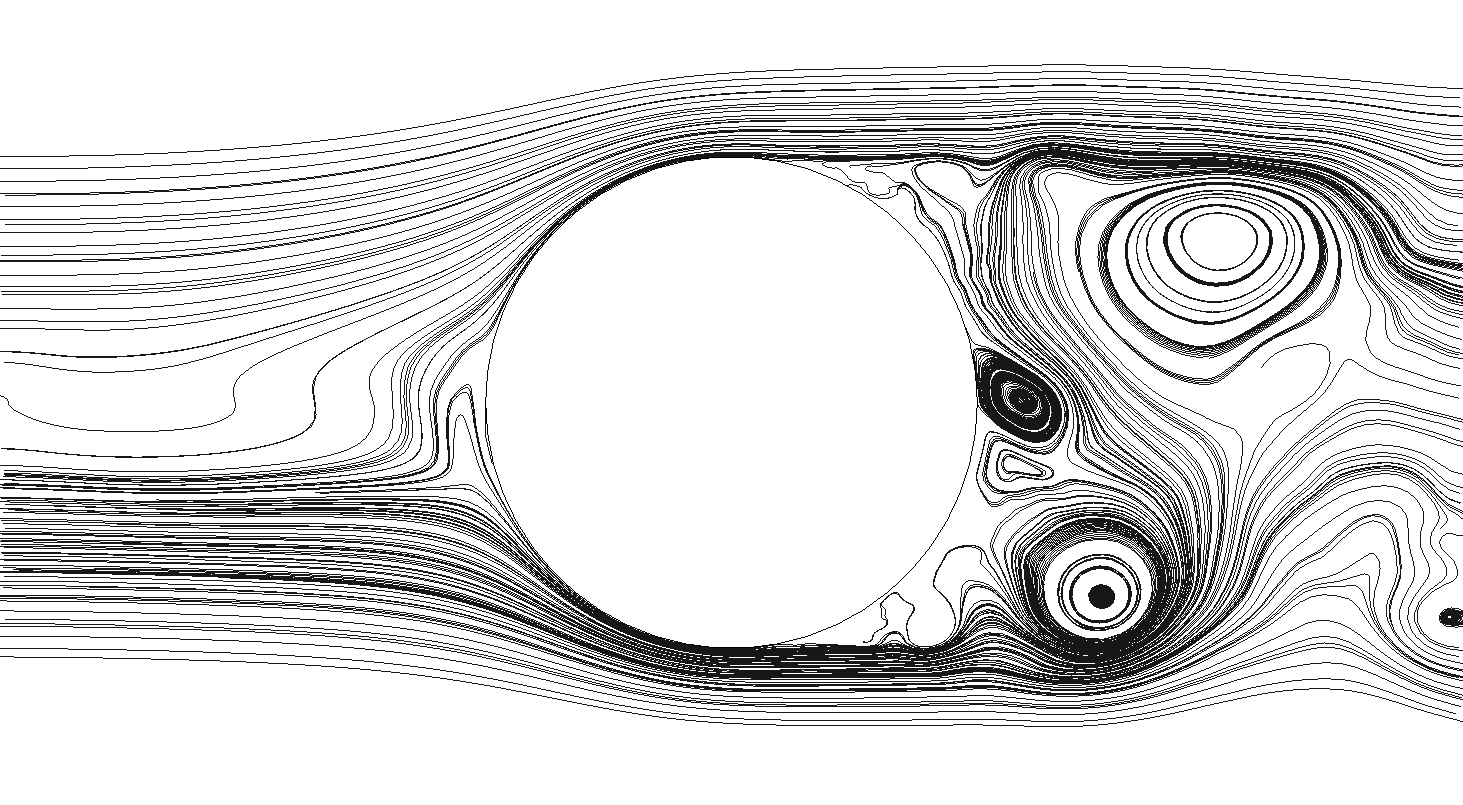}}
\caption{Fraenkel flow: solution in the whole domain.}\label{fig:FraenkelSymTime}
\end{center}
\end{figure}
%
%
\subsubsection{Jet impact problem}\label{subsubsec:Jet}
%
Here we consider the impact of an inviscid jet in order to illustrate the influence of the proposed boundary stabilization and the importance of the scaling of the absolute value matrix. The geometry is a rotated 'T' of the precise form $\Omega=]0; 2[\times]0; 2[\times[0; 1] \times [1; 2] \cup [0; 1] \times [1; 2]$. Starting from an inert state, the jet enters with a flat profile $(3-2t)t^2$ from the left, hits a vertical wall and leaves the domain through an upper and a lower channel, as shown in Figure~\ref{fig:TshirtFlowa}, which is taken at $t = 3s$; due  to symmetry, we only show the solutions on half of the domain.
\paragraph{Comparison with the alternative discretization}

We impose the inflow condition on $\GammaIn = {0}\times]1; 2[$, the outflow condition on $\GammaOut =]1; 2[\times{1}\cup]1; 2[\times{2}$ and the wall condition on the remaining parts of the boundary $\GammaWall = \partial\Omega \setminus (\GammaIn \cup \GammaOut)$. The flow develops two symmetric large vertices (the upper one is shown in Figure~\ref{fig:TshirtFlowa}), which break up into a series of smaller ones and stays non-stationary. Figures~\ref{fig:TshirtFlowb} and \ref{fig:TshirtFlowc} show the velocity magnitude computed with strong and weak enforcement of the normal velocity, respectively. For the strong implementation we can detect oscillations on the upper wall of the boundary, which propagate into the interior of the domain. For the weak implementation, oscillations behind the re-entrant corner are also visible, but clearly more localized. Notice that both methods employ the same SUPG stabilization, and that no shock-capturing terms are used. 
\begin{figure}[h!]
\begin{center}
\subfigure[Flow pattern.]{\includegraphics[width=0.33\textwidth]{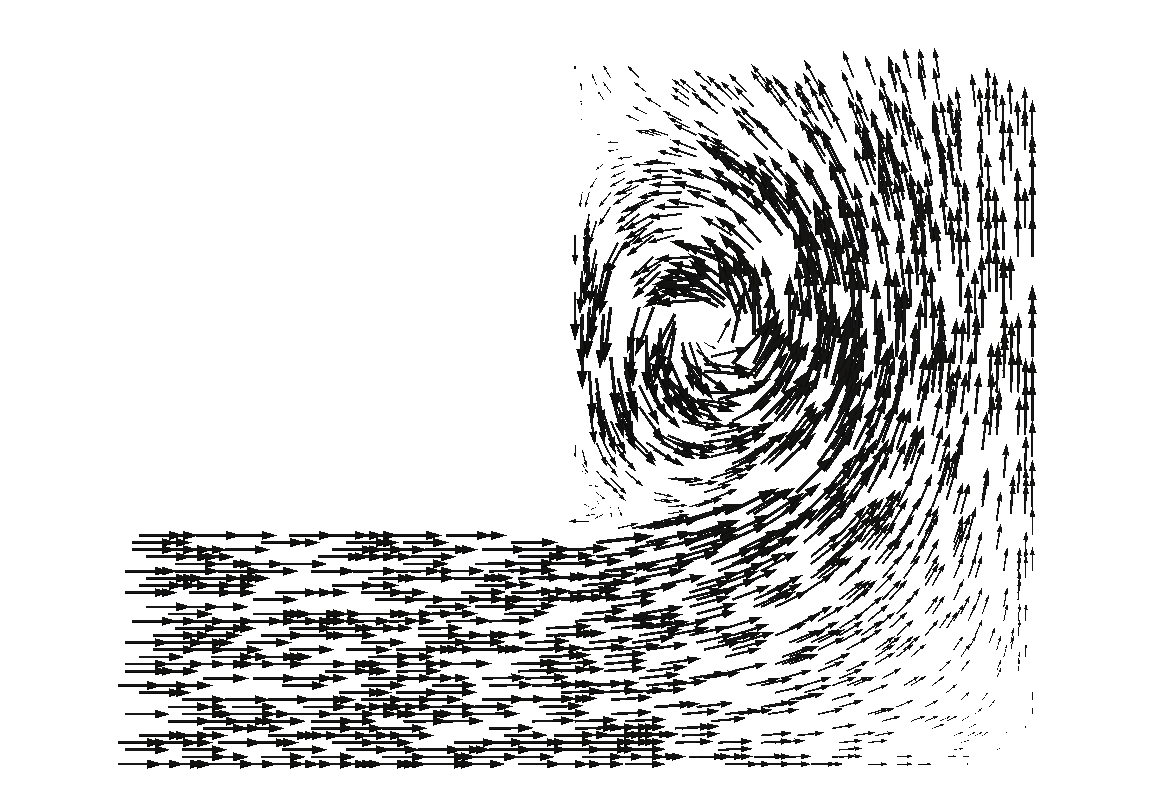}\label{fig:TshirtFlowa}}
\subfigure[Magnitude of $v_1$ (strong).]{\includegraphics[width=0.31\textwidth]{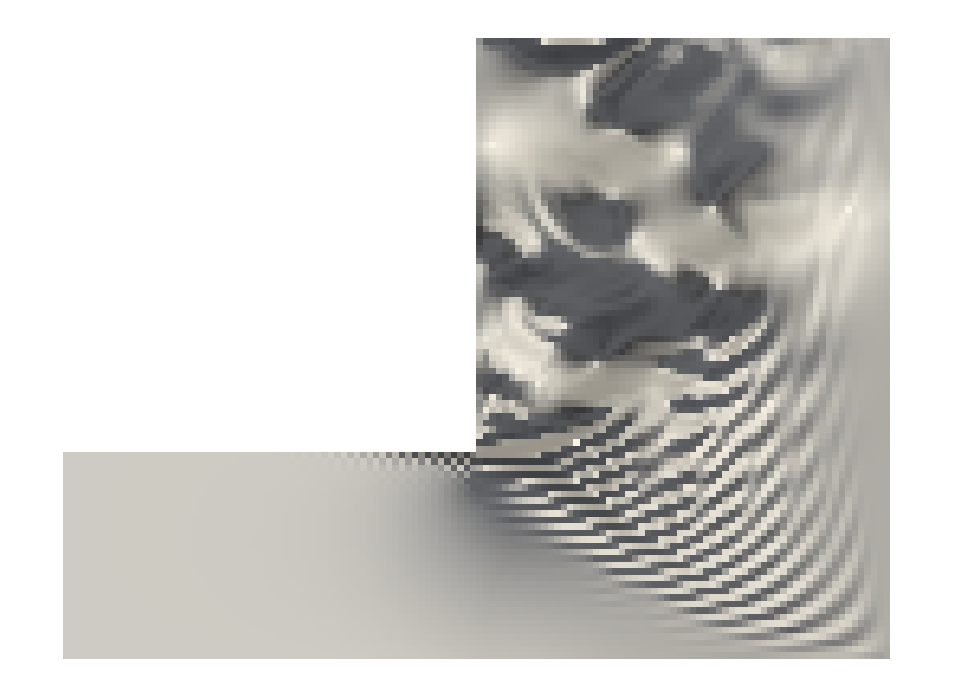}\label{fig:TshirtFlowb}}
\subfigure[Magnitude of $v_1$ (weak).]{\includegraphics[width=0.31\textwidth]{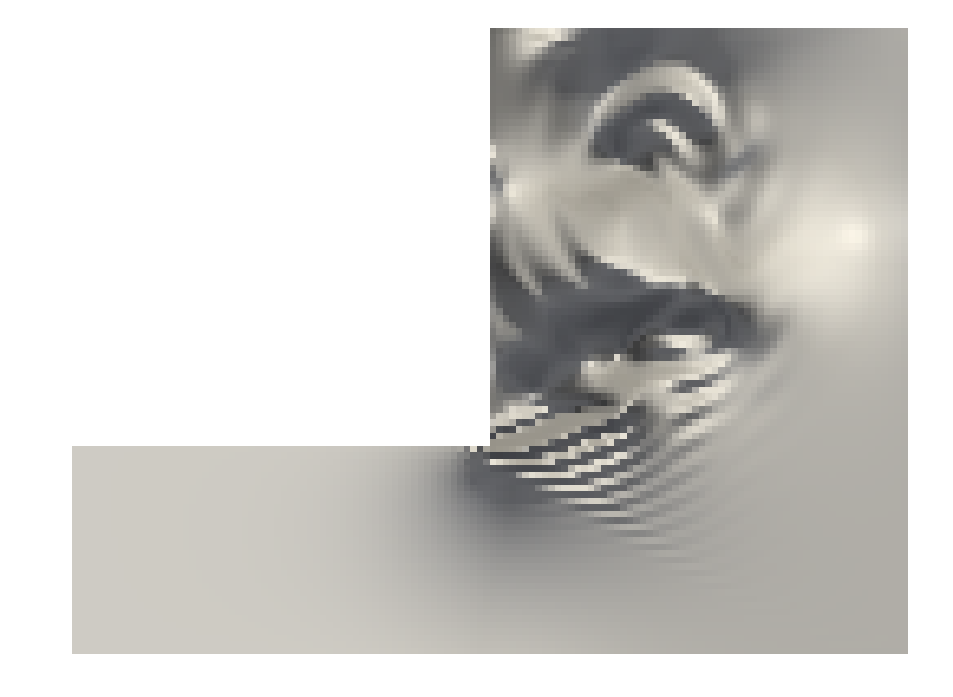}\label{fig:TshirtFlowc}}
\caption{Jet impact problem: strong vs. weak boundary conditions (half domain).}\label{fig:TshirtFlow}
\end{center}
\end{figure}
\paragraph{Influence of the scaling}
We use the same test case to investigate the behavior of the discrete solutions under scaling. Here, we replace the inflow and outflow boundary conditions by the characteristic one with the same $\vD$ as before on $\GammaIn$, $\vD=0$ on $\GammaOut$ and a pressure difference equal to $1$, i.e. $\pD = 1$ on $\GammaIn$ and $\pD = 0$ on $\GammaOut$. To this end, we compute the problem in the scaled domain $\widetilde\Omega= s\Omega$ with $s = 10$ and prescribe $\widetilde{v}^D = s\vD$ and $\widetilde{p}^D = s^2\pD$, such that the continuous solution verifies $\widetilde{v}= s{v}$ and $\widetilde{p} = s^2{p}$. The scaling of the domain implies $\widetilde{d}_h=s d_h$.
\begin{figure}[h!]
\begin{center}
\includegraphics[width=0.9\textwidth]{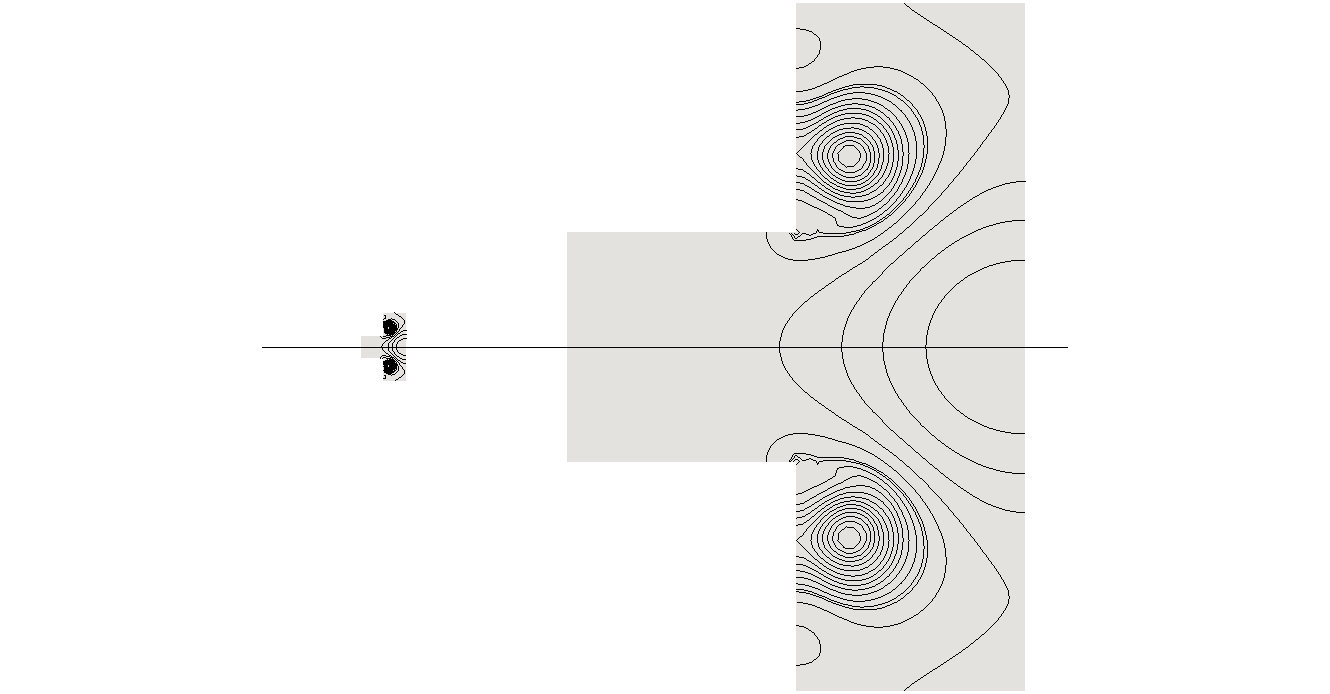}\label{fig:TshirtDomains}
\caption{Jet impact problem: pressure isolines on $\Omega$ and $\widetilde\Omega$.}
\end{center}
\end{figure}
For the sake of clarity, we show in Figure~\ref{fig:TshirtDomains} the two computational domains and the pressure isolines obtained with the proposed method.
\begin{figure}[h!]
\begin{center}
\subfigure[Proposed method.]{\includegraphics[width=0.43\textwidth]{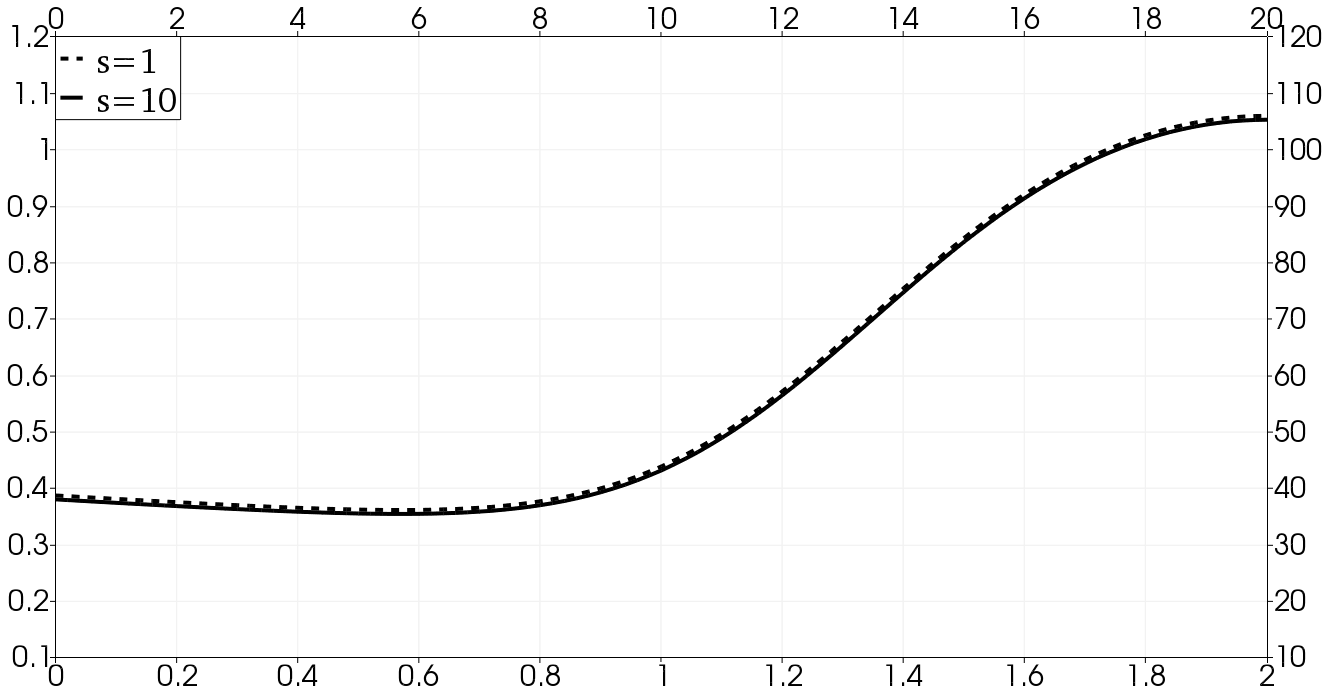}\label{fig:TshirtScaledA}}
\subfigure[Alternative method.]{\includegraphics[width=0.46\textwidth]{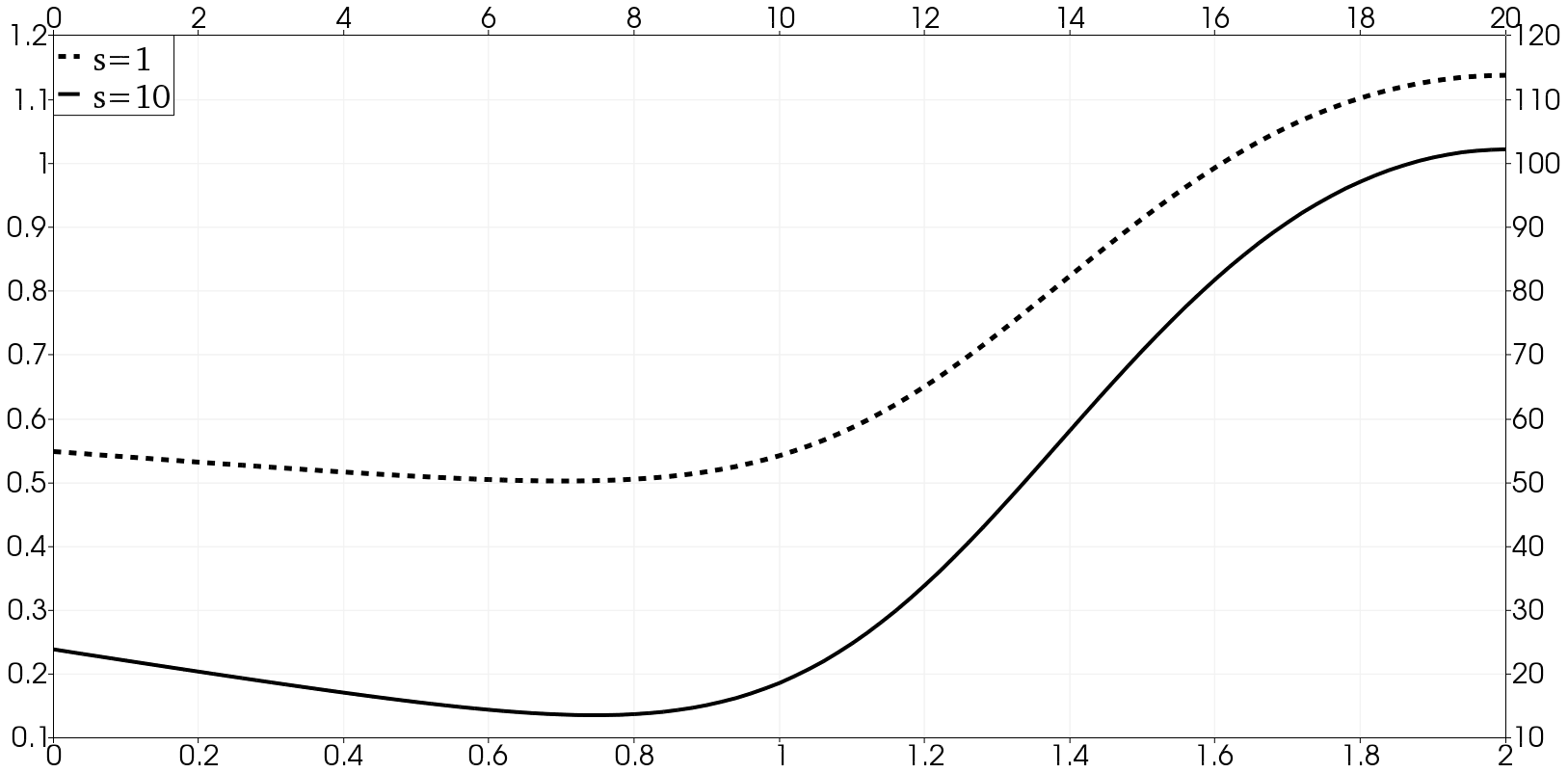}\label{fig:TshirtScaledB}}
\caption{Jet impact problem: pressure along the symmetry axis.}
\end{center}
\end{figure}
In Figures~\ref{fig:TshirtScaledA} and \ref{fig:TshirtScaledB} we represent the discrete pressures $p_h$ and $\widetilde{p}_h$ along the symmetry axis $y=1$ and $\widetilde y=10$, respectively, at time $t=3$. Two different scales are used for the space coordinates ($1$ and $10$), as well as for the pressure ($1$ and $10^2$), as indicated.
The left Figure~\ref{fig:TshirtScaledA} shows the results obtained with the proposed method, where $\theta$ is defined in (\ref{eq:ChoiceThetaFinal}). As expected, the invariance to scaling is respected up to machine accuracy.
For comparison, in the right Figure~\ref{fig:TshirtScaledB} we show the results obtained with the alternative discretization (\ref{eq:NavierStokesDiscreteAlternative}), which corresponds to $\theta=1$ in the definition of the characteristic boundary conditions. Clearly, the pressure is no longer invariant under scaling.
%
\section{Appendix}\label{sec:appendix}
%
We give here the proofs of Lemmas \ref{lemma:EulerEigen} and \ref{lemma:2} concerning the spectral decomposition of the Euler equations. 


\subsection{Proof of Lemma \ref{lemma:EulerEigen}}

Let $\begin{bmatrix}
x\\
y
\end{bmatrix}$ be an eigenvector to $\lambda$. Then
\begin{eqnarray*}
\theta x_n = \lambda y,\qquad \theta y n + (\rho v_n-\lambda)x=0.
\end{eqnarray*}
Multiplying the second equation by $\theta n$ and inserting the first equation gives
\begin{equation*}
y(\theta^2 + (\rho v_n-\lambda)\lambda)=0.
\end{equation*}
The case $y=0$ immediately gives $x=n^{\perp}$ in 2D, respectively $x=t_1$, $x=t_2$ in 3D and $\lambda =\rho v_n$. Otherwise we find the two roots $\lambdapm$ of the quadratic equation 
\begin{equation*}
\lambda^2 - \rho v_n\lambda - \theta^2=0.
\end{equation*}

\subsection{Proof of Lemma \ref{lemma:2}}

We note that
\begin{equation*}
\lambdap\lambdam = -\theta^2,\quad \lambdap^2+\lambdam^2 = (\lambdap+\lambdam)^2 - 2\lambdap\lambdam = 2\theta^2 +\rho^2 v_n^2
\end{equation*}
and that $|\Theta A_n(u)\Theta |=R |\Lambda | R^{T}$. For $d=2$, we have
\begin{equation*}
|\Lambda | =\begin{bmatrix}\rho |v_n | & 0 & 0\\ 0 & \lambdap & 0\\0& 0& -\lambdam \end{bmatrix}, \quad
 R^{T}\Theta^{-1}\psi=
 \begin{bmatrix}
 \phi_{n}^{\perp}\cdot {n}^{\perp}\\ \frac{1}{\sqrt{\theta^2+\lambdap^2}}(\lambdap \phi_n+ \chi)\\\frac{1}{\sqrt{\theta^2+\lambdam^2}}(\lambdam \phi_n+\chi) \end{bmatrix}
\end{equation*}
so we get
\begin{eqnarray*}
|A_n(u)|_{\Theta}\psi\cdot \psi' &=&|\Lambda | R^{T}\Theta^{-1}\psi\cdot R^{T}\Theta^{-1}\psi' \\
&=&\rho |v_n| \phi_{n}^{\perp}\cdot \phi_{n}^{'\perp} + \frac{\lambdap}{\theta^2+\lambdap^2}(\chi+\lambdap \phi_n)(\chi'+\lambdap \phi'_n)\\
&&- \frac{\lambdam}{\theta^2+\lambdam^2}(\chi+\lambdam \phi_n)(\chi'+\lambdam \phi'_n).
\end{eqnarray*}
It goes the same way for $d=3$ with $\phi_{n}^{\perp}\cdot \phi_{n}^{'\perp}=\phi\cdot \phi'-\phi_n\phi'_n$.
Using that 
\begin{equation}\label{eq:rel_lambda}
\frac{\lambdap}{\theta^2+\lambdap^2}=-\frac{\lambdam}{\theta^2+\lambdam^2} =\frac{1}{\lambdap-\lambdam}=\frac{1}{\sqrt{4\theta^2+\rho^2 v_n^2}}
\end{equation}
we deduce:
\begin{eqnarray*}
\begin{split}
&\frac{\lambdap}{\theta^2+\lambdap^2}(\chi+\lambdap \phi_n)(\chi'+\lambdap \phi'_n)- \frac{\lambdam}{\theta^2+\lambdam^2}(\chi+\lambdam \phi_n)(\chi'+\lambdam \phi'_n)\\
&=\frac{1}{\lambdap-\lambdam}\bigg((\chi+\lambdap \phi_n)(\chi'+\lambdap \phi'_n)+ (\chi+\lambdam \phi_n)(\chi'+\lambdam \phi'_n)\bigg)\\
&= \frac{1}{\lambdap-\lambdam}\bigg(2\chi\chi'+(\lambdap+\lambdam) (\phi_n \chi'+\phi'_n \chi)+(\lambdap^2+\lambdam^2) \phi_n\phi'_n\bigg)\\
&= \frac{1}{\lambdap-\lambdam}\bigg(2\chi\chi'+\rho v_n (\phi_n \chi'+\phi'_n \chi)+(2\theta^2 +\rho^2 v_n^2) \phi_n\phi'_n\bigg)\\
&= \frac{1}{\sqrt{4\theta^2+\rho^2 v_n^2}}\bigg(  (\chi+\rho v_n \phi_n)(\chi'+\rho v_n \phi'_n)  +\chi\chi'+2\theta^2 \phi_n\phi'_n\bigg).\\
\end{split}
\end{eqnarray*}

We next compute $\ThetaM{A_n(u)}\psi\cdot\psi' = \Lambda^{-} R^{T}\Theta^{-1}\psi\cdot R^{T}\Theta^{-1}\psi'$ where
\begin{equation*}
\Lambda^- =\begin{bmatrix}\rho v_n^- & 0 & 0\\ 0 & 0 & 0\\0& 0& \lambdam \end{bmatrix}.
\end{equation*}
We obtain:
\begin{equation*}
\ThetaM{A_n(u)}\psi\cdot\psi' = \rho v_n^-\phi_{n}^{\perp}\cdot \phi_{n}^{'\perp} + \frac{\lambdam}{\theta^2+\lambdam^2} (\lambdam\phi_n + \chi)(\lambdam\phi'_n + \chi').
\end{equation*}

Finally, noting that 
\begin{equation*}
\sqrt{4\theta^2+\rho^2 v_n^2}\simeq \theta+\rho |v_n|,\quad \chi^2+(\chi+\rho v_n \phi_n)^2   \simeq \chi^2+ \rho^2 v_n^2 \phi_n^2
\end{equation*}
 we obtain the equivalence of $ |A_n(u)|_{\Theta}\psi\cdot \psi$ with (\ref{eq:coerc}).

\bibliographystyle{siam}
\bibliography{../../../Bibliotheque/bibliotheque.bib}

\begin{thebibliography}{10}

\bibitem{Barth99}
{\sc T.~J. Barth}, {\em Numerical methods for gas dynamic systems on
  unstructured meshes}, in An introduction to recent developments in theory and
  numerics for conservation laws ({F}reiburg/{L}ittenweiler, 1997), vol.~5 of
  Lect. Notes Comput. Sci. Eng., Springer, Berlin, 1999, pp.~195--285.

\bibitem{BazilevsHughes07}
{\sc Y.~Bazilevs and T.~J.~R. Hughes}, {\em Weak imposition of {D}irichlet
  boundary conditions in fluid mechanics}, Comput. \& Fluids, 36 (2007),
  pp.~12--26.

\bibitem{BazilevsMichlerCalo07}
{\sc Y.~Bazilevs, C.~Michler, V.~M. Calo, and T.~J.~R. Hughes}, {\em Weak
  {D}irichlet boundary conditions for wall-bounded turbulent flows}, Comput.
  Methods Appl. Mech. Engrg., 196 (2007), pp.~4853--4862.

\bibitem{BazilevsMichlerCalo10a}
\leavevmode\vrule height 2pt depth -1.6pt width 23pt, {\em Isogeometric
  variational multiscale modeling of wall-bounded turbulent flows with weakly
  enforced boundary conditions on unstretched meshes}, Comput. Methods Appl.
  Mech. Engrg., 199 (2010), pp.~780--790.

\bibitem{Becker02}
{\sc R.~Becker}, {\em Mesh adaptation for {D}irichlet flow control via
  {N}itsche's method}, Comm. in Num. Meth. Engrg., 18 (2002), pp.~669--680.

\bibitem{BeckerCapatinaLuceTrujillo14d}
{\sc R.~Becker, D.~Capatina, R.~Luce, and D.~Trujillo}, {\em Stabilized finite
  element formulation with domain decomposition for incompressible flows}.
\newblock submitted, 2014.

\bibitem{BeckerHansbo00}
{\sc R.~Becker and P.~Hansbo}, {\em {Discontinuous Galerkin methods for
  convection-diffusion problems with arbitrary P\a'eclet number}}, in Numerical
  Mathematics and Advanced Applications, Proceedings of ENUMATH '99,
  P.~Neittaanm\"aki, T.~Tiihonen, and P.~Tarvainen, eds., 2000, pp.~100--109.

\bibitem{BoyerFabrie07}
{\sc F.~Boyer and P.~Fabrie}, {\em Outflow boundary conditions for the
  incompressible non-homogeneous {N}avier-{S}tokes equations}, Discrete Contin.
  Dyn. Syst. Ser. B, 7 (2007), pp.~219--250 (electronic).

\bibitem{BraackMucha13}
{\sc M.~Braack and P.~B. Mucha}, {\em Directional do-nothing condition for the
  navier-stokes equations}, tech. rep., U Kiel, 2013.

\bibitem{BrooksHughes82}
{\sc A.~Brooks and T.~Hughes}, {\em {Streamline upwind/Petrov-Galerkin
  formulations for convection dominated flows with particular emphasis on the
  incompressible Navier-Stokes equations.}}, Comput. Methods Appl. Mech.
  Engrg., 32 (1982), pp.~199--259.

\bibitem{BruneauFabrie96}
{\sc C.-H. Bruneau and P.~Fabrie}, {\em New efficient boundary conditions for
  incompressible {N}avier-{S}tokes equations: a well-posedness result}, RAIRO
  Mod\'el. Math. Anal. Num\'er., 30 (1996), pp.~815--840.

\bibitem{BurmanFernandezHansbo06}
{\sc E.~Burman, M.~A. Fern{\'a}ndez, and P.~Hansbo}, {\em Continuous interior
  penalty finite element method for {O}seen's equations}, SIAM J. Numer. Anal.,
  44 (2006), pp.~1248--1274 (electronic).

\bibitem{De-LellisSzekelyhidi13}
{\sc C.~De~Lellis and L.~Sz{\'e}kelyhidi, Jr.}, {\em Dissipative continuous
  {E}uler flows}, Invent. Math., 193 (2013), pp.~377--407.

\bibitem{ErnGuermond06a}
{\sc A.~Ern and J.-L. Guermond}, {\em {Discontinuous Galerkin methods for
  Friedrichs' systems. I: General theory.}}, SIAM J. Numer. Anal., 44 (2006),
  pp.~753--778.

\bibitem{EvansHughes13b}
{\sc J.~A. Evans and T.~J.~R. Hughes}, {\em Isogeometric divergence-conforming
  {B}-splines for the steady {N}avier-{S}tokes equations}, Math. Models Methods
  Appl. Sci., 23 (2013), pp.~1421--1478.

\bibitem{EvansHughes13c}
\leavevmode\vrule height 2pt depth -1.6pt width 23pt, {\em Isogeometric
  divergence-conforming {B}-splines for the unsteady {N}avier-{S}tokes
  equations}, J. Comput. Phys., 241 (2013), pp.~141--167.

\bibitem{FeistauerKuvcera07}
{\sc M.~Feistauer and V.~Ku\v{c}era}, {\em {On a robust discontinuous Galerkin
  technique for the solution of compressible flow.}}, J. Comput. Phys., 224
  (2007), pp.~208--221.

\bibitem{Fraenkel61}
{\sc L.~E. Fraenkel}, {\em On corner eddies in plane inviscid shear flow}, J.
  Fluid Mech., 11 (1961), pp.~400--406.

\bibitem{FrancaFrey92}
{\sc L.~P. Franca and S.~L. Frey}, {\em {Stabilized finite element methods. II:
  The incompressible Navier-Stokes equations.}}, Comput. Methods Appl. Mech.
  Engrg., 99 (1992), pp.~209--233.

\bibitem{HeywoodRannacherTurek92}
{\sc J.~Heywood, R.~Rannacher, and S.~Turek}, {\em Artificial boundaries and
  flux and pressure conditions for the incompressible {N}avier-{S}tokes
  equations}, Int. J. Numer. Math. Fluids., 22 (1992), pp.~325--352.

\bibitem{Hou09}
{\sc T.~Y. Hou}, {\em Blow-up or no blow-up? {A} unified computational and
  analytic approach to 3{D} incompressible {E}uler and {N}avier-{S}tokes
  equations}, Acta Numer., 18 (2009), pp.~277--346.

\bibitem{HughesFrancaMallet86}
{\sc T.~Hughes, L.~Franca, and M.~Mallet}, {\em {A new finite element
  formulation for computational fluid dynamics. I: Symmetric forms of the
  compressible Euler and Navier-Stokes equations and the second law of
  thermodynamics.}}, Comput. Methods Appl. Mech. Engrg., 54 (1986),
  pp.~223--234.

\bibitem{Jensen05}
{\sc M.~Jensen}, {\em Discontinuous Galerkin Methods for Friedrichs Systems
  with Irregular Solutions}, PhD thesis, Oxford, 2005.

\bibitem{Kovasznay48}
{\sc L.~I.~G. Kovasznay}, {\em Laminar flow behind two-dimensional grid}, Proc.
  Cambridge Philos. Soc., 44 (1948), pp.~58--62.

\bibitem{Lions96}
{\sc P.-L. Lions}, {\em Mathematical topics in fluid mechanics. {V}ol. 1},
  vol.~3 of Oxford Lecture Series in Mathematics and its Applications, The
  Clarendon Press, Oxford University Press, New York, 1996.
\newblock Incompressible models, Oxford Science Publications.

\bibitem{Nitsche71}
{\sc J.~Nitsche}, {\em {\"U}ber ein {V}ariationsprinzip zur {L}{\"o}sung von
  {D}irichlet-{P}roblemen bei {V}er\-wendung von {T}eil\-r{\"a}umen, die keinen
  {R}andbedingungen unterworfen sind}, Abh. Math. Univ. Hamburg, 36 (1971),
  pp.~9--15.

\bibitem{SchaferTurek96}
{\sc M.~Sch{\"a}fer and S.~Turek}, {\em Benchmark computations of laminar flow
  around a cylinder. ({W}ith support by {F}. {D}urst, {E}. {K}rause and {R}.
  {R}annacher)}, in Flow Simulation with High-Performance Computers {II}. {DFG
  priority research program results 1993-1995}, E.~Hirschel, ed., no.~52 in
  Notes Numer. Fluid Mech., Vieweg, Wiesbaden, 1996, pp.~547--566.

\bibitem{SzekelyhidiWiedemann12}
{\sc L.~Sz{\'e}kelyhidi and E.~Wiedemann}, {\em Young measures generated by
  ideal incompressible fluid flows}, Arch. Ration. Mech. Anal., 206 (2012),
  pp.~333--366.

\end{thebibliography}

\end{document}